\documentclass[12pt,reqno]{amsart}
         
\usepackage[left=1in, right=1in, top=1.1in,bottom=1.1in]{geometry} 
\usepackage[dvipsnames]{xcolor}
\usepackage[mathscr]{eucal}
\usepackage{latexsym,amsfonts,amssymb}
\usepackage{bm,pifont,stmaryrd}
\usepackage{graphicx}
\usepackage{color}
\usepackage{enumitem}
\usepackage{hyperref}
\hypersetup{
     colorlinks   = true,
     citecolor    = blue,
     linkcolor    = blue
} 

\usepackage{pdfsync}
\usepackage[font={scriptsize}]{caption}

\usepackage{tikz}
\usetikzlibrary{trees}

\definecolor{darkblue}{rgb}{0.1, 0.2, 0.75}
\definecolor{darkgreen}{rgb}{0.1, 0.35, 0}
\definecolor{candypink}{rgb}{0.89, 0.44, 0.48}
\definecolor{deepcerise}{rgb}{0.85, 0.2, 0.53}
\definecolor{atomictangerine}{rgb}{1.0, 0.6, 0.4}
\definecolor{fandango}{rgb}{0.71, 0.2, 0.54}

\newtheorem{theorem}{Theorem}[section]
\newtheorem{Def}[theorem]{Definition}

\newtheorem{thm}[theorem]{Theorem}

\newtheorem{cor}[theorem]{Corollary}
\newtheorem{lemma}[theorem]{Lemma}

\theoremstyle{remark}
\newtheorem{remark}[theorem]{Remark}

\numberwithin{equation}{section}

\def\RR{\mathbb{R}}

\def\NN{\mathbb{N}}
\def\mE{\mathbb{E}}

\def\ZZ{\mathbb{Z}}

\newcommand{\DD}{\mathbb{D}}

\def\bfy{{\bf y}}
\def\bfz{{\bf z}}

\def\crr{{\mathcal R}}

\newcommand{\ca}{{\mathcal A}}

\newcommand{\cc}{{\mathcal C}}

\newcommand{\ce}{{\mathcal E}}
\newcommand{\cf}{{\mathcal F}}
\newcommand{\cg}{{\mathcal G}}
\newcommand{\ch}{{\mathcal H}}
\newcommand{\ci}{{\mathcal I}}
\newcommand{\cj}{{\mathcal J}}

\newcommand{\cp}{{\mathcal P}}

\newcommand{\cs}{{\mathcal S}}

\newcommand{\cz}{{\mathcal Z}}

\def\la{{\lambda}}
\def\si{\sigma}

\def\la{{\lambda}}
\def\La{{\Lambda}}

\def\al{{\alpha}}

\def\be{{\beta}}
\def\Ga{{\Gamma}}
\def\ga{{\gamma}}

 \newcommand{\vp}{\varphi}
 
 \newcommand{\ep}{\varepsilon}
 
\newcommand{\lp}{\left(}
\newcommand{\rp}{\right)}
\newcommand{\lc}{\left[}
\newcommand{\rc}{\right]}

\def\ll{\llbracket}
\def\rr{\rrbracket}
\def\wt{\widetilde}  
\def\id{\text{Id}}  
  
\begin{document}
\title[Compensated weighted sums]
{Limit theorems for compensated weighted sums and  application to numerical approximations}
\date{}   

\author[Y. Liu]
{Yanghui Liu} 
\address{Y. Liu: Baruch College, City University of New York, New York}
\email{yanghui.liu@baruch.cuny.edu}

\keywords{ 
  {Compensated weighted sums,}
 {Skorohod-type Riemann sums,}
 {Limit theorems,}
 {Discrete rough paths,}
 {Fractional Brownian motion,}
 {Additive stochastic differential equation,}
 {Euler  method.}
}
    
\begin{abstract} 
In this   paper, we consider  a  ``compensated''  random sum   that arises  from numerical approximation of stochastic integrations and differential equations.  
We   show that the compensated sum exhibits  some surprising    cancellations among its components, a property which allows to transform it    into a Skorohod-type Riemann sum.      
We  then establish     limit theorem for the compensated  sum  based on   study of the Skorohod-type Riemann sum. Our proof employs techniques from Malliavin calculus and rough path.    
We apply  our limit theorem result to  the Euler   approximation method for stochastic integrals and additive stochastic differential equations,   filling a notable gap in this area of  research. We show that the Euler method converges to the solution at the  rate $(1/n)^{H+1/2}$, and that this rate is exact  in the sense that the asymptotic error distribution solves   a linear differential equation.  
\end{abstract}

\maketitle

{
\hypersetup{linkcolor=black}
\setcounter{tocdepth}{2}
}

\section{Introduction}\label{section.intro}

  Consider  a  stochastic process  $(h^{(n)}_{k}, k=0,\dots,n-1)$, $n\in\NN$    and an associated  ``weight'' process $(y_{t}, t\in[0,T])$. 
  Let  $\cp_{n}: 0=t_{0}<\cdots<t_{n}=T$  denote a partition  on $[0,T]$.   
Suppose that $\sum_{k=0}^{n-1}h^{(n)}_{k}$ converges in distribution as $n\to\infty$,   to a  Wiener process  for instance.  What can be said about  the asymptotic behavior of the weighted sum: 
     \begin{eqnarray}
\sum_{k=0}^{n-1} y_{t_{k}}h^{(n)}_{k} ?  
\label{e.ws}
\end{eqnarray}  
This problem   has drawn a lot of attention in the recent literature due to its   essential roles in topics  such as   stochastic integration  (see e.g. \cite{binotto2018weak, burdzy2010change, liu2023convergence, nourdin2010weak}),   numerical   approximation (see e.g. \cite{gradinaru2009milstein, hu2016rate, liu2019first}), and parameter estimation (see e.g. \cite{barndorff2009power, chong2022statistical, Paris, corcuera2006power}). The limit theorem results for \eqref{e.ws}   have been established     in various   settings (see e.g. 
\cite{burdzy2010change, LT20, nourdin2008asymptotic, nourdin2016quantitative, nourdin2010central, nourdin2009asymptotic}).

Inspired by these contributions,    the purpose of this paper  is     to study a         compensated weighted sum in the following form: 
    \begin{eqnarray}
\sum_{k=0}^{n-1} 
\big(y_{t_{k}} h_{k}^{(n),1}+y_{t_{k}}' h_{k}^{(n),2}+\cdots+y_{t_{k}}^{(\ell-1)} h_{k}^{(n),\ell}  \big)
\qquad\text{for}\quad \ell\geq 2,  
\label{e.cws}
\end{eqnarray}
 where  $h^{(n),i}$ are defined by:   
  \begin{eqnarray}
h^{(n),i}_{k} = \int_{t_{k}}^{t_{k+1}}x^{i}_{ t_{k}u} du    
\qquad
\text{with}\quad  x^{i}_{t_{k}u}:= (x_{u}-x_{t_{k}})^{i}/i! \,   
\notag
\end{eqnarray}
   and  $(x_{t}, t\geq0)$ is a fBm with Hurst parameter $H<1/2$. Here $(y,y',\dots,y^{(\ell-1)})$ is a process ``controlled'' by $x$ (see Definition \ref{def.control}).  
    As   an application of the compensated sum \eqref{e.cws}, we will    investigate   Euler's numerical   methods  for  stochastic integrals and  additive stochastic     differential equations driven by fBm. Our aim is to  derive the exact rate and asymptotic error distribution of these numerical  methods, and thus to fill  a  notable gap in this field of research.  

  A main  challenge   concerning the compensated weighted sum    \eqref{e.cws}  arises from  the cancellations present     among   individual weighted sums. 
For example, consider the case of $\ell=2$ in \eqref{e.cws}, namely  the compensated sum:      
\begin{eqnarray}
\sum_{k=0}^{n-1} \big(y_{t_{k}}h^{(n),1}_{k} +  y_{t_{k}}'h^{(n),2}_{k}\big) . 
\label{e.cws2}
\end{eqnarray}
It has been shown that both weighted sums  in \eqref{e.cws2} are of the order       $  (1/n)^{2H} $ (see \cite{liu2023power}).  Surprisingly, as we will see in this paper,      the compensated sum \eqref{e.cws2} converges to a non-zero limit at the rate   $(1/n)^{ 4H \wedge (H+1/2)}$, which  is strictly higher than those of the individual   sums. Here $a\wedge b$ stands for the minimum between $a$ and $b$.  In fact, we will show that the larger $\ell$     in \eqref{e.cws}    the higher    its     rate of convergence. 
  A main effort of this paper will be devoted    to understanding   the cancellations in \eqref{e.cws}.

In the following we state  our main limit theorem result. The reader is referred to   Theorem~\ref{thm.taylor} for a more precise statement. 

\begin{thm}\label{thm.taylor1}
Let $\ell\in\NN$ be such that $\ell H>1/2$. Then
the compensated sum   in \eqref{e.cws} converges stably  to the conditional Gaussian random variable $ c_{H}^{1/2}T^{H+1/2}\int_{0}^{T}y_{u}dW_{u}$ as the mesh size of $\cp_{n}$ tends to zero, where $W$ is a Brownian motion independent of $x$ and $c_{H}>0$ is a constant depending on $H$ only.   
\end{thm}
 
 The proof of Theorem \ref{thm.taylor1} is a combination of rough path techniques and Malliavin calculus. More specifically,   we will    show that   the compensated sum can be reduced  to some  monomial-weighted   sums based on the approach developed in \cite{LT20}. These simpler compensated  sums allow us to     investigate   the cancellations  by applying the Malliavin calculus tools. Eventually, the asymptotic distribution of \eqref{e.cws} is obtained  thanks to  the smart path interpolation approach  developed in \cite{nourdin2016quantitative, nourdin2012normal}.


As a main application of  Theorem \ref{thm.taylor1} we consider   the Euler method for additive stochastic differential equations (SDE for short).  
 Let $y$ be  solution of  the   following differential equation: 
 \begin{eqnarray}
 y_{t} &=&y_{0}+ \int_{0}^{t} b(y_{u})du +  \int_{0}^{t}  \si(u)  dx_{u}    \qquad
t\in[0,T],   
\label{e.sde2}
\end{eqnarray}
where   $b$   and $\si $ are deterministic continuous functions. It is well-known that equation \eqref{e.sde2} has a unique solution whenever $b$ is  continuously differentiable and  one-sided Lipschitz; see e.g. \cite{zhou2023backward}.     
Consider   the (first-order)  Euler method for   equation \eqref{e.sde2}: 
    \begin{eqnarray}
y_{t}^{(n)} &=&y^{(n)}_{t_{k}}+ b(y^{(n)}_{t_{k}}) (  t - t_{k}) +\si (t_{k})(x_{t}-x_{t_{k}})
\qquad
t\in[t_{k},t_{k+1}]. 
\label{e.euler2} 
\end{eqnarray}
Recall   that under proper regularity conditions for $b$ and $\si$ the  Euler method \eqref{e.euler2} converges to the solution $y$ at the rate $1/n$ in the case $H\geq 1/2$; see e.g.  \cite{neuenkirch2006optimal}.  Furthermore, the   asymptotic error distribution of the backward Euler method has been established in  \cite{zhou2023backward} under the one-sided Lipschitz condition of $b$ (see also \cite{hong2020optimal, riedel2020semi}), and the higher-order Euler methods for any $H\in(0,1)$ have been considered in \cite{hu2015taylor}.  
The reader is referred to    \cite{deya2012milstein, friz2014convergence,  hu2016rate, hu2021crank, leon2023euler, liu2023euler, liu2019first,  neuenkirch2008optimal} and reference therein   for other contributions in numerical approximation of SDEs.  
On the other hand, the analysis  of   Euler method \eqref{e.euler2} for the case $H<1/2$ is still absent in the literature. 
 
 
In this paper  we will see that the convergence of Euler method \eqref{e.euler2} for  $H<1/2$ is  much more involved, mainly due to the presence of the   complex  cancellations    that was mentioned previously.  
More specifically, we will show that the error process $y^{(n)}-y$ is dominated by    a compensated sum of the form \eqref{e.cws} as $n$ tends to infinity.   

In the following we state  our asymptotic error distribution result for the Euler method  (a more precise statement can be found in Theorem \ref{thm.euler}). 


 \begin{thm}
The    process $n^{H+1/2}(y^{(n)}-y) $ converges in distribution to  the solution of the linear differential equation:
\begin{eqnarray}
U_{t} = \int_{0}^{t}\partial b(y_{u})  U_{u}  du +c_{H}^{1/2}T^{H+1/2}\int_{0}^{t} \big[\partial b (y_{u}) \si (u)-\si'(u)\big]dW_{u}
\notag
\end{eqnarray}
$t\in[0,T]$ as $n $ tends to infinity. 
\end{thm}

 Our second application  is focused on     numerical approximation of  the regular integral $\int_{0}^{t} y_{u}du$,   stated as follows.  The reader is referred to  Theorem \ref{thm.integral} for its precise statement. 
  \begin{thm}
  Let the assumptions be as in Theorem \ref{thm.taylor1}. Then we have the convergence in distribution: 
\begin{eqnarray}
n^{H+1/2}\Big(\int_{0}^{T} y_{u}d {u}-\sum_{k=0}^{n-1}y_{t_{k}}\cdot (t_{k+1}-t_{k}) \Big) \to c_{H}^{1/2}T^{H+1/2}\int_{0}^{T}y_{u}'dW_{u}  
\notag
\end{eqnarray}
as $n $ tends to infinity. 
\end{thm}

The paper is organized as follows.  In Section \ref{section.prelim} we recall some basic results of rough path theory and Malliavin calculus.   In Section \ref{section.lim}, we prove the limit theorem  for a Skorhod-type Riemann sum, and then apply it to some  monomial-weighted compensated sums. In the second part of the section, we establish   limit theorem for the  general  compensated weighted sums. As an  application we will consider the  numerical approximation for stochastic integrals. In Section \ref{section.euler}, we focus on the Euler approximation method for additive stochastic differential equations driven by fBm. 

\subsection{Notation}\label{section.notation} 
Let $\cp:0=t_{0}<t_{1}<\cdots<t_{n}=T$ be a partition on $[0,T]$. 
We denote $\eta(u)=t_{k}$ for $u\in[t_{k},t_{k+1})$ and   $\la(v)=t_{k+1}$ for $v\in(t_{k},t_{k+1}]$.    For $s,t\in [0,T]$ such that $s<t$,  we define   the discrete interval $\ll s, t\rr$  to be the collection of $t_{k}$'s in $[s,t]$ and the endpoints $s$ and $t$, namely:   $\ll s, t\rr = \{t_{k}: s\leq t_{k} \leq t\}\cup \{s,t\}$.  For $N\in\NN=\{1,2,\dots\}$ we denote   the discrete simplex  $\cs_{N}(\ll s,t\rr)=
\{ (u_{1},\dots, u_{N}) \in \ll s,t\rr^{N} :\, u_{1}< \cdots< u_{N}  \}$. Similarly, we denote the continuous simplex: $\cs_{N}([s,t])= \{ (u_{1}  ,\dots, u_{N} ) \in [ s,t]^{N} :\, u_{1} < \cdots< u_{N}  \}$.  

Throughout the paper we work on a  probability space $(\Omega, \mathscr{F}, P)$. If $X$ is a random variable, we denote by $| X |_{L_{p}}  $ the $L_{p} $-norm of $X$.
The letter $K$  stands  for  a constant independent of    any important parameters which can change from line to line. We write $A\lesssim B$ if there is a constant $K>0$ such that $A\leq KB$. We denote $[a]  $   the integer part of   $a$,   $\id$  the identity matrix, and $\mathbf{i}=\sqrt{-1}$. 
We define the function $\rho(x)=|x|\vee1$ for $x\in\RR$. 

Let $f: y=(y^{1},\dots,y^{m})\to f(y)$ be a function in $  C^{1}(\RR^{m})  $. We denote by $\partial_{j}f $ its     partial derivative   in the $j$th direction: $\partial_{j}f = \frac{\partial f}{\partial y^{j}} $, $j=1,\dots,m$. Similarly, we denote the higher-order derivatives $\partial^{L}_{j_{1}\cdots j_{L}}f=\partial_{j_{L}}\cdots \partial_{j_{1}}f$, $j_{1},\dots,j_{L}=1,\dots,m $.  
 For $b=(b^{1},\dots,b^{m})  \in C^{1}(\RR^{m}, \RR^{m})  $  we denote     the matrix 
\begin{eqnarray}\label{e.pb}
\partial b =\Big( \frac{\partial b^{i}}{\partial y^{j}}:\,\, i,j=1,\dots, m \Big)=\lp  {\partial_{j} b^{i} } :\,\, i,j=1,\dots, m \rp   
\end{eqnarray}
and
\begin{eqnarray}
\partial^{L}b = \lp \partial^{L}_{j_{1}\cdots j_{L}} b^{i} :\,\, i,j_{1},\dots,j_{L}=1,\dots,m \rp . 
\label{e.pbl}
\end{eqnarray}

\section{Preliminary results}\label{section.prelim}

In this section we  recall some basic results from the    rough path theory and      Malliavin calculus. In the last part of the section we   also state  some useful lemmas. 


\subsection{Elements of rough paths}\label{section.rp}

This subsection is devoted to introducing the main rough paths notations which will be used in the sequel.   The reader is referred to   \cite{FH20, FV10} for an introduction to the rough path theory.
 
 We first introduce a general notion of controlled rough process which will be used throughout the paper.
 Recall that the continuous simplex $\cs_{k}([0,T])$ is defined in Section \ref{section.notation}. 
For functions $f :  [0,T] \to \RR$ and $g : \cs_{2}([0,T]) \to \RR$ we define the operator $\delta$ such that:
\begin{eqnarray}\label{e.delta}
\delta f_{st} = f_{t}-f_{s}
\quad
\text{ and }
\quad
\delta g_{sut} = g_{st}-g_{su}-g_{ut} 
\end{eqnarray} 
for all $(s,t)\in\cs_{2}([0,T])$ and $(s, u,t)\in\cs_{3}([0,T])$.

   \begin{Def}\label{def.control}
   Let $x$ and $y,y',y'',\dots, y^{(\ell-1)}$ be real-valued continuous processes on $[ 0,T] $. Denote:  $x^{i}_{st} = (\delta x_{st})^{i}/i!$, $(s,t)\in\cs_{2}([0,T])$, $i=0, \dots,\ell-1$. For convenience, we   also write $y^{(0)} = y$, $y^{(1)} = y'$, $y^{(2)}= y''$,\dots, and $\bfy=(y^{(0)},\dots, y^{(\ell-1)})$. 
            Let $r^{(i)}$, $i=0,\dots,\ell-1$ be the remainder processes such that:  
   \begin{eqnarray}\label{e.r}
   r^{(k)}_{st} &=& \delta y^{(k)}_{st} - \sum_{i=1}^{\ell-k-1} y^{(k+i)}_{s}x^{i}_{st}\,, 
  \qquad    k=0,1,\dots, \ell-1,
\end{eqnarray}
 for $(s,t) \in \cs_{2}([0,T]) $, where we use the convention that $\sum_{i=1}^{0}=0$.  
We call $\bfy$ \emph{a  rough path controlled by $(x, \ell , \al )$ almost surely} for some   $\al\in(0,1)$  
 if for any $\ep>0$
   there is    
a finite random variable $G_{\bfy}\equiv G_{\bfy,\ep}$     such that   
 $|r^{(k)}_{st}|\leq G_{\bfy} (t-s)^{(\ell-k)\al-\ep}  $  for all $(s,t)\in\cs_{2}([0,T])$ and $k=0, \dots, \ell-1$ almost surely. 
 We call $\bfy $      \emph{a     rough path  controlled by $(x,\ell, \al )$ in $L_{p}$} for some  $p\geq 1$ and $\al\in(0,1)$ if there is a constant $K>0$  such that $|y^{(k)}_{0}|_{L_{p}}\leq K$ and $|r^{(k)}_{st}|_{L_{p}}\leq K (t-s)^{(\ell-k)\al}$ for all $(s,t)\in\cs_{2}([0,T])$  and $k=0,\dots, \ell-1$.  
\end{Def}
  
  \begin{remark}
It follows from the Garsia-Rodemich-Rumsey lemma (see \cite{garsia1970real})   that if a continuous process is         controlled by $(x,\ell, H-\ep )$ in $L_{p}$  for any  $p\geq 1$ and $\ep>0$ then it is also  controlled by $(x,\ell, H )$ almost surely. 
\end{remark}

 In the following we recall an   algebraic  result for controlled rough paths; see e.g. \cite{LT20}.
  \begin{lemma}\label{lem.dr}
Let $x$, $\bfy$ and $r^{(i)}$, $i=0,\dots, \ell-1$ be continuous processes satisfying the algebraic relation \eqref{e.r}. Then    we have the relation: 
\begin{eqnarray}
\delta r^{(i-1)}_{sut} &=& \sum_{j=i+1}^{\ell} r^{(j-1)}_{su} 
x^{j-i}_{ut}\,,
\qquad
 (s,u,t)\in\cs_{3}([0,T]),\quad   i=1,\dots, \ell .
\label{e.dr}
\end{eqnarray}
\end{lemma}

Let us introduce some   ``discrete'' integrals defined as Riemann-Stieltjes  type sums. 
 Namely, let $f$ and $g$ be functions     on   $\cs_{2}([0,T])$ and let $\cp=\{ 0=t_{0}<\cdots<t_{n}=T \}$  be  a generic partition of $[0,T]$.  
  We define  the    discrete  integral of $f$ with respect to $g$ as:
  \begin{eqnarray}\label{e.jfg}
\cj_{s}^{t} (f, g) 
&:=&
\sum_{s\leq t_{k}  <t }   
  f_{\la(s) t_{k}} \otimes g_{t_{k}t_{k+1}} , \quad\quad  (s,t) \in \cs_{2} ([0,T]) ,
\end{eqnarray}
where recall that $\la(s)=t_{k+1}$ for $s\in(t_{k},t_{k+1}]$ and we  use the convention that $\sum_{s\leq t_{k}<t}=0$ when  $\{t_{k}: s\leq t_{k}<t\} = \emptyset$. 
 Similarly, if  $f  $ is a function on $[0,T]$, then we define:
\begin{eqnarray}\label{e.jfg2}
\cj_{s}^{t} (f, g) 
&:=&
\sum_{s\leq t_{k}  <t } 
   f_{ t_{k}} \otimes g_{t_{k}t_{k+1}}  , \quad\quad  (s,t) \in \cs_{2}([0,T]).
\end{eqnarray}

%

\subsection{Elements of Malliavin calculus}

We briefly recall some basic facts of Malliavin calculus. We refer the reader to \cite{nourdin2012normal, N06} for further details.  
Let $\ch$ be a separable Hilbert space with inner product $\langle \cdot, \cdot\rangle_{\ch}$ and norm $|\cdot|_{\ch}=\langle \cdot, \cdot\rangle_{\ch}^{1/2}$. For $q\in\NN$ we denote by $\ch^{\otimes q}$ and $\ch^{\odot q}$, respectively, the $q$th tensor product and the $q$th symmetric tensor product of $\ch$.  Let $X = \{X(h): h\in\ch\}$ be an isonormal Gaussian process over $\ch$, namely, $X$ is a centered Gaussian family defined on some probability space $(\Omega, \cf , P)$ such that $\mE[X(g)X(h)]=\langle g,h \rangle_{\ch}$ for every $g, h\in\ch$.  Let $\mathit{S}$   denote the set of   random variables of the form
$
F=f(W(h_{1}), \dots, W(h_{n}))$,  
where  $n\geq 1$,  $h_{1}, \dots, h_{n}\in\ch$,  and $f:\RR^{m}\to \RR$ is a $C^{\infty}$-function such that $f$ and its partial derivatives have at most polynomial growth. The $l$th Malliavin derivative of $F\in \mathit{S}$ is   a random variable with values in $\ch^{\odot p}$ defined by:
\begin{eqnarray*}
D^{l} F &=& \sum_{i_{1}, \dots, i_{l}=1}^{m} \frac{\partial^{l}f}{\partial x_{i_{1}}\cdots \partial x_{i_{l}} } (X(h_{1}), \dots, X(h_{m})) h_{i_{1}}\otimes \cdots h_{i_{l}}. 
\end{eqnarray*}
For any $p\geq 1$ and any integer $k\geq 1$, we define the Sobolev space $\DD^{k,p}$ as the closure of $\mathit{S}$ with respect to the norm:
\begin{eqnarray*}
\left\| F\right\| _{k,p}^{p}=\mathbb{E}\left[ \left| F\right| ^{p}\right] +
\sum_{l=1}^{k} \mathbb{E} \big[  
 \| D^{l} F \| ^p _{\mathcal{H}^{\otimes l}}
 \big] .
\end{eqnarray*}
We denote by $\delta^{\diamond,k} $ the adjoint of the derivative operator $D^{k}$. We say $u\in\text{Dom}\delta^{\diamond,k}$ if there is a $\delta^{\diamond,k} (u)\in L^{2}(\Omega)$ such that for any $F\in \DD^{k,2}$ the following duality relationship holds:
\begin{eqnarray}
\mE( \langle u, D^{k}F \rangle_{\ch^{\otimes k}} ) = \mE[\delta^{\diamond,k}(u)F]. 
\label{e.ibp}
\end{eqnarray}
When $k=1$ we simply write $\delta^{{\diamond},1}=\delta^{{\diamond}} $. 
 Recall that for  $h\in\ch: |h|_{\ch}=1$ we have the relation:  
 \begin{eqnarray}
\delta^{\diamond,k}(h^{\otimes k}) = H_{k}(\delta^{\diamond}(h)), 
\label{e.dkh}
\end{eqnarray}
where $H_{k}(x)= (-1)^{k}e^{x^{2}/2}\frac{d^{k}}{dx^{k}} e^{-x^{2}/2}$ is called the Hermite polynomial of order $k$. 

%
%
%
%
The following result is  an elaboration of \cite[Proposition 1.3.1]{N06}. 
\begin{lemma}\label{lem.dudv}
Let $h_{1},h_{2}\in\ch$ and $F_{1},F_{2}\in \DD^{1,2}$. Then we have the relation: 
\begin{eqnarray}\label{e.dfhdfh}
\mE[\delta^{\diamond} (F_{1}h_{1})\delta^{\diamond} (F_{2}h_{2})] = \mE[ F_{1}F_{2} ]\langle h_{1},h_{2} \rangle_{\ch} + \mE [ \langle DF_{1}, h_{2} \rangle_{\ch} \langle DF_{2}, h_{1} \rangle_{\ch} ]. 
\end{eqnarray}

\end{lemma}
\begin{proof}
By linearity of relation \eqref{e.dfhdfh} in $h_{1}$  and $h_{2}$  in order to show \eqref{e.dfhdfh} it suffices to consider two cases: (i)   $h_{1}=h_{2}$ and (ii) $h_{1}\perp\!\!\!\perp h_{2}$ such that $|h_{1}|_{\ch}=|h_{2}|_{\ch}=1$.


According to  \cite[Proposition 1.3.1]{N06}, for $u,v\in\DD^{1,2}(\ch)$ and   a complete orthonormal system  $\{e_{i},i\in\NN\}$ of $\ch$  we have the relation 
\begin{eqnarray}\label{e.dudv}
\mE[\delta^{\diamond} (u)\delta^{\diamond} (v)] = \mE [\langle u,v \rangle_{\ch}] + \mE   \sum_{i,j=1}^{\infty} [D_{e_{i}} \langle u, e_{j}\rangle_{\ch}\cdot
D_{e_{j}} \langle v, e_{i}\rangle_{\ch}]
   .  
\end{eqnarray}

In Case (i),  we take $ e_{1}=h_{1}=h_{2}$ and let  $\{e_{i},i\in\NN\}$ be a complete orthonormal system of $\ch$. Take $u=F_{1}h_{1}$, and   $v=F_{2}h_{2}$   in \eqref{e.dudv}.   A direct computation shows that  
\begin{eqnarray*}
\mE   \sum_{i,j=1}^{\infty} [D_{e_{i}} \langle u, e_{j}\rangle_{\ch}\cdot
D_{e_{j}} \langle v, e_{i}\rangle_{\ch}]
   = \mE \lc D_{e_{1}} \langle u, e_{1}\rangle_{\ch}\cdot
D_{e_{1}} \langle v, e_{1}\rangle_{\ch}
 \rc
 = \mE[D_{h_{1}}F_{1}D_{h_{2}}F_{2}].
\end{eqnarray*}
Substituting this into  \eqref{e.dudv} gives the relation \eqref{e.dfhdfh}. 
In Case (ii) we   take $ e_{1}=h_{1}$ and    $ e_{2}=h_{2}$. As before,  by  taking $u=F_{1}h_{1}$,  $v=F_{2}h_{2}$  in \eqref{e.dudv}   we obtain    the relation \eqref{e.dfhdfh}.
\end{proof}


  Let $x$ be a standard  fractional Brownian motion (fBm for short) with Hurst parameter $H\in(0,1)$, that is $x$ is a continuous Gaussian process with covariance function 
$
\mE[x_{s}x_{t}] = \frac12 (|s|^{2H}+|t|^{2H}-|s-t|^{2H})$. 
Recall that the fBm  $x$ is almost surely $\ga$-H\"older continuous for all $\ga<H$.  
 Let $\ch$ be the completion of the space of indicator functions with respect to the inner product $\langle \mathbf{1}_{[u,v]},  \mathbf{1}_{[s,t]}  \rangle_{\ch} =\mE(\delta x_{uv} \delta x_{st})$.  
 In the following we   state an estimate  result for indicate functions which we make extensive use of. 
 
\begin{lemma}\label{lem.abcd}
Let $\rho$ be the function $\rho(x)=|x|\vee1$ for $x\in\RR$. 
There exists a constant $K>0$ such that   the following estimate holds: 
\begin{eqnarray}
\left|
\langle \mathbf{1}_{[s,t]}, \mathbf{1}_{[u,v]} \rangle_{\ch} \right|&\leq& K|v-u|^{2H} \cdot\lp
 [|s'|^{2H-1}+|t'|^{2H-1}] \wedge 2^{2H-1}\rp
\label{e.abbd}
 \\
&\leq&K|v-u|^{2H} \cdot\lp \rho(s')^{2H-1}+ \rho(t')^{2H-1}\rp \label{e.abbd2}
\end{eqnarray}
for  all $(s,t)$ and $ (u,v)\in\cs_{2}(\RR)  $, where we   denote   $s' = \frac{s-u}{v-u} $ and $t'=\frac{t-u}{v-u}$.  

\end{lemma}
\begin{proof}
By self-similarity and stationarity of increments of fBm we have
\begin{eqnarray}\label{e.abcd}
\langle \mathbf{1}_{[s,t]}, \mathbf{1}_{[u,v]} \rangle_{\ch} = |v-u|^{2H} \langle \mathbf{1}_{[s',t']}, \mathbf{1}_{[0,1]} \rangle_{\ch}.
\end{eqnarray}
An elementary estimation gives: 
\begin{eqnarray*}
|\langle \mathbf{1}_{[s',t']}, \mathbf{1}_{[0,1]} \rangle_{\ch}|
&\leq&
|\langle \mathbf{1}_{[s',t']\cap [0,1]}, \mathbf{1}_{[0,1]} \rangle_{\ch}|+|\langle \mathbf{1}_{[s',t']\setminus [0,1]}, \mathbf{1}_{[0,1]} \rangle_{\ch}|
 \\
&\leq&   1+1 \leq 4\cdot 2^{2H-1}. 
\end{eqnarray*}
Substituting this into \eqref{e.abcd} we get: 
\begin{eqnarray}
\left|
\langle \mathbf{1}_{[s,t]}, \mathbf{1}_{[u,v]} \rangle_{\ch} \right|&\leq& 4\cdot 2^{2H-1} |v-u|^{2H} . 
\label{e.abbd1}
\end{eqnarray}
The estimate \eqref{e.abbd1} shows that the relation \eqref{e.abbd} holds when $|s'|^{2H-1}+|t'|^{2H-1}\geq 2^{2H-1}$.

It remains to prove  that  \eqref{e.abbd} holds under the condition that $|s'|^{2H-1}+|t'|^{2H-1}\leq 2^{2H-1}$. Note that this condition implies that $|s'|\geq 2$ and $|t'|\geq 2$, and so we must have either (i) $(s',t')\cap (0,1)=\emptyset$ or  (ii) 
$(0,1)\subset (s',t') $.

Assume first that $(s',t')\cap (0,1)=\emptyset$. Then  either   $t'\leq 2$ or $s'\geq 2$.  When    $t'\leq 2$,    we have
\begin{eqnarray}\label{e.abp}
|\langle \mathbf{1}_{[s',t']}, \mathbf{1}_{[0,1]} \rangle_{\ch}| \leq  |\langle \mathbf{1}_{(-\infty,t']}, \mathbf{1}_{[0,1]} \rangle_{\ch}| \lesssim
|t'|^{2H-1},   
\end{eqnarray}
which together with \eqref{e.abcd} gives the estimate \eqref{e.abbd}. 
In the same way we can show that  \eqref{e.abbd} also holds when     $s'\geq 2$. 

   Suppose now  that $(0,1)\subset (s',t') $. In this case we have:  
   \begin{eqnarray}
\langle \mathbf{1}_{[s',t']}, \mathbf{1}_{[0,1]} \rangle_{\ch} = \langle \mathbf{1}_{(-\infty,s' ]}, \mathbf{1}_{[0,1]} \rangle_{\ch}+\langle \mathbf{1}_{[t',\infty )}, \mathbf{1}_{[0,1]} \rangle_{\ch}. 
\label{e.abp2}
\end{eqnarray}
Since $(-\infty,s' )\cap (0,1)=\emptyset$ and $(t', \infty )\cap (0,1)=\emptyset$  we can   apply \eqref{e.abp} to the terms in the right-hand side of \eqref{e.abp2}. This gives the estimate: 
\begin{eqnarray}
|\langle \mathbf{1}_{[s',t']}, \mathbf{1}_{[0,1]} \rangle_{\ch} |\lesssim |s'|^{2H-1}+|t'|^{2H-1}. 
\label{e.abp3}
\end{eqnarray}
The estimate  \eqref{e.abp3}   together with \eqref{e.abcd} gives the estimate \eqref{e.abbd}. 
In summary,  we have obtained that  \eqref{e.abbd} holds for both cases (i) and   (ii). The proof is now complete. 
\end{proof}

\subsection{Some lemmas}\label{section.lem}
In this subsection we prove some technical lemmas. Our first result    provides  convenient  enlargements for   certain upper-bound estimates. 
\begin{lemma}\label{lem.stnbd}
Let    $A$ be a function from $\cs_{2}(\ll0,T\rr)\times \NN $ to $ \RR$. Assume that there exists   constants  $  \al>0$ and $ \be,\ga, K_{1}\geq0$    such that 
\begin{eqnarray}
A(s,t,n)\leq K_{1}(1/n)^{\al}(t-s)^{\be}[\log (n(t-s))]^{\ga}
\label{e.abd1}
\end{eqnarray}
for all $(s,t,n)\in\cs_{2}(\ll0,T\rr)\times \NN$. Then for any $\ep>0 $   there is a constant $K_{2}>0$ such that  
\begin{eqnarray}
A(s,t,n)\leq K_{2}(1/n)^{\al-\ep}(t-s)^{\be+\ep} 
\label{e.abd5}
\end{eqnarray}
for all  $(s,t,n)\in\cs_{2}(\ll0,T\rr)\times \NN$. Furthermore, we have the convergence:
\begin{eqnarray}
\lim_{n\to\infty} \sup_{(s,t)\in\cs_{2}(\ll0,T\rr)} n^{\al-\ep}  A(s,t,n) =0.
\label{e.alim}
\end{eqnarray}

\end{lemma}
\begin{proof}
Since $(\log x)^{\ga}/x^{\ep}\to 0$ as $x\to\infty$,  we can find a constant $K>0$ such that  
\begin{eqnarray}
[\log (n(t-s))]^{\ga} \leq   K (n(t-s))^{\ep} 
\label{e.abd4}
\end{eqnarray}
for all  $(s,t,n)\in\cs_{2}(\ll0,T\rr)\times \NN$.
Substituting \eqref{e.abd4} into the right-hand side of \eqref{e.abd1} we obtain \eqref{e.abd5}. To see that \eqref{e.alim} holds, replacing $\ep$ by $\ep/2$ in \eqref{e.abd5} to get: 
\begin{eqnarray*}
A(s,t,n)\leq K_{2}(1/n)^{\al-\ep/2}(t-s)^{\be+\ep/2} .
\end{eqnarray*}
This gives 
\begin{eqnarray*}
n^{\al-\ep}A(s,t,n) \leq  K_{2}(1/n)^{\ep/2}(t-s)^{\be+\ep/2}
 \leq  K_{2}(1/n)^{\ep/2}T^{\be+\ep/2}.  
\end{eqnarray*}
Now sending $n\to\infty$ we obtain the convergence \eqref{e.alim}. 
\end{proof}

In the  following we state  an elementary   result for $p$-series. 
\begin{lemma}\label{lem.rho}
Let $\rho$ be the function defined by: $\rho(x)=|x|\vee 1$, $x\in\RR$. Take $\al\leq 0$. 

\noindent\emph{(i)} 
There exists a constant $K>0$ such that   the following estimate holds: 
\begin{eqnarray}
\sum_{k: \, sn\leq k<tn} \rho(k-k')^{\al} &\leq& 
\begin{cases}
K(n(t-s))^{\al+1}&\al>-1\\
K\log (n(t-s))&\al=-1\\
K&\al<-1
\end{cases}
\label{e.rhobd}
\end{eqnarray}
for all $(s,t)\in\cs_{2}([0,T])$ and $k'\in[0,n]$. 

\noindent\emph{(ii)} When $k'\not\in[sn,tn)$ and $\al<-1$ we   have the estimate:
\begin{eqnarray}
\sum_{k: \,sn\leq k<tn} \rho(k-k')^{\al} &\leq& 
\begin{cases}
K\rho(k'-ns)^{\al+1} &k'<ns.
\\
K\rho(k'-nt)^{\al+1} &k'\geq nt.
\end{cases}
\label{e.rhobd3}
\end{eqnarray}

\end{lemma}
\begin{proof}
We first consider the case when $k' \in[sn,tn)$. Write:     
\begin{eqnarray}
\sum_{k: \,sn\leq k<tn} \rho(k-k')^{\al} &=&\sum_{k: \,k'\leq k<tn} \rho(k-k')^{\al} +\sum_{k: \,sn\leq k<k'} \rho(k-k')^{\al}.
\label{e.rhobd1}
\end{eqnarray}
Applying  a change of variables to each sum in the right-hand side of \eqref{e.rhobd1} we have 
\begin{eqnarray*}
\sum_{k: \,sn\leq k<tn} \rho(k-k')^{\al}  = \sum_{k: \,0\leq k<tn-k'} \rho(k)^{\al} +\sum_{k: \,0< k\leq k'-ns} \rho(k)^{\al}  . 
\end{eqnarray*}
It follows that 
\begin{eqnarray}
\sum_{k: \,sn\leq k<tn} \rho(k-k')^{\al}  \leq  2\sum_{k: \,0\leq k<(t-s)n} \rho(k)^{\al}
\leq 2\int_{1}^{(t-s)n} x^{\al}dx . 
\label{e.rhobd2}
\end{eqnarray}
 The estimate \eqref{e.rhobd3} then follows. 
 
 Assume now that  $k'<sn $. As before, a change of variables gives:  
\begin{eqnarray}
\sum_{k: \,sn\leq k<tn} \rho(k-k')^{\al}  = \sum_{k: \,sn-k'\leq k<tn-k'} \rho(k)^{\al}
\leq  \int_{sn-k'}^{tn-k'} x^{\al}dx
 .   
\label{e.rhobd4}
\end{eqnarray}
It follows that   the estimate \eqref{e.rhobd} holds. In the same way we can prove \eqref{e.rhobd}  for  $k'\geq tn$.  In summary, we have proved  Lemma \ref{lem.rho} (i).   
 


To see that   \eqref{e.rhobd3} is true, we note that when $\al<-1$ the right-hand side of \eqref{e.rhobd4} can be bounded by $ \int_{sn-k'}^{\infty} x^{\al}dx$, and thus by $ \rho(k'-sn)^{\al+1}$.   
 The estimate \eqref{e.rhobd3} then follows. In the same way we can prove the estimate for $k'>tn$. 
\end{proof}

We end the section by stating the following classical result. 
\begin{lemma}\label{lem.limlim}
Let $U_{n}, U_{\infty}, V_{n,M}, \wt{V}_{n,M}$, $M,n\in\NN$ be   random variables with values in $\RR^{N}$, and assume  that $U_{n}=V_{n,M}+ \wt{V}_{n,M}$ for all $M,n\in\NN$.  Suppose that  
\begin{eqnarray}
\lim_{M\to\infty}\lim_{n\to\infty}P(V_{n,M}\leq x) =P(U_{\infty}\leq x), 
\qquad x\in\RR^{N}. 
\label{e.vnmconv}
\end{eqnarray}
 Denote $\wt{V}_{M,n}= (\wt{V}_{M,n}^{1}, \dots, \wt{V}_{M,n}^{N})$ and assume  that 
\begin{eqnarray}
\limsup_{M\to\infty} \limsup_{n\to\infty}P(|\wt{V}_{n,M}^{i}|\geq \ep) \to 0
\qquad
\text{as $\ep\to0$ }
\label{e.limlim}
\end{eqnarray}
for all $i=1,\dots, N$. 
 Then for any   $x\in\RR^{N}$ such that $P(U_{\infty}=x)=0$  we have 
 \begin{eqnarray}
\lim_{n\to\infty}P(U_{n}\leq x) = P(U_{\infty}\leq x) .
\label{e.ulim}
\end{eqnarray} 
\end{lemma}
\begin{proof}
 We first   have the following elementary relation for $\ep>0$ and $x\in\RR^{N}$:
\begin{eqnarray*}
P(V_{M,n} \leq x-\ep) - \sum_{i=1}^{N}P(|\wt{V}_{M,n}^{i}|\geq \ep)    \leq P(V_{M,n}+\wt{V}_{M,n}\leq x)
\\
 \leq P(V_{M,n} \leq x+\ep) +  \sum_{i=1}^{N}P(|\wt{V}_{M,n}^{i}|\geq \ep)  .     
\end{eqnarray*}
Sending $n\to\infty$ and then $M\to\infty$, and substituting the relation $V_{M,n}+\wt{V}_{M,n}=U_{n}$ gives:
\begin{eqnarray*}
P(U_{\infty} \leq x-\ep) - \sum_{i=1}^{N}\limsup_{M\to\infty}\limsup_{n\to\infty}P(|\wt{V}_{M,n}^{i}|\geq \ep)   \leq \liminf_{n\to\infty}P(U_{n}\leq x) , 
\\
\text{and}\quad
\limsup_{n\to\infty}P(U_{n}\leq x)  \leq  P(U_{\infty} 
\leq x+\ep) + \sum_{i=1}^{N}\limsup_{M\to\infty}\limsup_{n\to\infty}P(|\wt{V}_{M,n}^{i}|\geq \ep) .     
\end{eqnarray*}
Sending $\ep\to\infty$ and invoking relation \eqref{e.limlim} and  that $P(U_{\infty}=x)=0$ we obtain~\eqref{e.ulim}. 
\end{proof}

\section{Limit theorem for  compensated weighted   sums}\label{section.lim}
This section aims to prove our main limit theorem results. 
 In Section \ref{section.skorohod} we show the convergence of the Skorohod-type Riemann-Stieltjes sum; see definition  in  \eqref{e.zn}. Then  in Section \ref{section.lt} we establish     limit theorem    for the compensated weighted sum  defined \eqref{e.ce}. 
 
  \subsection{Limit theorem  for   Skorokhod-type  sums}\label{section.skorohod}
  Let $\cp: 0=t_{0}<\cdots<t_{n}=T$ be the uniform   partition of $[0,T]$. For $t\in(t_{k},t_{k+1}]$ we denote $\la(t)=t_{k+1}$. This  section is devoted   to proving   convergence of the following  Skorohod-type Riemann-Stieltjes sum:
  \begin{eqnarray}
Z^{(n)}_{st} :=Z^{(n), i}_{st} :=   n^{H+1/2}\sum_{s\leq t_{k}<t}\int_{t_{k}}^{t_{k+1}}
\delta^{\diamond} \big(
  x^{i-1 }_{\la(s)t_{k}}\mathbf{1}_{[t_{k},v]}   \big) dv   
  \label{e.zn}
\end{eqnarray}
for $(s,t)\in\cs_{2}([0,T]) $ and $i\in\NN$, where recall that   $x$ is the fBm with Hurst parameter $H<1/2$ and  $x^{i}_{st} $ stands for $ (\delta x_{st})^{i}/i!$. 
   In the following we denote the constant:   \begin{eqnarray}
   c_{H} = \sum_{k\in\ZZ}\mu(k) ,
   \qquad
\mu (k) = \int_{ {0}}^{ {1}} \int_{ { k}}^{ { k+1}}
  \langle \mathbf{1}_{[  k,v']}      ,  
  \mathbf{1}_{[ 0,v]}     
   \rangle_{\ch} dv' dv .
   \label{e.rho}
\end{eqnarray}

\begin{theorem}\label{prop.du}
Let $Z^{(n)}$ be defined in \eqref{e.zn}, and $Z$ be  the process: 
\begin{eqnarray}
Z_{st}:=Z_{st}^{i}:=  {c_{H}}^{1/2}T^{H+1/2} \int_{s}^{t} x^{i-1}_{su}dW_{u} 
\qquad (s,t)\in\cs_{2}([0,T]), 
\label{e.zist}
\end{eqnarray}
where $W$ is a standard Brownian motion independent of $x$   and  $c_{H}$ is the constant defined in \eqref{e.rho}. Then the finite dimensional distribution  of $Z^{(n)}$ converges $\cf^{x}$-stably (see e.g. \cite{JS}) to that of $Z$.  (For convenience, we simply write $Z^{(n)}\xrightarrow{stable\,\, f.d.d. } Z$ in the sequel.)  Precisely, 
  for any partition $\pi: 0=s_{0}<s_{1}<\cdots<s_{N}=T$ of $[0,T]$ and $N\in\NN$ we have  
     the following convergence in law:
\begin{eqnarray}\label{e.zpic}
(Z_{\pi}^{(n)}, x_{\pi}) \to (Z_{\pi}, x_{\pi}) 
\qquad \text{as $n\to\infty$,}
\end{eqnarray}
where 
\begin{eqnarray}
Z_{\pi} &=& (Z_{s_{j}s_{j+1}},   j=0,\dots,N-1  ),
\qquad
x_{\pi}  =  (x_{s_{j+1}},  j=0,\dots,N-1  ),
\notag
\\
Z^{(n)}_{\pi}  &=&  (Z^{(n)}_{s_{j}s_{j+1}},   j=0,\dots,N-1  ).
\label{e.znpi}
\end{eqnarray} 
\end{theorem}
\begin{remark}
By  following the same argument below it can be shown that the convergence \eqref{e.zpic} still holds (possibly with a different constant $c_{H}$ in the limit of \eqref{e.zpic}) when the uniform parittion $\cp$ is replaced  by a non uniform one,  as long as the mesh size of the partition tends to zero. 
\end{remark}

\begin{proof}[Proof of Theorem \ref{prop.du}] 
Take $z=(z_{0},\dots,z_{2N-1})\in\RR^{2N}$ and denote  \begin{eqnarray}
z_{(1)}  = (z_{0},\dots,z_{N-1})
\qquad
\text{and}
\qquad
z_{(2)} = (z_{N},\dots,z_{2N-1}) .
\notag
\end{eqnarray}
 Let $\cz^{(n)}$ be    difference between   characteristic functions of $(Z^{(n)}_{\pi}, x_{\pi})$ and $(Z_{\pi}, x_{\pi})$. Namely:
\begin{eqnarray}\label{e.cz}
\cz^{(n)} =\mE\lc \exp \lp \mathbf{i}z_{(1)}Z^{(n)}_{\pi}+\mathbf{i}z_{(2)}x_{\pi}\rp \rc - \mE\lc \exp \lp \mathbf{i}z_{(1)}Z_{\pi}+\mathbf{i}z_{(2)}x_{\pi}\rp \rc ,  
\end{eqnarray}
where $z_{(1)}Z^{(n)}_{\pi} $, $ z_{(2)}x_{\pi}$ and $z_{(1)}Z_{\pi}$ should be interpreted as    dot products of vectors. 
In the following we show that $\cz^{(n)}\to0$ as $n\to\infty$, which then gives the convergence \eqref{e.zpic}.  Our proof is divided into several steps. 

\noindent\emph{Step 1: A decomposition for $\cz^{(n)} $.} 
Since $Z_{\pi} $ conditioned on $x$ is Gaussian  its characteristic function can          be written as:   
\begin{eqnarray}
\mE\lc \exp \lp \mathbf{i}z_{(1)}Z_{\pi}+\mathbf{i}z_{(2)}x_{\pi}\rp \rc &=& \mE^{x}\mE^{W}\lc \exp \lp \mathbf{i}z_{(1)} Z_{\pi}+\mathbf{i}z_{(2)}x_{\pi}\rp \rc 
\notag
\\
&=& \mE^{x} \big[ \exp \big( -\frac12z_{(1)}^{2} F_{\pi} +\mathbf{i}z_{(2)} x_{\pi} \big) \big], 
\label{e.chzn}
\end{eqnarray}
where    
\begin{eqnarray}\label{e.fpi}
F_{st} &=& {c_{H}}{T^{2H+1}}\int_{s}^{t} (x^{i-1}_{su})^{2}d{u} 
\qquad \text{and} \qquad F_{\pi} = (F_{s_{j}s_{j+1}},\, j=0,\dots,N-1) , 
\end{eqnarray}
and we denote    $z_{(1)}^{2} = (z_{0}^{2},\dots,z_{N-1}^{2})$.

We consider the so-called smart-path interpolation function (see \cite{nourdin2016quantitative, nourdin2012normal}):   
\begin{eqnarray}\label{e.interpolation}
\phi(\theta) =  \exp\big( \mathbf{i}\theta z_{(1)} Z_{\pi}^{(n)}-\frac12  (1-\theta^{2})z_{(1)}^{2}F_{\pi}+  \mathbf{i}z_{(2)}x_{\pi}\big)
\qquad \theta\in[0,1]. 
\end{eqnarray}
By taking $\theta=0$ and $\theta=1$ in \eqref{e.interpolation} and invoking relations \eqref{e.chzn} it is easily seen that $\phi(1)$ and $\phi(0)$ are identified with the two terms in the right-hand side of \eqref{e.cz}, respectively. Namely, we have $\cz^{(n)} = \mE[\phi(1)]-\mE[[\phi(0) ]$. 
By applying   Newton-Leibniz formula we thus have:  
\begin{eqnarray}\label{e.czn}
\cz^{(n)}  
=\int_{0}^{1}\mE[\phi'(\theta) ]d\theta
. 
\end{eqnarray}

Differentiating \eqref{e.interpolation}  we obtain:   
\begin{eqnarray}\label{e.phip}
\phi'(\theta) = \phi(\theta) (\mathbf{i} z_{(1)} Z_{\pi}^{(n)}+\theta z_{(1)}^{2}F_{\pi}) . 
\end{eqnarray}
Substituting  equation \eqref{e.phip}   into \eqref{e.czn} gives:  
\begin{eqnarray}
\cz^{(n)}  &=&\cz^{(n)}_{1}  +\cz^{(n)}_{2}\,, 
\label{e.czn2}
\end{eqnarray}
where
\begin{eqnarray}
 \cz^{(n)}_{1}= \mathbf{i} \int_{0}^{1}\mE\phi(\theta) z_{(1)}Z^{(n)}_{\pi}  d\theta
 \qquad\text{and}\qquad
  \cz^{(n)}_{2}= \int_{0}^{1}\mE\phi(\theta)  \theta z_{(1)}^{2}F_{\pi} d\theta . 
  \label{e.i1i2}
\end{eqnarray}

\noindent\emph{Step 2: A decomposition of $\cz^{(n)}_{1} $.} Recall that $Z^{(n)}$ is defined in \eqref{e.zn}. Substituting \eqref{e.zn} into  the first equation in  \eqref{e.i1i2}  and then applying   integration by parts formula \eqref{e.ibp} we obtain:
\begin{eqnarray}\label{e.ij1new}
\cz^{(n)}_{1}  =\sum_{j=0}^{N-1} \mathbf{i}n^{H+1/2}z_{j} \int_{0}^{1}\sum_{s_{j}\leq t_{k}<s_{j+1}}\int_{t_{k}}^{t_{k+1}}\mE\langle D\phi(\theta) ,  
  x^{i-1 }_{\la(s_{j})t_{k}}\mathbf{1}_{[t_{k},v]}     
   \rangle_{\ch} dv  d\theta . 
\end{eqnarray}
Let us work out the derivative $D\phi(\theta)$  in \eqref{e.ij1new}.   Differentiating \eqref{e.interpolation} gives:  
\begin{eqnarray}
D\phi(\theta) &=&  \phi(\theta) \big( \mathbf{i}\theta z_{(1)}  DZ^{(n)}_{\pi} -\frac12  (1-\theta^{2})z_{(1)}^{2}D F_{\pi} + \mathbf{i}z_{(2)}Dx_{\pi} \big)
\notag
 \\
&=& \sum_{j=0}^{N-1} (\cc_{0}^{j}+\cc_{3}^{j}+\cc_{4}^{j}),
\label{e.dphi2}
\end{eqnarray}
where
\begin{eqnarray}
\cc_{0}^{j}&=&\phi(\theta)   \mathbf{i}\theta z_{j}  DZ^{(n)}_{s_{j}s_{j+1}}
    \label{e.c0}
\\
\cc_{3}^{j} &= &-\frac12  (1-\theta^{2})\phi(\theta) z_{j}^{2} 
  D F_{s_{j}s_{j+1}}  
\notag  \\
  &=& - { c_{H}}{T^{2H+1}} (1-\theta^{2})\phi(\theta)  z_{j}^{2} 
   \int_{s_{j}}^{s_{j+1}}  2x^{i-1}_{s_{j}u}x^{i-2}_{s_{j}u} \mathbf{1}_{[s_{j},u]} du
     \label{e.c3}
  \\
   \cc_{4}^{j}&=&\mathbf{i}\phi(\theta) z_{N+j}Dx_{s_{j+1}}=\mathbf{i}\phi(\theta) z_{N+j}\mathbf{1}_{[0,s_{j+1}]}.
    \label{e.c4}
\end{eqnarray}

By differentiating $Z^{(n)}_{st}$ in  \eqref{e.zn} and then applying the relations $D(\delta^{\diamond} (u)) = u+\delta^{\diamond} (Du)$ (see e.g. \cite{N06}) and $Dx^{i-1 }_{\la(s)t_{k}} = x^{i-2 }_{\la(s)t_{k}}\mathbf{1}_{[\la(s),t_{k}]}$,  we obtain: 
\begin{eqnarray}
DZ^{(n)}_{st}&=&n^{H+1/2}\sum_{s\leq t_{k}<t}\int_{t_{k}}^{t_{k+1}}  
  x^{i-1 }_{\la(s)t_{k}}\mathbf{1}_{[t_{k},v]}     dv
 \notag
  \\
  &&+n^{H+1/2}\sum_{s\leq t_{k}<t}\int_{t_{k}}^{t_{k+1}}  
\delta^{\diamond}(  x^{i-2 }_{\la(s)t_{k}}\mathbf{1}_{[t_{k},v]})\mathbf{1}_{[\la(s),t_{k}]}     dv ,    
\label{e.dzn}
\end{eqnarray}
where we  use the convention that $x^{i}\equiv 0$ when  $i=-1,-2,\dots$. 
Substituting \eqref{e.dzn} into \eqref{e.c0} with $s=s_{j}$ and $t=s_{j+1}$ gives the following relation for $D\phi(\theta) $: 
\begin{eqnarray}
\cc_{0}^{j}&=&\cc_{1}^{j}+\cc_{2}^{j},
\label{e.c01}
\end{eqnarray}
where
\begin{eqnarray}
\cc_{1}^{j} &=&\mathbf{i} n^{H+1/2} \theta z_{j}  \phi(\theta) \sum_{s_{j}\leq t_{k}<s_{j+1}}\int_{t_{k}}^{t_{k+1}}
  x^{i-1 }_{\la(s_{j}) t_{k}}\mathbf{1}_{[t_{k},v]}     dv 
  \label{e.c1}
  \\
  \cc_{2}^{j} &= & \mathbf{i} n^{H+1/2} \theta z_{j}\phi(\theta)  
  \sum_{s_{j}\leq t_{k}<s_{j+1}}\int_{t_{k}}^{t_{k+1}}
\delta^{\diamond} \lp
  x^{i -2}_{\la(s_{j})t_{k}}\mathbf{1}_{[t_{k},v]}  \rp  \mathbf{1}_{[ \la(s_{j}),t_{k}]} dv.
  \label{e.c2}
\end{eqnarray}

Substituting  \eqref{e.c01} into \eqref{e.dphi2} we obtain: 
\begin{eqnarray}
D\phi(\theta) &=& \sum_{j=0}^{N-1}(\cc_{1}^{j}+\cc_{2}^{j}+\cc_{3}^{j}+\cc_{4}^{j}). 
\label{e.dphi}
\end{eqnarray}
 Substituting \eqref{e.dphi} into \eqref{e.ij1new} we obtain the following decomposition of $\cz^{(n)}_{1} $:  
\begin{eqnarray}
\cz^{(n)}_{1}  = 
 \ci_{11}+ \ci_{12}+ \ci_{13}+ \ci_{14} \,,
\label{e.z1d}
\end{eqnarray}
where
\begin{eqnarray}
\ci_{1e}   =\sum_{j_{1},j_{2}=0}^{N-1} \mathbf{i}n^{H+1/2}z_{j_{1}}\int_{0}^{1}\sum_{s_{j_{1}}\leq t_{k_{1}}<s_{j_{1}+1}}\int_{t_{k_{1}}}^{t_{k_{1}+1}}\mE\langle \cc_{e}^{j_{2}} ,  
  x^{i-1 }_{\la(s_{j_{1}})t_{k_{1}}}\mathbf{1}_{[t_{k_{1}},v_{1}]}     
   \rangle_{\ch} 
 \notag  \\
 \cdot  dv_{1}  d\theta 
. 
   \label{e.ic}
\end{eqnarray}

In fact, we can further expand  relation \eqref{e.z1d} into the following:
 \begin{eqnarray}
\cz^{(n)}_{1}  = 
\sum_{e=1,3,4}\ci_{1e}  + \sum_{e=1,3,4,5}\ci_{2e}  +\cdots+ \sum_{e=1,3,4,5}\ci_{Me}   +\ci_{M2}, \quad M\in\NN,
\label{e.czn1}
\end{eqnarray}
where
  \begin{eqnarray}
\ci_{ae} &=&  \sum_{j_{1},\cdots,j_{a}=0}^{N-1}  \sum_{\substack{s_{j_{1}}\leq t_{k_{1}} <s_{j_{1}+1}\\\vdots\\s_{j_{a}}\leq  t_{k_{a}}<s_{j_{a}+1}}} 
\int_{0}^{1}\Big(\prod_{c=1}^{a-1}g_{c}\Big)
 \vp_{ae}^{j_{1} \cdots  j_{a}} d\theta\,, 
 \qquad e=1,2,3,4,5
\label{e.iaej}
     \\
\vp_{ae}^{j_{1} \cdots  j_{a}}&=&\sum_{j_{a+1}=0}^{N-1} \vp_{ae}^{j_{1} \cdots  j_{a+1}}, \quad e=1,2,3,4
\label{e.fae}
\\
\vp_{ae}^{j_{1} \cdots  j_{a+1}}&=&  \mathbf{i}   z_{j_{a}}n^{H+1/2} \int_{t_{k_{a}}}^{t_{k_{a}+1}} \mE  \lc
 \langle \cc_{e}^{j_{a+1}},  
  \mathbf{1}_{[t_{k_{a}},v_{a}]} \rangle_{\ch}
 \cg_{a}\rc
      dv_{a},  
\label{e.faej}
\\
      \vp_{a5}^{j_{1} \cdots  j_{a}}&=&  \mathbf{i}   z_{j_{a}}n^{H+1/2} \int_{t_{k_{a}}}^{t_{k_{a}+1}} \mE    \Big[
 \langle D {\cg}_{a-1}  ,  
  \mathbf{1}_{[t_{k_{a}},v_{a}]} \rangle_{\ch}x^{i-2 }_{\la(s_{j_{a}})t_{k_{a}}}\phi(\theta)
 \Big]
      dv_{a}
       \label{e.fa5} \\ 
     g_{c}&=& \mathbf{i} \theta z_{j_{c}}n^{H+1/2}
\int_{t_{k_{c}}}^{t_{k_{c}+1}}
 \langle  
 \mathbf{1}_{[ \la(s_{j_{c+1}}),t_{k_{c+1}}]}   ,  
\mathbf{1}_{[t_{k_{c}},v_{c}]}     
   \rangle_{\ch}dv_{c} 
   \label{e.gc}
   \\
     \cg_{a} &=&   \Big(\prod_{c=2}^{a}x^{i-2 }_{\la(s_{j_{c}})t_{k_{c}}}
    \Big) 
     x^{i-1 }_{\la(s_{j_{1}})t_{k_{1}}} , 
        \label{e.Ga} 
\end{eqnarray}
and  we have used the convention that $\prod_{c=1}^{0}g_{c} =\prod_{c=2}^{1}x^{i-2 }_{\la(s_{j_{c}})t_{k_{c}}} =1$ in \eqref{e.iaej} and \eqref{e.Ga}. 
Indeed,   similar to   computations in \eqref{e.ij1new}, by applying integration by parts   to $\ci_{12}$ in \eqref{e.ic} and invoking relation \eqref{e.dphi} we can show that $\ci_{12} =\sum_{e=1}^{5} \ci_{2e} $, where $\ci_{2e}$ are defined in \eqref{e.iaej}. In general, applying integration by parts to $\ci_{a2}$ and then relation \eqref{e.dphi}  gives:  
\begin{eqnarray}
\ci_{a2} &=&\sum_{e=1}^{5} \ci_{a+1,e} 
\qquad
\text{for all $a\in\NN$.}
\label{e.ia2new}
\end{eqnarray}
 Now,    substituting \eqref{e.ia2new} into \eqref{e.z1d} with $a=1,2,\dots$ we obtain  relation \eqref{e.czn1}. For sake of conciseness we omit    details of   proof for \eqref{e.czn1} and leave it to the patient reader.   

For   later use, let us  note that with the   notations in \eqref{e.iaej} the equation \eqref{e.ic} becomes 
\begin{eqnarray}
\ci_{11} = \sum_{j_{1},j_{2}=0}^{N-1} \ci_{11}^{j_{1}j_{2}}   
\qquad\text{with}\quad
\ci_{11}^{j_{1}j_{2}} =  \sum_{ s_{j_{1}}\leq t_{k_{1}} <s_{j_{1}+1} } 
\int_{0}^{1} 
 \vp_{11}^{j_{1}  j_{2}} d\theta. 
\label{e.i11ij}
\end{eqnarray}

\noindent\emph{Step 3: Estimate  for $\vp_{ae}$.}
Recall that $\vp_{ae}$,   $e=1,2,3,4$ are defined in \eqref{e.faej}. 
By applying Cauchy-Schwartz inequality in \eqref{e.faej}  we have the estimate:
\begin{eqnarray}
| \vp_{ae}^{j_{1} \cdots  j_{a+1}}|  \lesssim   \bar{\vp}_{ae}^{j_{a}  j_{a+1}}:=n^{H-1/2} \sup_{v_{a}\in[t_{k_{a}}, t_{k_{a}+1}]} | 
 \langle \cc_{e}^{j_{a+1}},  
  \mathbf{1}_{[t_{k_{a}},v_{a}]} \rangle_{\ch}
  |_{L_{p}}
 . 
 \label{e.faebd2}
\end{eqnarray}

In the following we derive the estimate for  $ \bar{\vp}_{ae}$. 
  For sake of simplicity we will write $k, k', j, j' $ instead of $ k_{a}, k_{a+1}, j_{a}, j_{a+1} $. Without loss of  generality we     take $T=1$, and thus $nt_{k}=k$.   

{\bf Estimate for $\mathbf e=1.$}   
By the definition of $\cc_{1}^{j}$ in \eqref{e.c1} it is clear that 
\begin{eqnarray}
\qquad\quad
\langle \cc_{1}^{j'},  
  \mathbf{1}_{[t_{k},v]} \rangle_{\ch} = 
  \mathbf{i} n^{H+1/2} \theta z_{j'}  \phi(\theta)\!\! \sum_{s_{j'}\leq t_{k'}<s_{j'+1}}\int_{t_{k'}}^{t_{k'+1}}
  x^{i-1 }_{\la(s_{j'})t_{k'}}
  \langle \mathbf{1}_{[t_{k'},v']} ,  
  \mathbf{1}_{[t_{k},v]} \rangle_{\ch}
      dv' . 
      \label{e.c1bd1}
\end{eqnarray}
By self-similarity of the fBm we have the relation: 
\begin{eqnarray}
\langle \mathbf{1}_{[t_{k'},v']} ,  
  \mathbf{1}_{[t_{k},v]} \rangle_{\ch} = n^{-2H} \langle \mathbf{1}_{[ {k'},nv' ]} ,  
  \mathbf{1}_{[ {k},nv ]} \rangle_{\ch}.
\label{e.ssim}
\end{eqnarray}
Note that in \eqref{e.c1bd1} we have     $s_{j'}\leq t_{k'}\leq v'\leq t_{k'+1}<s_{j'+1}$ and $s_{j}\leq t_{k}\leq v\leq t_{k+1}<s_{j+1}$. 
This implies that $ nv \in[k,k+1]$ and  $ nv \in[k',k'+1]$. So we can bound \eqref{e.ssim} by:
\begin{eqnarray*}
|\langle \mathbf{1}_{[t_{k'},v']} ,  
  \mathbf{1}_{[t_{k},v]} \rangle_{\ch} |&\lesssim& n^{-2H}\rho(k-k')^{2H-2}. 
\end{eqnarray*}
 It follows that the right-hand side of  \eqref{e.c1bd1} is bounded by
\begin{eqnarray}
  \big|
\langle \cc_{1}^{j'},  
  \mathbf{1}_{[t_{k},v]} \rangle_{\ch}
  \big|_{L_{p}}
   \lesssim n^{H+1/2 } \sum_{s_{j'}\leq t_{k'}<s_{j'+1}}\int_{t_{k'}}^{t_{k'+1}}
  n^{-2H} \rho(k-k')^{2H-2}
  dv'
  \notag
    \\
  \lesssim n^{-H-1/2}\sum_{s_{j'}\leq t_{k'}<s_{j'+1}} \rho(k-k')^{2H-2}. 
  \label{e.c1bd2}
\end{eqnarray}
Applying the relations \eqref{e.rhobd}-\eqref{e.rhobd3} in  Lemma \ref{lem.rho} to the right-hand side of \eqref{e.c1bd2} with $\al=2H-2<-1$ we    get 
\begin{eqnarray}
\qquad\quad |
\langle \cc_{1}^{j'},  
  \mathbf{1}_{[t_{k},v]} \rangle_{\ch}
   |_{L_{p}}
   \lesssim 
   \begin{cases}
   n^{-H-1/2} &j=j' \\
  n^{-H-1/2} (\rho(k-ns_{j'})^{2H-1}+\rho(k-ns_{j'+1})^{2H-1})&j\neq j' . 
  \end{cases}
  \label{e.c1bd3}
\end{eqnarray} 
Recall that $\bar{\vp}_{a1}$ is defined in \eqref{e.faebd2}. Summing up \eqref{e.c1bd3} in $t_{k}$ and then 
invoking   Lemma \ref{lem.rho}~(i)  with $\al =2H-1$, and taking into account the relation 
 \eqref{e.faebd2} and $j=j_{a}$, $j'=j_{a+1}$, $k=k_{a}$ and $k'=k_{a+1}$,       we obtain: 
\begin{eqnarray}\label{e.fa1bd}
\sum_{s_{j_{a}}\leq  t_{k_{a}}<s_{j_{a}+1}}  \bar{\vp}_{a1}^{j_{a} j_{a+1}} &\lesssim&
\begin{cases}
1&j_{a}= j_{a+1}.
\\
n^{2H-1}&j_{a}\neq j_{a+1}. 
\end{cases}
\end{eqnarray}

{\bf Estimate for $\mathbf e=2.$} 
The estimate of $\bar{\vp}_{a2}$ can be derived in the similar way: 
By the definition of $\cc_{2}^{j}$   in  \eqref{e.c2} we have 
\begin{eqnarray}
\langle \cc_{2}^{j'},  
  \mathbf{1}_{[t_{k},v]} \rangle_{\ch}  = 
   \mathbf{i} n^{H+1/2} \theta z_{j'}\phi(\theta)  
  \sum_{s_{j'}\leq t_{k'}<s_{j'+1}}\int_{t_{k'}}^{t_{k'+1}}
\delta^{\diamond} \big(
  x^{i-2 }_{\la(s_{j'})t_{k'}}\mathbf{1}_{[t_{k'},v']}  \big)  
  \notag\\
\cdot  \langle \mathbf{1}_{[ \la(s_{j'}),t_{k'}]} ,  
  \mathbf{1}_{[t_{k},v]} \rangle_{\ch}
  dv'.
  \label{e.c2bd2}
\end{eqnarray}

Applying the identity $\delta^{\diamond}(Fu)=F\delta^{\diamond} (u)+\langle DF, u\rangle_{\ch}$  it is easy to see that we have: 
\begin{eqnarray}
\big|\delta^{\diamond} \big(
  x^{i-2 }_{\la(s_{j'})t_{k'}}\mathbf{1}_{[t_{k'},v']}  \big) \big|_{L_{p}}
    &\lesssim& \big| x^{i-2 }_{\la(s_{j'})t_{k'}}\delta^{\diamond} \big(
 \mathbf{1}_{[t_{k'},v']}  \big) \big|_{L_{p}}
 +
 \big| \langle D x^{i-2 }_{\la(s_{j'})t_{k'}},  \mathbf{1}_{[t_{k'},v']}  \rangle_{\ch} \big|_{L_{p}}
 \notag
\\
 &\lesssim&  n^{-H}. 
  \label{e.c2bd1}
\end{eqnarray}
On the other hand, applying relation \eqref{e.abbd2} 
 with $(s,t,u,v) = ( \la(s_{j'}),t_{k'} , t_{k},v )$ we obtain: 
\begin{eqnarray}
 |\langle \mathbf{1}_{[ \la(s_{j'}),t_{k'}]} ,  
  \mathbf{1}_{[t_{k},v]} \rangle_{\ch}
|&\lesssim& n^{-2H}  ( \rho(k-k')^{2H-1}+\rho(k-ns_{j'} )^{2H-1}) . 
\label{e.c2bd4}
\end{eqnarray}
 
Substituting the estimates \eqref{e.c2bd1} and \eqref{e.c2bd4} into \eqref{e.c2bd2}   and then applying   Lemma \ref{lem.rho} (i) with $\al=2H-1$   we get:   
\begin{eqnarray}
|\langle \cc_{2}^{j'},  
  \mathbf{1}_{[t_{k},v]} \rangle_{\ch} |_{L_{p}} &\lesssim& n^{H+1/2} n^{-H-1} n^{-2H}
  \sum_{s_{j'}\leq t_{k'}<s_{j'+1}} 
  ( \rho(k-k')^{2H-1}+\rho(k-ns_{j'})^{2H-1})
  \notag\\
  &\lesssim&     n^{-1/2-2H}
   ( n^{2H}+n\rho(k-ns_{j'})^{2H-1}).
   \label{e.c2bd}
\end{eqnarray}
 Now, as in \eqref{e.fa1bd}    reporting the estimate \eqref{e.c2bd} to \eqref{e.faebd2} and applying Lemma \ref{lem.rho} (i) gives:
  \begin{eqnarray}
 \sum_{s_{j_{a}}\leq  t_{k_{a}}<s_{j_{a}+1}} \bar{\vp}_{a2}^{j_{a}  j_{a+1}} &\lesssim& n^{H}. 
\label{e.fa2bd}
\end{eqnarray}

{\bf Estimate for $\mathbf e=3.$} 
According to the defininition of $C^{j}_{3}$ in  \eqref{e.c3} we have: 
\begin{eqnarray*}
\langle \cc_{3}^{j'},  
  \mathbf{1}_{[t_{k},v]} \rangle_{\ch} = -   
    {c_{H}}{T^{2H+1}} (1-\theta^{2})z_{j'}^{2}  \int_{s_{j'}}^{s_{j'+1}} 
 \phi(\theta) 2x^{i-1}_{s_{j'}u}x^{i-2}_{s_{j'}u}   \langle  \mathbf{1}_{[s_{j'},u]}  ,  
\mathbf{1}_{[t_{k},v]}     
   \rangle_{\ch} du  .
\end{eqnarray*}
Similar to the estimate in \eqref{e.c2bd}, by applying Lemma \ref{lem.abcd}     we obtain: 
\begin{eqnarray}
\big|\langle \cc_{3}^{j'},  
  \mathbf{1}_{[t_{k},v]} \rangle_{\ch}\big|_{L_{p}}  
  &\lesssim &
  \int_{s_{j'}}^{s_{j'+1}}   |\langle  \mathbf{1}_{[s_{j'},u]}  ,  
\mathbf{1}_{[t_{k},v]}     
   \rangle_{\ch}| du
   \notag
   \\
    &\lesssim &
      \int_{s_{j'}}^{s_{j'+1}}  
 n^{-2H}[\rho( k-ns_{j'})^{2H-1}+\rho(k-nu)^{2H-1}]  du.  
 \label{e.c3bd}
\end{eqnarray}
  Substituting \eqref{e.c3bd} into \eqref{e.faebd2} and then applying Lemma \ref{lem.rho} (i) with $\al=2H-1$ we have:
  \begin{eqnarray}
 \sum_{s_{j_{a}}\leq  t_{k_{a}}<s_{j_{a}+1}} \bar{\vp}^{j_{a}  j_{a+1}}_{a3} &\lesssim&  n^{H-1/2}. 
\label{e.fa3bd}
\end{eqnarray}

{\bf Estimate for $\mathbf e=4.$} 
By applying  \eqref{e.c4} and then  Lemma \ref{lem.abcd} we have
\begin{eqnarray}
\qquad\big|
\langle \cc_{4}^{j'},  
  \mathbf{1}_{[t_{k},v]} \rangle_{\ch}\big|  \lesssim  
\big|\langle
\mathbf{1}_{[0,s_{j'+1}]},  \mathbf{1}_{[t_{k},v]} 
\rangle_{\ch}\big| 
 \lesssim  
n^{-2H } [\rho(k)^{2H-1}+\rho(k-ns_{j'+1})^{2H-1}]. 
\label{e.c4bd}
\end{eqnarray}
Summing up the right-hand side of \eqref{e.c4bd} in $k$ and applying   Lemma \ref{lem.rho} (i)   gives:
  \begin{eqnarray}
 \sum_{s_{j_{a}}\leq  t_{k_{a}}<s_{j_{a}+1}} \bar{\vp}^{j_{a}  j_{a+1}}_{a4} &\lesssim&  n^{H-1/2}.
\label{e.fa4bd}
\end{eqnarray}

{\bf Estimate for $\mathbf e=5.$} 
Recall that $\vp_{a5}$ is defined in \eqref{e.fa5}.   Similar   to \eqref{e.faebd2} we have: 
\begin{eqnarray}
|\vp_{a5}^{j_{1}  \cdots  j_{a+1}}| \lesssim n^{-1/2-H}=:\bar{\vp}^{j_{a}  j_{a+1}}_{a5} . 
\label{e.fa5bd}
\end{eqnarray}

In summary of \eqref{e.fa1bd},  \eqref{e.fa2bd}, \eqref{e.fa3bd}, \eqref{e.fa4bd}, \eqref{e.fa5bd}, we have obtained the following:
\begin{eqnarray}
 \sum_{s_{j_{a}}\leq  t_{k_{a}}<s_{j_{a}+1}} \bar{\vp}^{j_{a}  j_{a+1}}_{ae} &\lesssim&  n^{\al_{e}},
\qquad e=1,\dots,5, 
\label{e.faebd}
\end{eqnarray}
where 
\begin{eqnarray*}
\al_{1}=
\begin{cases}
0&j_{a}=j_{a+1}
\\
2H-1&j_{a}\neq j_{a+1}
\end{cases}, \quad
\al_{2}=H,\quad \al_{3}=H-1/2,
\\
 \al_{4}=H-1/2,\quad \al_{5}=1/2-H.
\end{eqnarray*}

\noindent\emph{Step 4: Estimate for $\ci_{ae}$ and $  \cz^{(n)}_{1} $.} 
Applying the relation $| \vp_{ae}^{j_{1} \cdots  j_{a+1}}|  \lesssim   \bar{\vp}_{ae}^{j_{a}  j_{a+1}}$ in \eqref{e.faebd2}  and \eqref{e.fa5bd}   to $\ci_{ae}$, $e=1,\dots,5$  in  \eqref{e.iaej} we obtain:
\begin{eqnarray}
|\ci_{ae}| &\lesssim&  
\sum_{j_{1}, \dots,j_{a+1}=0}^{N-1}
\sum_{\substack{s_{j_{1}}\leq t_{k_{1}} <s_{j_{1}+1}\\\vdots\\s_{j_{a}}\leq  t_{k_{a}}<s_{j_{a}+1}}}  
\int_{0}^{1}\big|  \prod_{c=1}^{a-1}g_{c} \big| \cdot
    \bar{\vp}^{j_{a}   j_{a+1}}_{ae}  d\theta
. 
\label{e.iaebd3}
\end{eqnarray}
Note that $\bar{\vp}_{ae}$ does not depend on $k_{1}$, \dots, $k_{a-1}$. So   \eqref{e.iaebd3} can be  rewritten as: 
\begin{eqnarray}
|\ci_{ae}|  \lesssim    
\sum_{j_{1}, \dots,j_{a+1}=0}^{N-1}\sum_{s_{j_{a}}\leq  t_{k_{a}}<s_{j_{a}+1}}  
\int_{0}^{1}
    \bar{\vp}^{j_{a} j_{a+1}}_{ae}  
\sum_{s_{j_{a-1}}\leq t_{k_{a-1}} <s_{j_{a-1}+1}}  |  g_{a-1}  |
\notag
\\
\cdots\sum_{s_{j_{1}}\leq t_{k_{1}} <s_{j_{1}+1}}  |  g_{1}  |
d\theta
. 
\label{e.iaebd4}
\end{eqnarray}

Applying Lemma \ref{lem.abcd} to \eqref{e.gc} and then invoking Lemma \ref{lem.rho} gives the estimate: 
\begin{eqnarray}
 \sum_{s_{j_{c}}\leq t_{k_{c}} <s_{j_{c}+1}}  |g_{c}|&\lesssim& n^{H+1/2} 
   \sum_{s_{j_{c}}\leq t_{k_{c}} <s_{j_{c}+1}}\int_{t_{k_{c}}}^{t_{k_{c}+1}} n^{-2H} \big( \rho(k_{c}-k_{c+1})^{2H-1}
   \notag
   \\
 &&\qquad\qquad\qquad\qquad\qquad
 \qquad  +\rho(k_{c}-ns_{j_{c+1}})^{2H-1} \big) dv_{c} 
 \notag  \\
   &\lesssim& n^{H-1/2}. 
   \label{e.iaebd}
\end{eqnarray}
Applying  \eqref{e.iaebd} to \eqref{e.iaebd4} for $c=1,\dots,a-1$ and then invoking the estimate \eqref{e.faebd}  we obtain:  
\begin{eqnarray}
|\ci_{ae} | 
&\lesssim& \sum_{j_{1}, \dots,j_{a+1}=0}^{N-1}\sum_{s_{j_{a}}\leq  t_{k_{a}}<s_{j_{a}+1}}n^{(H-1/2)(a-1)} 
\sum_{s_{j_{a}}\leq  t_{k_{a}}<s_{j_{a}+1}}
\int_{0}^{1}
    \bar{\vp}^{j_{a} , j_{a+1}}_{ae}  d\theta
    \notag
\\
&\lesssim&  \sum_{j_{1}, \dots,j_{a+1}=0}^{N-1}\sum_{s_{j_{a}}\leq  t_{k_{a}}<s_{j_{a}+1}}
n^{(H-1/2)(a-1)}
n^{\al_{e}}   . 
  \label{e.iaebd2}
\end{eqnarray}

We turn to the estimate of  $\ci_{11}^{j_{1}j_{2}} $.  As in \eqref{e.iaebd3}     applying \eqref{e.faebd2}  and  \eqref{e.faebd} to  \eqref{e.i11ij}  gives: 
\begin{eqnarray}
\ci_{11}^{j_{1}j_{2}}&\lesssim&  \sum_{s_{j_{a}}\leq  t_{k_{a}}<s_{j_{a}+1}}
n^{(H-1/2)(a-1)}
n^{\al_{e}}   . 
\label{e.iaebd5}
\end{eqnarray}

Substituting the values of $\al_{e}$  into \eqref{e.iaebd2}-\eqref{e.iaebd5} for $e=1,\dots, 5 $ we   obtain that the following six quantities   
\begin{eqnarray}
|\ci_{a1}|, 
\qquad
 |\ci_{11}^{j_{1}j_{2}}| ,~ 
j_{1}\neq j_{2}  ,
\qquad |\ci_{a2}|,
\qquad |\ci_{a3}|,
\qquad |\ci_{a4}|,
\qquad |\ci_{a5}| 
\notag
\end{eqnarray}
are respectively bounded by 
   \begin{eqnarray}
   n^{(H-1/2)(a-1)} , ~~      n^{(H-1/2)(a+1)} ,  ~~
      n^{(H-1/2)(a-1)+H} ,
   \notag   \\
     n^{(H-1/2)a}  ,~~
     n^{(H-1/2)a} ,  ~~  
n^{(H-1/2)(a-2)}  .
\label{e.ibd}
\end{eqnarray}
Since $H-1/2<0$ these estimates imply  the following convergences: 
\begin{eqnarray}
\lim_{n\to\infty}\ci_{a1}^{j_{1}j_{2}}  = 0 
\qquad\qquad
&&\text{when $a\geq 2$ or when $a\geq 1$ and $j_{1}\neq j_{2}$}
\label{e.i1c}
\\
\lim_{n\to\infty}\ci_{a2} = 0 
\qquad\qquad
&&\text{when $a>  \frac{H}{1/2-H}+1$}
\label{e.i2c}
\\
\lim_{n\to\infty}\ci_{a3} =\lim_{n\to\infty}\ci_{a4} = 0 
\qquad\qquad
&&\text{when $a\geq  1$}
\label{e.i3c}
.
\end{eqnarray}
Furthermore, we have $
\lim_{n\to\infty}\ci_{a5} = 0$ 
   for  all  $a\geq  3 $.  
In the following we show that the convergence of $\ci_{a5}$ also holds for $a=2$, and therefore we have
\begin{eqnarray}
\lim_{n\to\infty}  \ci_{a5}  =0
\qquad \text{when $a\geq 2$.}
\label{e.ia5}
\end{eqnarray}  

Recall that $\ci_{25}$ and $\vp_{25}$ are defined in \eqref{e.iaej} and \eqref{e.fa5}. So we have 
\begin{eqnarray}
\ci_{25} &=& \sum_{j_{1}, j_{2}=0}^{N-1} \sum_{t_{k_{1}}, t_{k_{2}}}g_{1} \int_{0}^{1} \vp^{j_{1}j_{2}}_{25} d\theta , 
\label{e.i25}
\end{eqnarray}
where
\begin{eqnarray}
\vp^{j_{1}j_{2}}_{25}&=&  \mathbf{i}\theta   z_{j_{1}}n^{H+1/2} \int_{t_{k_{1}}}^{t_{k_{1}+1}} \mE  \big[
 \langle D {G}_{1}  ,  
  \mathbf{1}_{[t_{k_{1}},v_{1}]} \rangle_{\ch}x^{i-2 }_{\la(s_{j_{1}})t_{k_{1}}}\phi(\theta)
 \big]
      dv_{1}. 
      \label{e.f25}
\end{eqnarray}
Since  $D\cg_{1}=D  x^{i-1 }_{\la(s_{j_{1}})t_{k_{1}}}  =  x^{i-2 }_{\la(s_{j_{1}})t_{k_{1}}} \mathbf{1}_{[\la(s_{j_{1}}), t_{k_{1}}]}$, applying Lemma \ref{lem.abcd}     gives   the   estimate: 
\begin{eqnarray*}
\big|
\mE  [ \langle D\cg_{1},  
  \mathbf{1}_{[t_{k_{2}},v_{2}]} \rangle_{\ch}x^{i-2 }_{\la(s_{j_{2}})t_{k_{2}}}   \phi(\theta)] 
\big|&\lesssim& \langle \mathbf{1}_{[\la(s_{j_{1}}),t_{k_{1}}]} , \mathbf{1}_{[t_{k_{2}},v_{2}]} \rangle_{\ch}
\\
&\lesssim& n^{-2H}[\rho(k_{2}-k_{1})^{2H-1}+\rho(k_{2}-ns_{j_{1}})^{2H-1}]. 
\end{eqnarray*}
Substituting this estimate into \eqref{e.f25} gives:
\begin{eqnarray}
|\vp^{j_{1}j_{2}}_{25}|&\lesssim& n^{H-1/2}
n^{-2H}[\rho(k_{2}-k_{1})^{2H-1}+\rho(k_{2}-ns_{j_{1}})^{2H-1}]. 
\label{e.f25bd}
\end{eqnarray}
 Applying relation \eqref{e.f25bd} to \eqref{e.i25} and    bounding $g_{1}$   in the similar way as in  \eqref{e.iaebd}  we have  
\begin{eqnarray} 
|\ci_{25} | &\lesssim& \sum_{j_{1}, j_{2}=0}^{N-1}\sum_{\substack{s_{j_{1}}\leq t_{k_{1}} <s_{j_{1}+1}\\s_{j_{2}}\leq  t_{k_{2}}<s_{j_{2}+1}}}  n^{H-1/2}
n^{-2H}[\rho(k_{2}-k_{1})^{2H-1}+\rho(k_{2}-ns_{j_{1}})^{2H-1}]
\notag
\\
&&   \qquad\qquad\qquad  
\cdot
n^{-1/2-H}
[\rho(k_{2}-k_{1})^{2H-1}+\rho(k_{1}-ns_{j_{2}})^{2H-1}]
.
\label{e.i25bd1}
\end{eqnarray}
By applying Lemma \ref{lem.rho} (i) to the right-hand side of  \eqref{e.i25bd1} it is easy to see that 
\begin{eqnarray*}
|\ci_{25}| \lesssim  n^{-2H-1} (n^{4H }+n\log n)
 \leq n^{2H-1}\vee (n^{-2H}\log n)   
\end{eqnarray*}
for all $H<1/2$. This implies that $\lim_{n\to\infty}\ci_{25}\to0$. We thus conclude \eqref{e.ia5}.

Sending $n\to\infty$ in equation \eqref{e.czn1} and invoking 
   relations \eqref{e.i1c}-\eqref{e.i3c} and \eqref{e.ia5}  gives:
\begin{eqnarray}
  \lim_{n\to\infty}\cz^{(n)}_{1} =   \lim_{n\to\infty}   \sum_{j=0}^{N-1}\ci_{11}^{jj} . 
  \label{e.czn1bd}
\end{eqnarray}

\noindent\emph{Step 5: A decomposition  of $\ci_{11}^{jj}$.}
According to the definition of  $\ci_{11}^{jj}$   in \eqref{e.i11ij} we have
\begin{eqnarray}
\ci_{11}^{jj}&=& -n^{2H+1}z_{j}^{2}\int_{0}^{1}\mE     \phi(\theta)\theta \sum_{s_{j}\leq t_{k},t_{k'}<s_{j+1}}\int_{t_{k}}^{t_{k+1}} \int_{t_{k'}}^{t_{k'+1}}
   x^{i-1 }_{\la(s_{j})t_{k'}}  x^{i-1 }_{\la(s_{j})t_{k}}
\notag
\\
&&\qquad\qquad\qquad\qquad\qquad\qquad\qquad\qquad\qquad\cdot\langle \mathbf{1}_{[t_{k'},v']}      ,  
  \mathbf{1}_{[t_{k},v]}     
   \rangle_{\ch} dv' dv  d\theta  .
   \label{e.i11j}
\end{eqnarray}
Note that by the self-similarity and increment stationarity of fBm we have \begin{eqnarray*}
\langle \mathbf{1}_{[t_{k'},v']}      ,  
  \mathbf{1}_{[t_{k},v]}     
   \rangle_{\ch}   &=& (T/n)^{2H}\langle \mathbf{1}_{[ {k'},nv'/T]}      ,  
  \mathbf{1}_{[{k},nv/T]}     
   \rangle_{\ch}
  \\
  &=& (T/n)^{2H}\langle \mathbf{1}_{[ {k'}-k,nv'/T-k]}      ,  
  \mathbf{1}_{[{0},nv/T-k]}     
   \rangle_{\ch}.
\end{eqnarray*}
Applying this relation and     then     a change of variables  yields:  
\begin{eqnarray}
    \int_{t_{k}}^{t_{k+1}} \int_{t_{k'}}^{t_{k'+1}}
  \langle \mathbf{1}_{[t_{k'},v']}      ,  
  \mathbf{1}_{[t_{k},v]}     
   \rangle_{\ch} dv' dv  
         &=& (T/n)^{2H+2} \int_{ {0}}^{ {1}} \int_{ {k'-k}}^{ {k'-k+1}}
  \langle \mathbf{1}_{[ {k'}-k,v']}      ,  
  \mathbf{1}_{[ 0,v]}     
   \rangle_{\ch} dv' dv 
 \notag  \\
       &=&  \mu(k'-k) (T/n)^{2H+2} , 
      \label{e.rho2}
\end{eqnarray}
where recall that $\mu$ is defined in  \eqref{e.rho}.
  Substituting \eqref{e.rho2} into \eqref{e.i11j} we obtain 
\begin{eqnarray*}
\ci_{11}^{jj} &=& -n^{-1}T^{2H+2}   z_{j}^{2}\int_{0}^{1}\mE     \phi(\theta)\theta \sum_{s_{j}\leq t_{k} , t_{k'}<s_{j+1}}  x^{i-1 }_{\la(s_{j})t_{k'}}   
x^{i-1 }_{\la(s_{j})t_{k}}  
\mu(k'-k) d\theta 
\\
&=& -   z_{j}^{2} 
\int_{0}^{1}\mE   \wt{Z}^{n,j}   \phi(\theta)  \theta 
d\theta 
,  
\end{eqnarray*}
where
\begin{eqnarray*}
\wt{Z}^{n,j} &=& n^{-1}T^{2H+2} \sum_{s_{j}\leq t_{k} , t_{k'}<s_{j+1}}  x^{i-1 }_{\la(s_{j})t_{k'}}   
x^{i-1 }_{\la(s_{j})t_{k}}  
\mu(k'-k) .
\end{eqnarray*}

\noindent\emph{Step 6: Estimate for $\cz^{(n)}_{2}+\sum_{j=0}^{N-1} \ci_{11}^{jj}$.}
Recall that $\cz^{(n)}_{2}$ is defined in \eqref{e.i1i2}. So we have
\begin{eqnarray}
\cz^{(n)}_{2}+\sum_{j=0}^{N-1} \ci_{11}^{jj}&=&\sum_{j=0}^{N-1}       z_{j}^{2} 
\int_{0}^{1}\mE  \big[(F_{s_{j}s_{j+1}}-\wt{Z}^{n,j} )   \phi(\theta)\big]\theta 
d\theta
\notag
\\
 &\lesssim& \sum_{j=0}^{N-1}   |F_{s_{j}s_{j+1}}-\wt{Z}^{n,j} |_{L_{1}}. 
\label{e.fz2}
\end{eqnarray}
In the following we show that: 
\begin{eqnarray}
 |F_{s_{j}s_{j+1}}-\wt{Z}^{n,j} |_{L_{1}}\to0
 \qquad
 \text{ as $n\to\infty$ for $j=0,\dots, N-1$.}
 \label{e.fzc}  
\end{eqnarray}

 Consider the following decomposition: 
 \begin{eqnarray}
F_{s_{j}s_{j+1}}-\wt{Z}^{n,j}&=& \wt{Z}^{n,j}_{1}+ \wt{Z}^{n,j}_{2}+ \wt{Z}^{n,j}_{3},
\label{e.fz}
\end{eqnarray}
where
\begin{eqnarray}
 \wt{Z}^{n,j}_{1}&=&  {c_{H}}{T^{2H+1}}   \Big(  \int_{s_{j}}^{s_{j+1}} (x^{i-1 }_{s_{j}u})^{2} du-\frac{T}{n} \sum_{s_{j}\leq t_{k} <s_{j+1}}(x^{i-1 }_{s_{j} t_{k}})^{2} \Big) 
 \notag
 \\
  \wt{Z}^{n,j}_{2}&=& \frac1n  T^{2H+2} \sum_{s_{j}\leq t_{k} <s_{j+1}}(x^{i-1 }_{s_{j} t_{k}})^{2} \Big(  c_{H}-
  \sum_{s_{j}\leq  t_{k'}<s_{j+1}}\mu(k'-k) 
  \Big)
  \label{e.zt2}
   \\
  \wt{Z}^{n,j}_{3}&=& \frac1n  T^{2H+2} \sum_{s_{j}\leq t_{k}, t_{k'} <s_{j+1}}\big[    (x^{i-1 }_{s_{j} t_{k}})^{2}-x^{i-1 }_{\la(s_{j}) t_{k'}} x^{i-1 }_{\la(s_{j}) t_{k}}    \big]
 \mu(k'-k) .
   \label{e.zt3}
\end{eqnarray}

It is easy to see that we have the following two estimates:
\begin{eqnarray*}
\big|x^{i-1 }_{\la(s_{j})t_{k'}}x^{i-1 }_{\la(s_{j})t_{k}} -(x^{i-1 }_{ s_{j} t_{k}})^{2}   \big|_{L_{p}} \lesssim n^{-H} \rho(k-k')^{H}
\quad \text{and}\quad
|\mu(k'-k) |\lesssim \rho(k'-k)^{2H-2}. 
\end{eqnarray*}
Substituting these two estimates into \eqref{e.zt3}   gives: 
\begin{eqnarray*}
|\widetilde{Z}_{3}^{n,j}|_{L_{p}}
&\lesssim&  n^{-1} \sum_{s_{j}\leq t_{k},t_{k'}<s_{j+1}}
n^{-H} \rho(k-k')^{H} \cdot  \rho(k-k')^{2H-2}
\\
&\lesssim&
n^{-H-1} \sum_{s_{j}\leq t_{k},t_{k'}<s_{j+1}} \rho(k-k')^{3H-2}
. 
\end{eqnarray*}
Applying Lemma \ref{lem.rho} with $\al=3H-2$ gives:   
\begin{eqnarray*}
|\widetilde{Z}_{3}^{n,j}|_{L_{p}}&\lesssim& n^{-H-1} (n^{3H}+ n\log n)  
\end{eqnarray*}
for all $H<1/2$. 
In particular, we have the convergence   $|\widetilde{Z}_{3}^{n,j}|_{L_{p}} \to 0$ as $n\to\infty$. 

We turn to the estimate of $\widetilde{Z}_{2}^{n,j}$. 
Note that 
\begin{eqnarray}
\sum_{s_{j}\leq  t_{k'}<s_{j+1}}\mu(k'-k)  - c_{H}&=&\sum_{s_{j}\leq  t_{k'}<s_{j+1}}\mu(k'-k)  - \sum_{k'\in\ZZ}\mu(k'-k) 
\notag
\\
&=& \sum_{   t_{k'}<s_{j}}\mu(k'-k)+\sum_{   t_{k'}\geq s_{j+1}}\mu(k'-k)
. 
\label{e.rhos}
\end{eqnarray}
Since $\mu(k)\lesssim \rho(k)^{2H-2}$, applying Lemma \ref{lem.rho} (i) to \eqref{e.rhos}   we obtain the estimate:  
\begin{eqnarray*}
\big|\sum_{s_{j}\leq  t_{k'}<s_{j+1}}\mu(k'-k)  - c_{H}\big|&\lesssim& \rho(k-ns_{j})^{2H-1}+ \rho(ns_{j+1}-k)^{2H-1}   
\end{eqnarray*}
for $s_{j}\leq t_{k}<s_{j+1}$. 
Substituting this estimate into \eqref{e.zt2} for $ \wt{Z}_{2}^{n,j} $ and then applying Lemma \ref{lem.rho} we obtain the estimate: 
\begin{eqnarray*}
|\wt{Z}_{p}^{n,j}|_{L_{p}} \lesssim n^{-1}  \sum_{s_{j}\leq t_{k} <s_{j+1}}\lc \rho(k-ns_{j})^{2H-1}+ \rho(ns_{j+1}-k)^{2H-1} \rc \lesssim (1/n)^{1-2H}. 
\end{eqnarray*}
  It follows that $|\wt{Z}_{2}^{n,j}|_{L_{p}}\to0$ as $n\to\infty$. 
It is clear  that   we also have $\lim_{n\to\infty}| \wt{Z}^{n,j}_{1}|_{L_{p}}= 0$. 

In summary, we have shown that the right-hand side of \eqref{e.fz}  converges in $L_{p}$ to zero, and therefore the convergence in \eqref{e.fzc} holds. Invoking the relation \eqref{e.fz2} we thus have:    
\begin{eqnarray}
\cz^{(n)}_{2}+\sum_{j=0}^{N-1} \ci_{11}^{jj} \to 0
 \qquad \text{as $n\to\infty$.}
 \label{e.czz}
\end{eqnarray} 
  
%

%

   \noindent\emph{Step 7: Conclusion.} It follows from \eqref{e.czn2} and \eqref{e.czn1bd}   that we have
   \begin{eqnarray*}
\lim_{n\to\infty}\cz^{(n)}= \lim_{n\to\infty}(\cz^{(n)}_{1}  + \cz^{(n)}_{2})
= \lim_{n\to\infty}(\sum_{j=0}^{N-1}\ci_{11}^{jj}  + \cz^{(n)}_{2} ).
\end{eqnarray*}
Then by \eqref{e.czz} we obtain $\cz^{(n)}\to0$ as $n\to\infty$. This concludes the convergence \eqref{e.zpic}. 
\end{proof}


\subsection{Compensated weighted  sums}\label{section.lt}

Take $\ell\in\NN$. 
Let $(z,z', \dots,z^{(\ell-1)})$ be a continuous process controlled by $(x,\ell,H)$ almost surely (see Definition \ref{def.control}). 
In this section we establish   limit theorem for the following compensated weighted sum:
\begin{eqnarray}\label{e.ce}
\ce^{z, n}_{\ell}(s,t) &:=& \sum_{i=1}^{\ell}\cj_{s}^{t} (z^{(i-1)},h^{i}), 
\qquad (s,t)\in\cs_{2}([0,T]) ,
\end{eqnarray}
where 
\begin{eqnarray}
h^{i}_{st} &=&  \sum_{s\leq t_{k}<t} \int_{t_{k}}^{t_{k+1}} x^{i}_{t_{k}u} du,     
\label{e.hi}
\end{eqnarray}
     $x^{i}_{st} $ stands for $ (\delta x_{st})^{i}/i!$, and   the integral operator $\cj_{s}^{t}$ is defined in \eqref{e.jfg2}. 
As a first main step, we   consider a special case of \eqref{e.ce} when $z$ is a monomial of $x$. Namely, we define  
  \begin{eqnarray}
\ce_{L}^{x}(s,t) &:=& \sum_{i=1}^{L}\cj_{s}^{t} (x^{L-i},h^{i})   
\label{e.cex}
\end{eqnarray}
for $  (s,t)\in\cs_{2}(\ll0,T\rr)$ and $ L\in\NN$, where   $\cj_{s}^{t}$ is defined in \eqref{e.jfg}. 

\begin{lemma}\label{lem.cex}
Let $\ce^{x}_{L}(s,t)$, $  (s,t)\in\cs_{2}(\ll0,T\rr)$,  $ L\in\NN$ be defined in \eqref{e.cex}. Then 

\noindent (i) There exists a constant $K>0$ such that 
  the following estimate holds: 
\begin{eqnarray}
|\ce_{L}^{x}(s,t) |_{L_{p}}&\leq & K(1/n)^{H+1/2}(t-s)^{1/2+(L-1)H}   
\label{e.cexbdl2}
\end{eqnarray} 
for all $  (s,t)\in\cs_{2}(\ll0,T\rr)$, $n\in\NN$ and     $p\geq 1$.

\noindent(ii)  Let 
  $Z^{(n), L}_{st} $ be defined in \eqref{e.zn}. Then   we   have the convergence: 
\begin{eqnarray}
\lim_{n\to\infty}\sup_{ (s,t)\in\cs_{2}(\ll0,T\rr)}\big|n^{H+1/2}\ce_{L}^{x}(s,t) -  
Z^{(n), L}_{st}\big|_{L_{p}} = 0   
\label{e.cexbd5}
\end{eqnarray}
\end{lemma}

\begin{proof}
The   proof  is divided  into several steps.   

\noindent\emph{Step 1: Decomposition of $\ce_{L}^{x}(s,t) $.}
Let  $X_{k} = \delta x_{t_{k}u}(u-t_{k})^{-H}$, $k=0,\dots,n-1$. It is clear that $X_{k}$ is a standard normal random variable, and   we can write: 
\begin{eqnarray}
x^{i}_{t_{k}u} =(\delta x_{t_{k}u})^{i}/i!= X_{k}^{i}\cdot (u-t_{k})^{iH}/{i!}. 
\label{e.xk}
\end{eqnarray}
 Substituting \eqref{e.xk} into \eqref{e.cex}  we get:  
\begin{eqnarray}
\ce_{L}^{x} (s,t)&=&\sum_{i=1}^{L}\sum_{s\leq t_{k}<t} x_{st_{k}}^{ L-i} \int_{t_{k}}^{t_{k+1}}      \frac{1}{i!} X_{k}^{i}\cdot(u-t_{k})^{iH} du  .
\label{e.ce1}
\end{eqnarray}
Recall that for  monomials we have    the following Hermite decomposition: 
\begin{eqnarray}
X^{i}_{k} =  \sum_{q=0}^{[\frac{i}{2}]}a^{i}_{i-2q} {H_{i-2q}(X_{k})}  \,, 
\qquad
a^{i}_{i-2q} = \frac{i!}{2^{q}q!(i-2q)!}\,,   
\qquad 
i\in\NN . 
\label{e.xhermite}
\end{eqnarray}
 Substituting \eqref{e.xhermite} into   \eqref{e.ce1} and then taking into account the formula   of $a^{i}_{i-2q} $ we obtain: 
\begin{eqnarray}
\ce_{L}^{x}(s,t)&=& \sum_{i=1}^{L} \sum_{s\leq t_{k}<t}x_{st_{k}}^{ L-i}   \int_{t_{k}}^{t_{k+1}}   \frac{1}{i!} \cdot\sum_{q=0}^{[\frac{i}{2}]}a^{i}_{i-2q} H_{i-2q}(X_{k})
\cdot (u-t_{k})^{iH} du 
\notag
\\
&=&\sum_{i=1}^{L} \sum_{s\leq t_{k}<t}  \int_{t_{k}}^{t_{k+1}}        \sum_{q=0}^{[\frac{i}{2}]}
\frac{1}{2^{q}q!(i-2q)!}
\lc x_{st_{k}}^{ L-i} H_{i-2q}(X_{k}))  \rc  (u-t_{k})^{iH} du  . 
 \label{e.ce2} 
\end{eqnarray}

Now let us derive a decomposition of the quantity $x_{st_{k}}^{ L-i} H_{i-2q}(X_{k}) $ in \eqref{e.ce2}. 
By the definition of $X_{k}$ it is clear that we have the relation $X_{k} = \delta^{\diamond} ( \mathbf{1}_{[t_{k},u]} (u-t_{k})^{-H} )$. According to relation \eqref{e.dkh} we can thus write: 
\begin{eqnarray*}
H_{i-2q}(X_{k}) =\delta^{{\diamond},i-2q} \big(( \mathbf{1}_{[t_{k},u]} (u-t_{k})^{-H})^{\otimes (i-2q)}\big)
=\delta^{{\diamond},i-2q}( \mathbf{1}_{[t_{k},u]}^{\otimes (i-2q)} )  (u-t_{k})^{-(i-2q)H}. 
\end{eqnarray*}
It follows that:  
\begin{eqnarray}
x_{st_{k}}^{ L-i} {H_{i-2q}(X_{k})}  &=&  
x_{st_{k}}^{L-i}\cdot \delta^{{\diamond},i-2q}( \mathbf{1}_{[t_{k},u]}^{\otimes (i-2q)} )  (u-t_{k})^{-(i-2q)H}. 
\label{e.xhq}
\end{eqnarray}
Applying the identity $F\delta^{\diamond} (u)=\delta^{\diamond} (Fu)+\langle DF,u \rangle_{\ch}$ to the right-hand side of \eqref{e.xhq} several times we obtain   the following relation: 
\begin{eqnarray}
x_{st_{k}}^{ L-i} {H_{i-2q}(X_{k})}  =
 (u-t_{k})^{-(i-2q)H} \sum_{j=0}^{(i-2q)\wedge(L-i)}  
\binom{i-2q}{j}
\delta^{{\diamond},i-2q-j}\big( x^{ L-i-j}_{st_{k}} 
\mathbf{1}_{[t_{k},u]}^{\otimes (i-2q-j)}
\big)
\notag
\\
\cdot \langle \mathbf{1}_{[s,t_{k}]} ,\mathbf{1}_{[t_{k},u]} \rangle_{\ch}^{j}.
\label{e.xhq2}
\end{eqnarray}

Now we come back to the expression of $\ce_{L}^{x}(s,t)$ in \eqref{e.ce2}. 
Substituting \eqref{e.xhq2} into \eqref{e.ce2} we obtain: 
\begin{eqnarray}
\ce_{L}^{x}(s,t) = \sum_{i=1}^{L} \sum_{q=0}^{[\frac{i}{2}]} \sum_{j=0}^{(i-2q)\wedge(L-i)}\ce_{Liqj}^{x} (s,t),
\label{e.cex2}
\end{eqnarray}
where
\begin{eqnarray}
\ce_{Liqj}^{x}(s,t) &=& \sum_{s\leq t_{k}<t}  \int_{t_{k}}^{t_{k+1}}        
\frac{\binom{i-2q}{j}}{2^{q}q!(i-2q)!}
(u-t_{k})^{2qH} 
\notag
\\
&&
\qquad\qquad\qquad\qquad
\cdot\delta^{{\diamond},i-2q-j}\lp x^{ L-i-j}_{st_{k}} 
\mathbf{1}_{[t_{k},u]}^{\otimes (i-2q-j)}
\rp \langle \mathbf{1}_{[s,t_{k}]} ,\mathbf{1}_{[t_{k},u]} \rangle_{\ch}^{j}
   du. 
   \label{e.ceij}
\end{eqnarray}

Define  the following      index sets for $i,q,j$: 
\begin{eqnarray}
A &=& \{(i,q,j): 1\leq i\leq L,\, 0\leq q \leq [i/2],  \, j \leq L-i,\, j = i-2q\}, 
\label{e.aset}
\\
A' &=& \{(i,q,j): 1\leq i\leq L, 0\leq q \leq [i/2],   j   \leq L-i, j < i-2q\}. 
\label{e.aset2}
\end{eqnarray} 
Then we can write the summation in \eqref{e.cex2} as: 
\begin{eqnarray}
\ce_{L}^{x}(s,t) &=& \sum_{(i,q,j)\in A\cup A'} \ce_{Liqj}^{x} (s,t).   
\label{e.cex3}
\end{eqnarray}

\noindent\emph{Step 2: Illustration of the cancellations in $\ce_{L}^{x}(s,t) $.}
In Section \ref{section.intro} we have mentioned that there are   cancellations within the     compensated sum $\ce^{x}_{L}(s,t)$. With the decomposition \eqref{e.cex2} in hand let us now be more specific about them: these cancellations  occur among the components $\ce_{Liqj}^{x} (s,t)$   in \eqref{e.cex3} for which $(i,q,j)\in A$.  

The following table contains the components $\ce_{Liqj}^{x} (s,t)$, $(i,q,j)\in A$ in \eqref{e.cex3}. As an example we have taken  $L=6$. 
\[ 
\begin{array}{|c|c|c|c|c|c|c|}
\hline  &i=1  & i=2  &i=3  &i=4&i=5 &i=6   \\ \hline   j=6&&   &  &&   &\ce^{x}_{L606}    \\ \hline j=5&   &   &  &   &\ce^{x}_{L505} & 
 \\ \hline  j=4& &   &     &\ce^{x}_{L404} & &\ce^{x}_{L614}
 \\ \hline j=3&  &   &    \ce^{x}_{L303} &  &\ce^{x}_{L513} & 
 \\
 \hline j=2&   & \ce^{x}_{L202}   &     & \ce^{x}_{L412} &  &\ce^{x}_{L622}
 \\
 \hline j=1&
  \ce^{x}_{L101} &    &   \ce^{x}_{L311}& &\ce^{x}_{L521} &  \\
 \hline j=0&  &  \ce^{x}_{L210} &    & \ce^{x}_{L420} & &\ce^{x}_{L630}
 \\\hline
\end{array} 
\]
We will see (in   Step 5)  that there are three cancellations in the above table. 
More precisely, we will show the following relations: 
\begin{eqnarray}
&&n^{H+1/2}(\ce^{x}_{L101}+\ce^{x}_{L210})= o(1), 
\qquad
n^{H+1/2}(\ce^{x}_{L202}+\ce^{x}_{L311}+\ce^{x}_{L420}) = o(1)
\notag
\\
&& n^{H+1/2}(\ce^{x}_{L303}+\ce^{x}_{L412}+\ce^{x}_{L521}+\ce^{x}_{L630})= o(1). 
\notag
\end{eqnarray}

\noindent\emph{Step 3: Second decomposition of $\ce_{L}^{x}(s,t) $.}
Let us make the following decomposition of relation \eqref{e.cex3}:
\begin{eqnarray}
\ce_{L}^{x}(s,t) = \ce_{L}^{x,1}(s,t)+\ce_{L}^{x,2}(s,t)+\ce_{L}^{x,3}(s,t) ,  
\label{e.cexd}
\end{eqnarray}
where
\begin{eqnarray}
&& \ce_{L}^{x,1}(s,t)=\sum_{(i,q,j)\in A}\ce_{Liqj}^{x} (s,t),
\qquad\quad
\ce_{L}^{x,2}(s,t)=
\sum_{(i,q,j)\in A'\cap \{i\geq 2\}}\ce_{Liqj}^{x} (s,t). 
\label{e.cexa}
\\
&&\ce_{L}^{x,3}(s,t) =\sum_{(i,q,j)\in  A'\cap \{i=1\}}\ce_{Liqj}^{x} (s,t)
. 
\label{e.cexc}
\end{eqnarray}

In the following steps we derive the estimate for each  $ \ce_{L}^{x,i}(s,t) $,  $i=1,2,3$, which together will give the desired relations \eqref{e.cexbdl2}-\eqref{e.cexbd5} for $ \ce_{L}^{x}(s,t) $. More specifically,  Step 4-7 considers     the cancellation and  estimate for $ \ce_{L}^{x,1}(s,t) $. Step 8-11 derives the estimate for $ \ce_{L}^{x,2}(s,t) $, and Step 12 is for $ \ce_{L}^{x,3}(s,t) $. Finally, in Step 13 we report these estimates to \eqref{e.cexd}. 

\noindent\emph{Step 4: The decomposition  of $\ce_{Liqj}^{x}(s,t) $ when $(i,q,j)\in A$.}
Suppose that  $i$, $q$ and $j$ are such that   $j=i-2q$, or equivalently,    $i+j=2(j+q)$. In this case we have 
\begin{eqnarray*}
\frac{\binom{i-2q}{j}}{2^{q}q!(i-2q)!} = \frac{1}{2^{q}q!j!}
\qquad
\text{and}
\qquad
\delta^{{\diamond},i-2q-j}\big( x^{ L-i-j}_{st_{k}} 
\mathbf{1}_{[t_{k},u]}^{\otimes (i-2q-j)}
\big) = x^{ L-i-j}_{st_{k}} , 
\end{eqnarray*}
and therefore  \eqref{e.ceij}  becomes: 
\begin{eqnarray}
\ce_{Liqj}^{x}(s,t) &=& \sum_{s\leq t_{k}<t}   x^{L-i-j}_{st_{k}} \int_{t_{k}}^{t_{k+1}}        
\frac{1}{2^{q}q!j!}
(u-t_{k})^{2qH} 
  \langle \mathbf{1}_{[s,t_{k}]} ,\mathbf{1}_{[t_{k},u]} \rangle_{\ch}^{j}
   du
  \notag
 \\
 &=& \wt{\ce}_{Liqj}^{x}(s,t)  +r^{ijq},
    \label{e.etaijq}
\end{eqnarray}
where 
\begin{eqnarray}
\qquad \wt{\ce}_{Liqj}^{x}(s,t) &=& \sum_{s\leq t_{k}<t}   x^{L-i-j}_{st_{k}} \int_{t_{k}}^{t_{k+1}}        
\frac{1}{2^{q}q!j!}
(u-t_{k})^{2qH} 
  \langle \mathbf{1}_{(-\infty,t_{k}]} ,\mathbf{1}_{[t_{k},u]} \rangle_{\ch}^{j} 
   du 
   \label{e.ceb}
\\
r^{ijq} &=&  \sum_{s\leq t_{k}<t}   x^{ L-i-j}_{st_{k}} \int_{t_{k}}^{t_{k+1}}        
\frac{1}{2^{q}q!j!}
(u-t_{k})^{2qH} 
\notag
 \\
 &&\quad\qquad\qquad\qquad\qquad
\cdot \lp \langle \mathbf{1}_{[s,t_{k}]} ,\mathbf{1}_{[t_{k},u]} \rangle_{\ch}^{j}-\langle \mathbf{1}_{(-\infty,t_{k}]} ,\mathbf{1}_{[t_{k},u]} \rangle_{\ch}^{j}\rp
   du.
   \label{e.rijq}
\end{eqnarray}
Note that 
\begin{eqnarray}
\langle \mathbf{1}_{(-\infty,t_{k}]} ,\mathbf{1}_{[t_{k},u]} \rangle_{\ch}=-\frac12 \cdot (u-t_{k})^{2H} . 
\label{e.half}
\end{eqnarray}
Substituting \eqref{e.half} into \eqref{e.ceb} and taking into account that $2q+2j=i+j$ we obtain: 
 \begin{eqnarray}
   \wt{\ce}_{Liqj}^{x}(s,t)  &=& \sum_{s\leq t_{k}<t}   x^{ L-i-j}_{st_{k}} \int_{t_{k}}^{t_{k+1}}        
\frac{1}{2^{q}q!j!}
(u-t_{k})^{ (i+j)H} 
 (-1/2)^{j}
   du  ,
   \label{e.ceb2}
\end{eqnarray}

\noindent\emph{Step 5: The cancellation of $\wt{\ce}_{Liqj}^{x}(s,t) $ for $(i,q,j)\in A$.}
Define the index set:
\begin{eqnarray*}
A_{\tau}&=& \{  (\tau,0,\tau),(\tau+1,1,\tau-1),\dots,(2\tau,\tau,0)\} 
\end{eqnarray*}
for   $\tau\in\NN$. Recall that the set $A$ is defined in \eqref{e.aset}.
 It is easy to see that   we have the inclusion relation $A_{\tau}\subset A$   when      $\tau=1,\dots,[L/2]$. 

Recall that we have the relation    \eqref{e.ceb2} for $\wt{\ce}_{Liqj}^{x}(s,t)  $. Summing up \eqref{e.ceb2} we have: 
\begin{eqnarray}
\sum_{(i,q,j)\in A_{\tau}} \wt{\ce}_{Liqj}^{x}(s,t)  = \sum_{(i,q,j)\in A_{\tau}} \sum_{s\leq t_{k}<t}   x^{ L-i-j}_{st_{k}} \int_{t_{k}}^{t_{k+1}}        
\frac{1}{2^{q}q!j!}
(u-t_{k})^{ (i+j)H} 
\notag\\
\cdot \big(-\frac12\big)^{j} 
   du. 
   \label{e.ceb3}
\end{eqnarray}
Note that for  $(i,q,j)\in A_{\tau}$ we have     $i+j=2\tau$ and $j+q=\tau$. So equation \eqref{e.ceb3} gives 
\begin{eqnarray}
  \sum_{(i,q,j)\in A_{\tau}} \wt{\ce}_{Liqj}^{x}(s,t)    &=& \sum_{s\leq t_{k}<t}   x^{ L-2\tau}_{st_{k}} \int_{t_{k}}^{t_{k+1}}   2^{-\tau}\Big(     
\sum_{(i,q,j)\in A_{\tau}} \frac{(-1)^{j}}{q!j!}\Big)
(u-t_{k})^{ 2\tau H} 
   du
\notag\\
   &=&0. 
   \label{e.cebz}
\end{eqnarray}

Now let 
$\wt{A}_{\tau}=\{(i,q,j)\in A:q+j=\tau \}  $.
It is easy to see that we have  $A_{\tau}= \wt{A}_{\tau}$ for     $\tau=1,\dots,[L/2]$, and $\wt{A}_{\tau}=\emptyset$ for $\tau> [L/2]$. It follows that  
\begin{eqnarray}\label{e.aatau}
A = \bigcup_{\tau=1}^{\infty}\wt{A}_{\tau}
= \bigcup_{\tau=1}^{[L/2]}\wt{A}_{\tau}= {\bigcup_{\tau=1}^{[L/2]}} A_{\tau}. 
\end{eqnarray}
 Applying  \eqref{e.aatau} and then  the relation  \eqref{e.cebz} we obtain:    
\begin{eqnarray}
\sum_{(i,q,j)\in A}   \wt{\ce}_{Liqj}^{x}(s,t) =   \sum_{\tau=1}^{[L/2]}\sum_{(i,q,j)\in A_{\tau}}    \wt{\ce}_{Liqj}^{x}(s,t) =0.
\label{e.cexbcancel}
\end{eqnarray}

\noindent\emph{Step 6: The estimate of $r^{ijq}$ when $(i,q,j)\in A$.}
Recall that $r^{ijq}$ is defined in \eqref{e.rijq}. 
By applying  Lemma \ref{lem.abcd} we have  the following estimate: 
\begin{eqnarray}
|\langle \mathbf{1}_{[s,t_{k}]} ,\mathbf{1}_{[t_{k},u]} \rangle_{\ch}^{j}-\langle \mathbf{1}_{(-\infty,t_{k}]} ,\mathbf{1}_{[t_{k},u]} \rangle_{\ch}^{j}| &\lesssim& (u-t_{k})^{2jH} \rho( k-ns )^{2H-1}
\notag
\\
&\leq& (1/n)^{2jH} \rho( k-ns )^{2H-1}.
\label{e.rijqbd1}
\end{eqnarray}
Applying   \eqref{e.rijqbd1} to \eqref{e.rijq} and then applying Lemma \ref{lem.rho} (i) with $\al=2H-1>-1$
gives 
\begin{eqnarray}
|r^{ijq}|_{L_{p}}&\lesssim&\sum_{s\leq t_{k}<t} (1/n)^{2(q+j)H+1} \rho(k-ns)^{2H-1} (t-s)^{(L-i-j)H} 
\notag\\
&\lesssim&  (1/n)^{2(q+j)H+1} (n(t-s))^{2H} (t-s)^{(L-i-j)H} 
\notag\\
&=&(1/n)^{2(q+j-1)H+1}   (t-s)^{(L-i-j+2)H}  .  
\label{e.rbd1}
\end{eqnarray}

 Take  $\al=2(q+j-1)H+1$, $\be=(L-i-j+2)H$, $\ga=0$  in Lemma \ref{lem.stnbd}. According to \eqref{e.rbd1}    the relation \eqref{e.abd1} holds for $A(s,t,n):=|r^{ijq}|_{L_{p}}$. Take $\ep=\al-(H+1/2)$. 
Since $ 2(q+j) =  {i+j}{ }  \geq 2$,  we have  $ 2(q+j-1)H\geq0$, and thus $\al\geq1$. This implies that  $\ep>0$. In summary, we have shown that the conditions of Lemma \ref{lem.stnbd} hold for $|r^{ijq}|_{L_{p}}$. So applying Lemma \ref{lem.stnbd} we have  the estimate \eqref{e.abd5}, namely:   
\begin{eqnarray}
|r^{ijq}|_{L_{p}}&\leq&K (1/n)^{H+1/2} (t-s)^{(L-1)H+1/2} 
\label{e.rbd}
\end{eqnarray}
and  
\begin{eqnarray}
\lim_{n\to\infty}n^{H+1/2}\sup_{(s,t)\in\cs_{2}(\ll0,T\rr)}|r^{ijq}|_{L_{p}} =0  
\qquad
\text{for all $(i,q,j)\in A$. }
\label{e.rlim}
\end{eqnarray}

\noindent\emph{Step 7: Estimate of $\ce_{L}^{x,1}(s,t)$.} Substituting the relation \eqref{e.etaijq}   into   \eqref{e.cexa} gives:  
\begin{eqnarray}
\ce_{L}^{x,1}(s,t) =
\sum_{(i,q,j)\in A}\wt{\ce}_{Liqj}^{x} (s,t) +\sum_{(i,q,j)\in A} r^{ijq}.
\label{e.cexa2}
\end{eqnarray}
Applying  \eqref{e.cexbcancel} to \eqref{e.cexa2}    we get: 
 \begin{eqnarray}
\ce_{L}^{x,1}(s,t) = \sum_{(i,q,j)\in A} r^{ijq}. 
\label{e.cexa1}
\end{eqnarray}
Now invoking the relations \eqref{e.rbd}-\eqref{e.rlim} we obtain:
 \begin{eqnarray}
|\ce_{L}^{x,1}(s,t) |_{L_{p}}&\leq&K (1/n)^{H+1/2} (t-s)^{(L-1)H+1/2}
\label{e.cexabd}
\end{eqnarray}
 and  
 \begin{eqnarray}
\lim_{n\to\infty}n^{H+1/2}\sup_{(s,t)\in\cs_{2}(\ll0,T\rr)}|\ce_{L}^{x,1}(s,t) |_{L_{p}} =0. 
\label{e.cexal} 
\end{eqnarray}

\noindent\emph{Step 8: The estimate of $\ce_{Liqj}^{x}(s,t)$ when $(i,q,j)\in A'$.} 
Recall that $\ce_{Liqj}^{x}(s,t)$ is defined in~\eqref{e.ceij}. 
Take $(i,q,j)$ such that $i-2q-j>0$. 
It follows from  \eqref{e.ceij} that we have 
\begin{eqnarray*}
|\ce_{Liqj}^{x}(s,t)|_{L_{2}}^{2} 
 &\lesssim& (1/n)^{4(q+j)H} \sum_{s\leq t_{k},t_{k'}<t} \int_{t_{k}}^{t_{k+1}}  \int_{t_{k'}}^{t_{k'+1}}  \big|\ca_{kk'uu'}\big|du' du.  
\end{eqnarray*}
where
\begin{eqnarray*}
\ca_{kk'uu'}&=& \mE\lc\delta^{{\diamond},i-2q-j}\lp x^{ L-i-j}_{st_{k}} 
\mathbf{1}_{[t_{k},u]}^{\otimes (i-2q-j)}
\rp \delta^{{\diamond},i-2q-j}\lp x^{ L-i-j}_{st_{k'}} 
\mathbf{1}_{[t_{k'},u']}^{\otimes (i-2q-j)}
\rp\rc . 
\end{eqnarray*}
By symmetry in $k$ and $k'$ we also have
\begin{eqnarray}
|\ce_{Liqj}^{x}(s,t)|_{L_{2}}^{2} 
&\lesssim& (1/n)^{4(q+j)H} \sum_{s\leq t_{k}\leq t_{k'}<t} \int_{t_{k}}^{t_{k+1}}  \int_{t_{k'}}^{t_{k'+1}}  \big|\ca_{kk'uu'}\big|du' du.  
 \label{e.cexbd}
\end{eqnarray}
Note that by applying the integration by parts it is easy to see that $\ca_{kk'uu'}$ is equal to the sum of terms in the following form:  
\begin{eqnarray}
\ca_{kk'uu'}^{c}  &:=& \langle \mathbf{1}_{[s,t_{k}]}, \mathbf{1}_{[t_{k'}, u']} \rangle_{\ch}^{c}\langle \mathbf{1}_{[s,t_{k'}]}, \mathbf{1}_{[t_{k}, u]} \rangle_{\ch}^{c}
\notag
\\
&&  \qquad\qquad\cdot\langle \mathbf{1}_{[t_{k'}, u']} , \mathbf{1}_{[t_{k}, u]} \rangle_{\ch}^{i-2q-j-c} 
 \mE[x_{st_{k}}^{ L-i-j-c}x_{st_{k'}}^{ L-i-j-c}] 
\label{e.ac}
\end{eqnarray}
for 
$ c=0, \dots, (i-2q-j)\wedge( L-i-j)$. So it follows from \eqref{e.cexbd} that we have the estimate 
\begin{eqnarray}
 |\ce_{Liqj}^{x}(s,t)|_{L_{2}}^{2}  \lesssim  (1/n)^{4(q+j)H}\sum_{c=0}^{(i-2q-j)\wedge( L-i-j)}\sum_{s\leq t_{k}\leq t_{k'}<t}  \int_{t_{k}}^{t_{k+1}}  \int_{t_{k'}}^{t_{k'+1}}      \big|\ca_{kk'uu'}^{c}\big|  
\notag\\
\qquad\qquad\qquad\qquad
\cdot   du' du.  
   \label{e.cexbd1}
\end{eqnarray}

Applying Lemma \ref{lem.abcd} to  \eqref{e.ac} gives the estimate:   
\begin{eqnarray}
\big|
\ca_{kk'uu'}^{c}
\big| 
&\lesssim& (1/n)^{4cH+2(i-2q-j-c)H}\rho(k-k')^{(2H-1)c}
\notag\\
&& \quad
\cdot(\rho(k-k')^{(2H-1)c}+\rho(k-ns)^{(2H-1)c}) 
\notag
\\
&&
 \qquad
\cdot\rho(k-k')^{(2H-2)(i-2q-j-c)}\cdot (t-s)^{2( L-i-j-c)H}  
\label{e.acbd}
\end{eqnarray}
for all $u\in[t_{k}, t_{k+1}]$, $u'\in[t_{k'}, t_{k'+1}]$ and $k\leq k'$. 
Substituting \eqref{e.acbd} into \eqref{e.cexbd1} we get
\begin{eqnarray}
|\ce_{Liqj}^{x}(s,t)|_{L_{2}}^{2} &\lesssim& \sum_{c=0}^{(i-2q-j)\wedge( L-i-j)} (t-s)^{2( L-i-j-c)H} (1/n)^{C_{0}}  J_{n}, 
 \label{e.cexbd2} 
\end{eqnarray}
where 
\begin{eqnarray*}
C_{0}=4cH+2(i-2q-j-c)H+4(q+j)H+2
= 2(i+j+c)H +2 
\end{eqnarray*}
and
\begin{eqnarray}
J_{n}&=&\sum_{s\leq t_{k}\leq t_{k'}<t} \rho(k-k')^{(2H-1)c}\big(\rho(k-k')^{(2H-1)c}+\rho(k-ns)^{(2H-1)c}\big) 
\notag
\\
&&\qquad\qquad\qquad\qquad\qquad\qquad\qquad\cdot\rho(k-k')^{(2H-2)(i-2q-j-c)}. 
\label{e.jn2}
\end{eqnarray}
Let us separate $J_{n}$ into two parts: 
\begin{eqnarray}
J_{n} &=&J_{n}^{1}+J_{n}^{2}\,,
\label{e.jn12}  
\end{eqnarray}
where
\begin{eqnarray}
\quad J_{n}^{1}=  \sum_{s\leq t_{k}\leq t_{k'}<t} \rho(k-k')^{C_{1}} , 
\qquad 
J_{n}^{2} =   \sum_{s\leq t_{k}\leq t_{k'}<t} \rho(k-k')^{C_{2}} \rho(k-ns)^{C_{3}}. 
\label{e.j1j2}
\end{eqnarray}
and
\begin{eqnarray}
&&C_{1}=2(2H-1)c+(2H-2)(i-2q-j-c), 
\label{e.const1}
\\
&&C_{2}=(2H-1)c+(2H-2)(i-2q-j-c), 
\qquad
C_{3}= (2H-1)c\,. 
\notag
\end{eqnarray}
Substituting \eqref{e.jn12} 
into \eqref{e.cexbd2} we get: 
\begin{eqnarray}
|\ce_{Liqj}^{x}(s,t)|_{L_{2}}^{2} &\lesssim& E_{Liqj}^{1}(s,t)+E_{Liqj}^{2}(s,t), 
\label{e.cel2bd}
\end{eqnarray}
where we denote
\begin{eqnarray}
E_{Liqj}^{ \xi}(s,t)&:=& \sum_{c=0}^{(i-2q-j)\wedge( L-i-j)} (t-s)^{2( L-i-j-c)H} (1/n)^{C_{0}}  J_{n}^{\xi} , 
\qquad
\xi=1,2. 
\label{e.cexx}
\end{eqnarray}

In the next two steps, we derive the estimate of \eqref{e.cexx}.

\noindent\emph{Step 9: The estimate of $E_{Liqj}^{1}(s,t)$  when $(i,q,j)\in A'\cap\{i\geq2\}$.} 
Applying Lemma \ref{lem.rho} (i) to $J_{n}^{1}$ in \eqref{e.j1j2} we obtain the estimate:
 \begin{eqnarray}
J_{n}^{1}&\lesssim&
\begin{cases}
(n(t-s))^{  C_{1}+2 }, &C_{1}> -1. 
\\
n(t-s)\log [n(t-s)], & C_{1}\leq -1. 
\end{cases}
\label{e.j1c1}
\end{eqnarray}

Recall that  $E_{Liqj}^{1}(s,t)$ is defined in \eqref{e.cexx}. 
 In the following we estimate  $E_{Liqj}^{1}(s,t)$ according to the two cases in \eqref{e.j1c1}. 

{\bf Case $C_{1}>-1$:}  
Applying \eqref{e.j1c1} to \eqref{e.cexx} we get: 
\begin{eqnarray}
E_{Liqj}^{1}(s,t)
&\lesssim&  \sum_{c=0}^{(i-2q-j)\wedge( L-i-j)}   (t-s)^{2( L-i-j-c)H}(1/n)^{C_{0}} (n(t-s))^{C_{1}+2} 
\notag
\\
&=&(1/n)^{\al}(t-s)^{\be} ,
\label{e.cex1bd3}
\end{eqnarray}
where
 \begin{eqnarray}
\al&=&C_{0}-(C_{1}+2)=2H(2j+2q)+2(i-2q-j), 
\notag
\\
 \be&=& 2( L-i-j-c)H+C_{1}+2 . 
\label{e.cex1bd1}
\end{eqnarray}
The relation \eqref{e.cex1bd3} shows that the condition \eqref{e.abd1} in Lemma \ref{lem.stnbd} holds for $A(s, t, n):= E_{Liqj}^{1}(s,t)$ and $\ga=0$.  
Take $\ep=\al-(2H+1)$.  Since $i-2q-j>0$   we have $\al\geq2$ and thus $\ep>0$.  Applying Lemma \ref{lem.stnbd} with  $\al$ and $ \be $ given in \eqref{e.cex1bd1} and the $ \ep  $ we just defined gives: 
\begin{eqnarray}
E_{Liqj}^{1}(s,t)
&\leq&K(t-s)^{2(L-1)H+1}(1/n)^{2H+1}.
\label{e.cex1bd2}
\end{eqnarray}
Furthermore, we have the convergence:
\begin{eqnarray}
\label{e.cex1lim}
\lim_{n\to\infty}n^{2H+1}\sup_{(s,t)\in\cs_{2}(\ll0,T\rr)}E_{Liqj}^{1}(s,t) = 0. 
\end{eqnarray}

{\bf Case $C_{1} \leq -1$:}  As in the previous case,   substituting the second inequality of   \eqref{e.j1c1} into \eqref{e.cexx} we   get: 
\begin{eqnarray}
E_{Liqj}^{1}(s,t)
&\lesssim&  (1/n)^{\al}(t-s)^{\be}\log (n(t-s)) . 
\label{e.cex1bd5}
\end{eqnarray}
In this case we have 
 \begin{eqnarray}
\al=C_{0}-1=2H(i+j+c)+1, 
\qquad
 \be= 2( L-i-j-c)H+1 . 
\label{e.cex1bd4}
\end{eqnarray}

The relation \eqref{e.cex1bd5} shows that the condition \eqref{e.abd1} in Lemma \ref{lem.stnbd} holds for $A(s, t, n):= E_{Liqj}^{1}(s,t)$ and $\ga=1$.
 Take $\ep=\al-(2H+1)$.  Since $i\geq2$ the definition of $\al$ in \eqref{e.cex1bd4} shows that  $\al\geq 4H+1$, and therefore we have $\ep>0$. As in the previous case, by applying Lemma \ref{lem.stnbd}  we obtain that the relations   \eqref{e.cex1bd2} and \eqref{e.cex1lim}   hold. 


\noindent\emph{Step 10: The estimate of $E_{Liqj}^{2}(s,t)$ when    $(i,q,j)\in A'\cap\{i\geq2\}$.} As for the estimate of $J_{n}^{1}$ in \eqref{e.j1c1} we can bound $J_{n}^{2}$  in   several different cases. Applying Lemma \ref{lem.rho} (i) to $J_{n}^{2}$ in \eqref{e.j1j2}, first to the sum $\sum_{t_{k'}}$ and then to  the sum $\sum_{t_{k}}$, we obtain:  
\begin{eqnarray}
\qquad J_{n}^{2}&\lesssim&
\begin{cases}
(n(t-s))^{  C_{2}+1 }\cdot (n(t-s))^{  C_{3}+1 },  & \text{when $C_{2}>-1$ and $C_{3}>-1$} . 
\\   
 (n(t-s))^{C_{3}+1} \log [(n(t-s))],  & \text{when $C_{2} \leq-1$ and $C_{3} >-1$} . 
\\   
(n(t-s))^{C_{2}+1} \log [(n(t-s))] , 
   & \text{when $C_{2} >-1$ and $C_{3}\leq -1$} . 
\\   
(\log [n(t-s)])^{2}, & \text{when $C_{2} \leq-1$ and $C_{3} \leq-1$} . 
\end{cases}  
\label{e.j2bd}
\end{eqnarray}

{\bf Case: $C_{2}>-1$ and $C_{3}>-1$.}  
Since $(C_{2}+1) + (C_{3}+1)= C_{1}+2$, the estimate of
  $J_{n}^{2}$ in \eqref{e.j2bd} is the same  as   in \eqref{e.j1c1} for   $J_{n}^{1}$. 
So following the same argument as for $J_{n}^{1}$ we obtain:
\begin{eqnarray}
&&
E_{Liqj}^{2}(s,t)
 \leq K(t-s)^{2(L-1)H+1}(1/n)^{2H+1} 
\label{e.cex2bd2}
\\
&&\qquad\text{and}\qquad\lim_{n\to\infty}n^{2H+1}\sup_{(s,t)\in\cs_{2}(\ll0,T\rr)}E_{Liqj}^{2}(s,t) = 0. 
\label{e.cex2lim}
\end{eqnarray}

{\bf Case: $C_{2} \leq-1$ and $C_{3} >-1$.}   
Substituting \eqref{e.j2bd} into \eqref{e.cexx} we  obtain:
\begin{eqnarray*}
E_{Liqj}^{2}(s,t)
&\lesssim&  \sum_{c=0}^{(i-2q-j)\wedge( L-i-j)} (t-s)^{2( L-i-j-c)H}  (1/n)^{C_{0}} 
\cdot (n(t-s))^{C_{3}+1} \log (n(t-s))
\\
&=& \sum_{c=0}^{(i-2q-j)\wedge( L-i-j)}   (1/n)^{\al}(t-s)^{\be} \log (n(t-s)),
\end{eqnarray*}
where
\begin{eqnarray}
\qquad\qquad
\al&=&C_{0}-(C_{3}+1)= 2(i+j)H+1+c\,,
\label{e.cex2bd1}
\\
\be &=& 2( L-i-j-c)H +(C_{3}+1). 
\notag
\end{eqnarray}

Take $\ep=\al-(2H+1)$.  Since $i\geq 2$, the definition of $\al$ in \eqref{e.cex2bd1} shows that we have $\al\geq 4H+1$,  and thus $\ep>0$. Applying Lemma \ref{lem.stnbd}  we obtain that the relations \eqref{e.cex2bd2} and \eqref{e.cex2lim} hold in this case.

{\bf Case: $C_{2} >-1$ and $C_{3}\leq -1$.}  
As before, substituting \eqref{e.j2bd} into \eqref{e.cexx} gives:  
\begin{eqnarray*}
E_{Liqj}^{2}(s,t)
&\lesssim&  \sum_{c=0}^{(i-2q-j)\wedge( L-i-j)}   (t-s)^{2( L-i-j-c)H} (1/n)^{C_{0}} \cdot 
 (n(t-s))^{C_{2}+1} \log  (n(t-s))
\\
&=&\sum_{c=0}^{(i-2q-j)\wedge( L-i-j)} 
(1/n)^{\al}(t-s)^{\be} \log (n(t-s)), 
\end{eqnarray*}
where
\begin{eqnarray*}
&&\al = C_{0}-(C_{2}+1) = 2H(2j+2q+c)+2(i-2q-j)-c+1
\\
&&\be= 2( L-i-j-c)H + (C_{2}+1). 
\end{eqnarray*}

Take $\ep=\al-(2H+1)$ as before.   Since $i-2q-j\geq c$ and $i-2q-j>0$,  we have $2(i-2q-j)-c\geq i-2q-j\geq 1$. This implies that $\al\geq2$, and therefore $\ep>0$. Applying Lemma \ref{lem.stnbd} we obtain that relations \eqref{e.cex2bd2} and \eqref{e.cex2lim}  holds in this case. 
      

{\bf Case: $C_{2} \leq-1$ and $C_{3} \leq-1$.}  As before, by applying \eqref{e.j2bd} to \eqref{e.cexx} we obtain:  
\begin{eqnarray*}
E_{Liqj}^{2}(s,t)
&\lesssim&  \sum_{c=0}^{(i-2q-j)\wedge( L-i-j)}  (t-s)^{\be} (1/n)^{\al} (\log (n(t-s)))^{\ga}
 , 
\end{eqnarray*}
where
 $
\al = C_{0} $,  $
\be = 2( L-i-j-c)H$,   and
$\ga = 2$. 
The fact   that $i\geq 2$ implies that $\al\geq 4H+2>2H+1$ and thus $\ep:=\al-(2H+1)>0$. Applying Lemma \ref{lem.stnbd} with these values of $\al$, $\be$, $\ga$, $\ep$ we obtain that relations \eqref{e.cex2bd2} and \eqref{e.cex2lim}  hold in this case. 

In summary, we have shown that relations \eqref{e.cex2bd2}-\eqref{e.cex2lim} hold for all $(i,q,j)\in A'\cap\{i\geq 2\}$. 

\noindent\emph{Step 11: The estimate of $\ce_{L}^{x,2}(s,t)$.} 
Substituting the relation   \eqref{e.cel2bd} into \eqref{e.cexa} we obtain the following relation:
\begin{eqnarray}
|\ce_{L}^{x,2}(s,t)|_{L_{2}}^{2} &\lesssim& \sum_{(i,q,j)\in A'\cap \{i\geq 2\}} \lp
E_{Liqj}^{1}(s,t)+E_{Liqj}^{2}(s,t)
\rp. 
\label{e.cexbbd1}
\end{eqnarray}
Applying   the relations \eqref{e.cex1bd2}, \eqref{e.cex2bd2}, \eqref{e.cex1lim},  \eqref{e.cex2lim}  to \eqref{e.cexbbd1} we obtain:
\begin{eqnarray}
|\ce_{L}^{x,2}(s,t)|_{L_{2}}^{2} &\leq &K (1/n)^{2H+1}(t-s)^{2(L-1)H+1}  
\label{e.cexbbd}
\end{eqnarray}
for all $(s,t,n)\in\cs_{2}(\ll0,T\rr)\times \NN$ and  
\begin{eqnarray}
\lim_{n\to\infty} n^{2H+1} \sup_{(s,t)\in\cs_{2}(\ll0,T\rr)}|\ce_{L}^{x,2}(s,t)|_{L_{2}}^{2}  =0.
\label{e.cexbl} 
\end{eqnarray}

\noindent\emph{Step 12: Estimate of $\ce_{Liqj}^{x}(s,t)$ when    $(i,q,j)\in A'\cap\{i=1\}$.} 
Recall that $\ce_{L}^{x,3}(s,t)$ is defined in \eqref{e.cexc}.  
 Note that when $i=1$ and $i-2q-j>0$ we   have $j=0$, $q=0$. This implies that  
\begin{eqnarray}
 \ce_{L}^{x,3}(s,t) =  \ce_{L100}^{x}(s,t)= \sum_{s\leq t_{k}<t}  \int_{t_{k}}^{t_{k+1}}        
 \delta^{\diamond} \lp x^{ L-1}_{st_{k}} 
\mathbf{1}_{[t_{k},u]} \rp 
   du. 
\label{e.cexc2}
\end{eqnarray}
 In this case the relation  \eqref{e.jn2} becomes:
 \begin{eqnarray*}
J_{n}&=& \begin{cases}
\sum\limits_{s\leq t_{k}\leq t_{k'}<t} 
\rho(k-k')^{2H-1}(\rho(k-k')^{2H-1}+\rho(k-ns)^{2H-1} )
&\text{when }c=1.
\\
\sum\limits_{s\leq t_{k}\leq t_{k'}<t} 
\rho(k-k')^{2H-2}&\text{when }c=0.
\end{cases}
\end{eqnarray*}
Applying Lemma \ref{lem.rho} (i) we obtain: 
 \begin{eqnarray}
|J_{n}|&\lesssim&  \begin{cases}
n+n^{4H}
&\text{when }c=1.
\\
n&\text{when }c=0.
\end{cases}
\label{e.jnbd1}
\end{eqnarray}
Substituting \eqref{e.jnbd1} into \eqref{e.cexbd2} and taking into account that $C_{0}=4H+2$ when $c=1$ and $C_{0}=2H+2$ when $c=0$,   we  obtain    
  the estimate: 
\begin{eqnarray}
|\ce_{L}^{x,3}(s,t)|_{L_{2}}^{2} &\leq &K (1/n)^{2H+1}(t-s)^{2(L-1)H+1}  . 
\label{e.cexcbd}
\end{eqnarray} 

 
\noindent\emph{Step 13: Conclusion.}  
 Reporting the estimates \eqref{e.cexabd}, \eqref{e.cexbbd}, \eqref{e.cexcbd} to \eqref{e.cexd} we obtain that the upper-bound estimate \eqref{e.cexbdl2} holds for $|\ce_{L}^{x}(s,t) |_{L_{2}}$. 
Since $\{\ce_{L}^{x}(s,t), n\in\NN\}$ stays in the first $L$ Wiener chaos,   by hypercontractivity property the relation   \eqref{e.cexbdl2}   holds for $|\ce_{L}^{x}(s,t) |_{L_{p}}$ for any $p\geq 1$. This concludes the estimate    \eqref{e.cexbdl2}. 

Recalling the definition of $Z^{(n)}$ in \eqref{e.zn} it is clear that 
$n^{H+1/2} \ce_{L}^{x,3}(s,t) = Z^{(n),L}_{st}$ for $ (s,t)\in\cs_{2}(\ll0,T\rr)$. 
Applying this relation and the two convergences in \eqref{e.cexal} and \eqref{e.cexbl} to the right-hand side of \eqref{e.cexd} we obtain the  relation \eqref{e.cexbd5}.  
\end{proof}

\begin{remark}\label{remark.int}
Suggested by the proof of Lemma \ref{lem.cex}, we conjecture that the compensated Riemann-Stieltjes sum can be identified by a Skorohod-type Riemann sum, namely, 
\begin{eqnarray}
\lim_{n\to\infty}\sum_{0\leq t_{k}<t} \sum_{i=1}^{\ell}f^{(i-1)}(x_{t_{k}}) x^{i}_{t_{k}t_{k+1}} =\lim_{n\to\infty} \sum_{0\leq t_{k}<t} \delta^{\diamond} ( f(x_{t_{k}})\mathbf{1}_{[t_{k},t_{k+1}]} ),
\label{e.ints}
\end{eqnarray}
where $x$ a standard fBm, $f$ is a smooth function, and  $\ell$ is such that $\ell (H+1)>1$.    
  Note that   the left-hand side of \eqref{e.ints} is equal to the rough integral $\int_{0}^{t} f(x_{s})dx_{s}$ almost surely. This implies that  relation \eqref{e.ints} would suggest an alternative way to define   rough integrals. We will explore this problem in a future research.
\end{remark}

With Lemma \ref{lem.cex} in hand, in the following we consider the convergence of $\ce^{n,z}_{\ell}(s,t)$. 
\begin{theorem}\label{thm.taylor}
Let $\ell\in\NN$   be the least integer  such that $\ell H+1/2>1$.    
Let $\bfz=(z,z', \dots,z^{(\ell-1)})$ be   continuous processes on $[0,T]$  and   $r^{(i)}$, $i=0,\dots, \ell-1$ be the remainders of $\bfz$ defined in \eqref{e.r}.  
Let $\ce^{z, n}_{\ell} (s,t)$, $(s,t)\in\cs_{2}([0,T]) $ be defined in \eqref{e.ce},  and denote the Wiener integral:    
\begin{eqnarray}
\ce^{z}(s,t) &:=& c_{H}^{1/2}T^{H+1/2}\int_{s}^{t}z_{u}dW_{u}\,, 
\qquad
(s,t)\in\cs_{2}([0,T]), 
\label{e.cezl}
\end{eqnarray}
where $W$ is a   Brownian motion independent of $x$ and $c_{H}$ is the constant defined in \eqref{e.rho}.

\noindent (i) Suppose   that $\bfz$ is controlled by $(x,\ell,H-\ep)$ in $L_{p}$ for any  $\ep>0$ and $p\geq 1$.  
Then for any   $p\geq1$ there exists a constant $K>0$ such that: 
\begin{eqnarray}
|\ce^{z, n}_{\ell}(s,t)|_{L_{p}}&\leq &K (1/n)^{H+1/2}(t-s)^{1/2} 
\label{e.czbd}
\end{eqnarray}
for all  $(s,t)\in\cs_{2}(\ll0,T\rr)$ and $n\in\NN$.

\noindent (ii)  
 Suppose that $\bfz$ is controlled by $(x,\ell,H)$ almost surely. Then we have the convergence $n^{H+1/2}\ce^{z,n}_{\ell} \xrightarrow{stable \,\,f.d.d.} \ce^{z}$ as $n\to\infty$. That is, the f.d.d. of the process  $n^{H+1/2}\ce^{z,n}_{\ell}(s,t)$, $(s,t)\in\cs_{2}([0,T])$ converges $\cf^{x}$-stably to that of $\ce^{z}(s,t)$,  $(s,t)\in\cs_{2}([0,T])$ as $n\to\infty$. 
 
\end{theorem}

\begin{proof}
 By localization 
 (cf. \cite[Lemma 3.4.5]{JP})     
it suffices to prove Theorem \ref{thm.taylor} (ii) under the condition that $\bfz$ is   controlled by $(x,\ell,H-\ep)$ in $L_{p}$ for any  $\ep>0$ and $p\geq 1$.  In the following  we always assume that this condition holds.  

\noindent\emph{Step 1: Estimate of remainders.}  
Define  the following remainder process:     
\begin{eqnarray*}
\crr^{\ell}_{st}  &:=& \sum_{i=1}^{\ell} \cj_{s}^{t} (r^{(i-1)}, h^{i}), 
\qquad
(s,t)\in\cs_{2}(\ll0,T\rr). 
\end{eqnarray*}
In the following we   derive an estimate of $\crr^{\ell}_{st}  $.  
We start by  computing $\delta\crr^{\ell}_{sut} $. 
Recall that the operator $\delta$ is defined in \eqref{e.delta}. A direct    computation gives:  
\begin{eqnarray}
\delta\crr^{\ell}_{sut}  
&=&\sum_{i=1}^{\ell} \sum_{u\leq t_{k}<t} (\delta r^{(i-1)}_{sut_{k}}+r^{(i-1)}_{su})h^{i}_{t_{k}t_{k+1}}, \qquad
(s,u,t)\in\cs_{3}(\ll0,T\rr) . 
\label{e.drl}
\end{eqnarray}
Substituting the relation  \eqref{e.dr} into \eqref{e.drl} we have:
\begin{eqnarray*}
\delta\crr^{\ell}_{sut} =\sum_{i=1}^{\ell} \sum_{u\leq t_{k}<t} 
 \sum_{j=i}^{\ell} r^{(j-1)}_{su} 
x^{j-i}_{ut_{k}}
h^{i}_{t_{k}t_{k+1}} 
= \sum_{i=1}^{\ell}  \sum_{j=i}^{\ell}
 r^{(j-1)}_{su} 
\cj_{u}^{t}(x^{j-i},h^{i})  
.
\end{eqnarray*}
Invoking the definition of $ \ce_{j}^{x}(u,t)$ in \eqref{e.cex} we obtain:  
\begin{eqnarray}
\delta\crr^{\ell}_{sut} &=&  \sum_{j=1}^{\ell} 
 r^{(j-1)}_{su} \ce_{j}^{x}(u,t) .
 \label{e.drl2}
\end{eqnarray}
Applying Lemma \ref{lem.cex} (i) to $ \ce_{j}^{x}(u,t) $ in \eqref{e.drl2} and taking into account the estimate $ |r^{(j-1)}_{su}|_{L_{2p}} \leq K  (u-s)^{(\ell - j+1)H}$ we obtain:
\begin{eqnarray}
|\delta\crr^{\ell}_{sut} |_{L_{p}}&\lesssim&  \sum_{j=1}^{\ell} 
 |r^{(j-1)}_{su}|_{L_{2p}}\cdot| \ce_{j}^{x}(u,t)|_{L_{2p}}
\notag \\
 &\lesssim& (u-s)^{(\ell - j+1)(H-\ep)}  (1/n)^{H+1/2}(t-u)^{(j-1)H+1/2}
\notag \\
 &\leq&  (1/n)^{H+1/2}(t-s)^{\ell (H-\ep)+1/2}. 
 \label{e.drl1}
\end{eqnarray}
Since $\ell H+1/2>1$, according to  the   discrete sewing lemma (see e.g. \cite[Lemma 2.7]{liu2019first}) the relation \eqref{e.drl1} implies that  we have the estimate: 
\begin{eqnarray}
| \crr^{\ell}_{st} |_{L_{p}}&\lesssim&  (1/n)^{H+1/2}(t-s)^{\ell  (H-\ep)+1/2}. 
\label{e.rlbd}
\end{eqnarray}

\noindent\emph{Step 2: Estimate of  $\ce^{z,n}_{\ell}(s,t)$.}
By Definition \ref{def.control} we have
 $z^{(i-1)}_{t} = \sum_{j=i}^{\ell} z^{(j-1)}_{s}x^{j-i}_{st}+r^{(i-1)}_{st}$ 
for  $ i=1,\dots,\ell$.
Substituting this relation into \eqref{e.ce}  we obtain: 
\begin{eqnarray}
\ce^{z,n}_{\ell}(s,t) &=&  \sum_{i=1}^{\ell}\sum_{j=i}^{\ell}z^{(j-1)}_{s}\cj_{s}^{t}(x^{j-i}, h^{i}) + \sum_{i=1}^{\ell}\cj_{s}^{t}(r^{(i-1)}, h^{i})
\notag
\end{eqnarray}
Exchanging the order of summations in $i$ and $j$ and invoking the definition of $\crr^{\ell}_{st} $ we obtain:   
\begin{eqnarray}
\ce^{z,n}_{\ell}(s,t) =   \sum_{j=1}^{\ell}z^{(j-1)}_{s}\sum_{i=1}^{j}\cj_{s}^{t}(x^{j-i}, h^{i}) +  \crr^{\ell}_{st} 
= \sum_{j=1}^{\ell}z^{(j-1)}_{s} \ce^{x}_{j}(s,t)+ \crr^{\ell}_{st}  . 
\label{e.cez2}
\end{eqnarray}
Taking $L_{p}$ norm in both sides of \eqref{e.cez2} and 
then applying the estimates  \eqref{e.cexbdl2} and \eqref{e.rlbd} to the right-hand side   we obtain: 
\begin{eqnarray}
|\ce^{z,n}_{\ell}(s,t)|_{L_{p}}  \lesssim  \sum_{j=1}^{\ell} (1/n)^{H+1/2}(t-s)^{1/2+(L-1)H} 
\notag
\\
\qquad+(1/n)^{H+1/2}(t-s)^{\ell  (H-\ep)+1/2} . 
\label{e.cezbd1}
\end{eqnarray}
It is readily checked   that the right-hand side of \eqref{e.cezbd1} is bounded by $K(1/n)^{H+1/2}(t-s)^{ 1/2} $ for some constant $K $. 
This completes  the proof of Theorem \ref{thm.taylor} (i). 


\noindent\emph{Step 3: Convergence of   $\ce^{z,n}_{\ell} $.}
Recall that $\ll0,T\rr$ stands for the uniform parition of $[0,T]$: $0=t_{0}<\cdots<t_{n}=T$. 
Take  another partition of $[0,T]$:   
  $\pi=\{0=s_{0} <\cdots < s_{N}=T\}$. For each subinterval $[s_{l}, s_{l+1}]$, $l=0,\dots, N-1$ take    $M\in\NN$ and let $s_{l}=u_{0}^{l}<u_{1}^{l}<\cdots < u_{M}^{l}=s_{l+1}$ be the uniform partition of $[s_{l},s_{l+1}]$. For convenience  we will  write $u^{l}_{j}=u_{j}$ in the sequel. 
  
  Let $u_{j}'\in\ll0,T\rr$, $j=0,\dots,M$  be such that $
\{t_{k}: u_{j}'\leq t_{k}<u_{j+1}'\}=\{t_{k}: u_{j} \leq t_{k}<u_{j+1} \}$. 
 By the definition of $\ce^{z,n}_{\ell}(s_{l},s_{l+1}) $ in \eqref{e.ce} it is easy to see that   
\begin{eqnarray}
\ce^{z,n}_{\ell}(s_{l},s_{l+1}) &=&\sum_{j=0}^{M-1}\ce^{z,n}_{\ell}(u_{j} ,u_{j+1} ) =\sum_{j=0}^{M-1}\ce^{z,n}_{\ell}(u_{j}',u_{j+1}') .
\label{e.cezsj}
\end{eqnarray}
Substituting  the relation  \eqref{e.cez2} into the right-hand side of \eqref{e.cezsj} and then multiplying both sides of \eqref{e.cezsj} by $n^{H+1/2}$ we have 
\begin{eqnarray}
n^{H+1/2}\ce^{z,n}_{\ell}(s_{l},s_{l+1}) &=&
n^{H+1/2}\sum_{i=1}^{\ell}\sum_{j=0}^{M-1}z^{(i-1)}_{u'_{j}} \ce^{x}_{i}(u'_{j},u'_{j+1})+ n^{H+1/2}\sum_{j=0}^{M-1}\crr^{\ell}_{u'_{j}  u'_{j+1}}
\notag
\\
&=& V_{n,M}^{l}+ \wt{V}_{n,M}^{l},  
\label{e.cezd}
\end{eqnarray}
where
\begin{eqnarray}
V_{n,M}^{l}&=& n^{H+1/2}\sum_{j=0}^{M-1}z_{u'_{j} } \ce^{x}_{1}(u'_{j} ,u'_{j+1})
\label{e.vnm}
\\
 \wt{V}_{n,M}^{l}&=&n^{H+1/2}\sum_{i=2}^{\ell}\sum_{j=0}^{M-1}z^{(i-1)}_{u'_{j}} \ce^{x}_{i}(u'_{j},u'_{j+1})+ n^{H+1/2}\sum_{j=0}^{M-1}\crr^{\ell}_{u'_{j} u'_{j+1}}. 
 \label{e.vt}
\end{eqnarray}
In the following we consider the  convergence of the components in \eqref{e.vnm}-\eqref{e.vt}. The convergence of $n^{H+1/2}\ce^{z,n}_{\ell} $ will then follow by applying Lemma \ref{lem.limlim}. 

Recall the estimate \eqref{e.rlbd} for $\crr^{\ell}_{st}$. So we have 
\begin{eqnarray}
 \qquad\qquad
 \big|\sum_{j=0}^{M-1}\crr^{\ell}_{u_{j}' u_{j+1}'} \big|_{L_{p}}
  \leq  \sum_{j=0}^{M-1}\big|\crr^{\ell}_{u_{j}' u_{j+1}'} \big|_{L_{p}}
 \lesssim \sum_{j=0}^{M-1}(1/n)^{H+1/2}(u_{j+1}'-u_{j}')^{\ell H+1/2}. 
\label{e.rlbd2}
\end{eqnarray}
Since $\lim_{n\to\infty}(u_{j}',u_{j+1}')= (u_{j},u_{j+1})$,  
sending $n\to\infty$ and then $M\to\infty$ in \eqref{e.rlbd2} gives:   
\begin{eqnarray}
\qquad
\lim_{M\to\infty}\limsup_{n\to\infty} n^{H+1/2}\big|\sum_{j=0}^{M-1}\crr^{\ell}_{u_{j}' u_{j+1}'} \big|_{L_{p}}
&\lesssim&
\lim_{M\to\infty} \sum_{j=0}^{M-1} (u_{j+1}-u_{j})^{\ell H+1/2}=0, 
\label{e.rlbd4}
\end{eqnarray}
where the last relation in \eqref{e.rlbd4} holds since the exponent $\ell H+1/2 $ is greater than  $ 1$. 

We turn to the convergence of $n^{H+1/2}\sum_{j=0}^{M-1}z^{(i-1)}_{u_{j}'} \ce^{x}_{i}(u_{j}',u_{j+1}') $ in \eqref{e.vt}.  By applying Theorem \ref{lem.cex} (ii)   we have the convergence: 
\begin{eqnarray}
\lim_{n\to\infty}\big|n^{H+1/2}\sum_{j=0}^{M-1}z^{(i-1)}_{u_{j}'} \ce^{x}_{i}(u_{j}',u_{j+1}') -  \sum_{j=0}^{M-1}z^{(i-1)}_{u_{j}'} Z^{(n),i}_{u_{j}'u_{j+1}'}\big|_{L_{p}} =0.
\label{e.cexz1}
\end{eqnarray}
  Note that  by   Definition \ref{def.control}  we have $|z^{(i-1)}_{u_{j}'}-z^{(i-1)}_{u_{j}}|_{L_{p}}\to 0$. 
  On the other hand, according to relation  \eqref{e.zn}  we have 
$ Z^{(n),i}_{u_{j}'u_{j+1}'}=  Z^{(n),i}_{u_{j}u_{j+1}} $.
 It follows that  we can replace  $z^{(i-1)}_{u_{j}'}Z^{(n),i}_{u_{j}'u_{j+1}'}$ by $z^{(i-1)}_{u_{j}}  Z^{(n),i}_{u_{j}u_{j+1}} $ in \eqref{e.cexz1}. Namely, we have  
\begin{eqnarray}
\lim_{n\to\infty}\big|n^{H+1/2}\sum_{j=0}^{M-1}z^{(i-1)}_{u_{j}'} \ce^{x}_{i}(u_{j}',u_{j+1}')   -  \sum_{j=0}^{M-1}z^{(i-1)}_{u_{j}} Z^{(n),i}_{u_{j}u_{j+1}} \big|_{L_{p}} =0. 
\label{e.zcex} 
\end{eqnarray}
Applying  Theorem \ref{prop.du} to $\sum_{j=0}^{M-1}z^{(i-1)}_{u_{j}} Z^{(n),i}_{u_{j}u_{j+1}}$ in \eqref{e.zcex}  we obtain the $\cf^{x}$-stable convergence:   
\begin{eqnarray}&&\Big( n^{H+1/2}\sum_{j=0}^{M-1}z^{(i-1)}_{u_{j}'} \ce^{x}_{i}(u_{j}',u_{j+1}') , \,l=0,\dots,N-1
 \Big)
\notag
\\
&&\qquad\qquad \to 
 \Big(
   \sum_{j=0}^{M-1}z^{(i-1)}_{u_{j}}  Z^{i}_{u_{j}u_{j+1}} , \,l=0,\dots,N-1
 \Big) 
 \qquad\text{as $n\to\infty$}. 
\label{e.zcexconv}
\end{eqnarray}

The It\^o isometry implies that
\begin{eqnarray}
|\sum_{j=0}^{M-1}z^{(i-1)}_{u_{j}}  Z^{i}_{u_{j}u_{j+1}}|_{L_{2}}^{2}   &\lesssim&  \mE \sum_{j=0}^{M-1}|z^{(i-1)}_{u_{j}}|^{2} \int_{u_{j}}^{u_{j+1}}(x^{i-1}_{u_{j}u})^{2}du
\notag\\
 &\lesssim&  \sum_{j=0}^{M-1} |u_{j+1}-u_{j}|^{2(i-1)H+1} . 
\label{e.zzl2}
\end{eqnarray}
Since the exponenet $2(i-1)H+1>1$ for $i\geq 2$, it follows from \eqref{e.zzl2}    that we have: 
\begin{eqnarray}
\sum_{j=0}^{M-1}z^{(i-1)}_{u_{j}}  Z^{i}_{u_{j}u_{j+1}} \to 0 
\label{e.zzconv}
\end{eqnarray}
in probability as $M\to\infty$. In summary of \eqref{e.zcexconv} and \eqref{e.zzconv}  we have shown that   
\begin{eqnarray}
n^{H+1/2}\sum_{j=0}^{M-1}z^{(i-1)}_{u_{j}'} \ce^{x}_{i}(u_{j}',u_{j+1}') \to  0
\label{e.zcexconv1}
\end{eqnarray}
in distribution as $n\to\infty$ and then $M\to\infty$ for  all  $i=2,\dots,\ell$. 

 When $i=1$ equation \eqref{e.zist} becomes $Z^{1}_{u_{j}u_{j+1}} =c_{H}^{1/2}T^{H+1/2} ( W_{u_{j+1}}-W_{u_{j}}) $, and therefore 
\begin{eqnarray*}
 \sum_{j=0}^{M-1}z^{(i-1)}_{u_{j}}  Z^{i}_{u_{j}u_{j+1}} = c_{H}^{1/2}T^{H+1/2}\sum_{j=0}^{M-1}z_{u_{j}}  (W_{u_{j+1}}-W_{u_{j}})
.
\end{eqnarray*}
Sending $M\to\infty$  we obtain the convergence in $L_{2}$ to the Wiener integral: 
\begin{eqnarray}
 \sum_{j=0}^{M-1}z^{(i-1)}_{u_{j}}  Z^{i}_{u_{j}u_{j+1}}  \to 
    c_{H}^{1/2}T^{H+1/2}\int_{s_{l}}^{s_{l+1}} z_{u}dW_{u} .  
 \label{e.zzconv1}
\end{eqnarray} 

\noindent\emph{Step 4: Conclusion.}
Let  
$
V_{n,M} = (V_{n,M}^{l},    x_{s_{l}s_{l+1}}^{1}, l=0,\dots,N-1 )$, 
 where  $V_{n,M}^{l}$  are defined in \eqref{e.vnm}. 
The convergence in \eqref{e.zcexconv}  and \eqref{e.zzconv1}    shows that we have
\begin{eqnarray}
V_{n,M}
\to 
\Big(
c_{H}^{1/2}T^{H+1/2}\int_{s_{l}}^{s_{l+1}} z_{u}dW_{u},\,  \, x_{s_{l}s_{l+1}}^{1}, \,l=0,\dots,N-1
\Big) 
\label{e.vconv}
\end{eqnarray}
in distribution as $n\to\infty$ and then $M\to\infty$. 
On the other hand, 
by applying the convergences \eqref{e.rlbd4} and \eqref{e.zcexconv1} to   \eqref{e.vt} we obtain:  
\begin{eqnarray}
\lim_{M\to\infty} \limsup_{n\to\infty}P(|\wt{V}^{l}_{n,M}|\geq \ep) =0.
\label{e.vtconv}
\end{eqnarray}

Let 
$\wt{V}_{n,M} = (\wt{V}_{n,M}^{0}, \dots, \wt{V}_{n,M}^{N-1}, 0,\dots,0 )$ and $V_{n,M}$ be defined as above. Then it follows from \eqref{e.cezd} that   $V_{n,M}+\wt{V}_{n,M}$ equals the f.d.d. of $(\ce_{\ell}^{z,n}, x^{1})$, namely, 
\begin{eqnarray}
V_{n,M}+\wt{V}_{n,M} = (n^{H+1/2}\ce^{z,n}_{\ell}(s_{l},s_{l+1}),\,  \, x_{s_{l}s_{l+1}}^{1}, \,l=0,\dots,N-1 ). 
\notag
\end{eqnarray}
 According to 
\eqref{e.vconv}  and \eqref{e.vtconv}  the conditions \eqref{e.vnmconv} and \eqref{e.limlim} in Lemma \ref{lem.limlim} are satisfied.  It follows from  Lemma \ref{lem.limlim} that the f.d.d. of $(n^{H+1/2}\ce^{z, n}_{\ell},x^{1})$ converges in distribution to that of $(\ce^{z}, x^{1})$ as $n\to\infty$.  This proves Theorem \ref{thm.taylor} (ii).  
\end{proof}


We end the section by stating convergence of Riemann sum for the regular integral $\int z_{u}du$, where $z$ is a continuous path. The result is a consequence of Theorem \ref{thm.taylor} (i).  Recall that $x$ is a fBm with Hurst parameter $H<1/2$ and $\ell$ is the least integer such that $\ell H>1/2$. 

\begin{thm}\label{thm.integral}
Let $\bfz=(z,z',\dots,z^{(\ell)})$ be a continuous process controlled by $(x,\ell+1,H)$ almost surely. Let
\begin{eqnarray}
U^{n}_{st}= n^{H+1/2}\big(\int_{s}^{t} z_{u}d {u}-\sum_{s\leq t_{k}<t}z_{t_{k}}\cdot {T}/{n} \big),\qquad (s,t)\in\cs_{2}([0,T]). 
\notag
\end{eqnarray}
 Then   the f.d.d. of $U^{n}$    converges $\cf^{x}$-stably to that of $\ce^{z'}$   as $n\to \infty$, where $\ce^{z'}(s,t)=c_{H}^{1/2}T^{H+1/2}\int_{s}^{t}z_{u}'dW_{u}$ and $W$ is a standard Brownian motion indpendent of $x$ (see \eqref{e.cezl}). 
\end{thm}
 \begin{proof}
 Let $(s',t')\in \cs_{2}(\ll0,T\rr)$ be such that $\{t_{k}:s'\leq t_{k}<t'\} = \{t_{k}:s \leq t_{k}<t \} $. 
It is easy to see that we have 
\begin{eqnarray}
U^{n}_{st} &=& n^{H+1/2} \big( \sum_{s\leq t_{k}<t} \int_{t_{k}}^{t_{k+1}} \delta z_{t_{k}u} du + \int_{s}^{s'}  z_{u}du-\int_{t}^{t'} z_{u}du\big). 
\label{e.unst}
\end{eqnarray}
 Since $\bfz$ is controlled by $(x,\ell+1,H)$ (see Definition \ref{def.control})  we have  the relation $\delta z_{t_{k}u} = \sum_{i=1}^{\ell}z_{t_{k}}^{(i)} x^{i}_{t_{k}u}+r^{(0)}_{t_{k}u}$. Substituting this relation into \eqref{e.unst}   we obtain 
 \begin{eqnarray}
U^{n}_{st} &=& n^{H+1/2} \ce^{z',n}_{\ell}(s,t) 
+R^{U}, 
\label{e.ust2}
\end{eqnarray}
where 
\begin{eqnarray}
R^{U}&=&n^{H+1/2}  \big(\sum_{s\leq t_{k}<t} \int_{t_{k}}^{t_{k+1}} r^{(0)}_{t_{k}u} du 
+ \int_{s}^{s'}  z_{u}du-\int_{t}^{t'} z_{u}du\big)
\end{eqnarray}
and recall that $\ce^{z',n}_{\ell}(s,t) $ is defined in  \eqref{e.ce}.

By the definition of controlled processes we can find a finite random variable $G$ and a sufficiently small $\ep>0$ such that $ |r^{(0)}_{t_{k}u}|\leq G (1/n)^{(\ell+1)H-\ep} $ for all $(t_{k},u)\in\cs_{2}(t_{k},t_{k+1})$, $k=0,\dots,n-1$. Applying this relation to $R^{U}$ and taking into account that $\ell H>1/2$ we obtain that $R^{U}\to0$ almost surely as $n\to\infty$. On the other hand, applying Theorem \ref{thm.taylor} (i) to $\ce^{z',n}_{\ell}(s,t) $ we obtain that the f.d.d. of $\ce^{z',n}_{\ell}(s,t) $ converges $\cf^{x}$-stably to  that  of  $\ce^{z'}$. Applying these two   convergences   to  \eqref{e.ust2}   we obtain the desired   convergence of  $U^{n}$. This completes the proof.    
\end{proof}

\section{Euler method for additive SDEs}\label{section.euler}
Let $y$ be   solution of    additive differential equation: 
\begin{eqnarray}\label{e.sde}
y_{t} = y_{0} +\int_{0}^{t} b(y_{s})ds+\int_{0}^{t}\si(s) dx_{s} 
\qquad t\in[0,T] , 
\end{eqnarray}
where  $y_{0} $ is a constant vector  in $\RR^{m}$, $x$ is a  $1$-dimensional standard fBm with Hurst parameter $H<1/2$, $\si  $  and  $b$ are  deterministic mappings  from $t\in [0,T]$ to $ \si(t)=(\si_{1}(t),\dots, \si_{m}(t))\in\RR^{m} $ and  from $y\in\RR^{m}$ to  $b(y)= (b_{1}(y),\dots, b_{m}(y)) \in\RR^{m}$, respectively.  
Let $\cp:0=t_{0}<t_{1}<\cdots<t_{n}=T$ be the uniform  partition on $[0,T]$. 
We consider the following  Euler method for equation \eqref{e.sde}:  
\begin{eqnarray}
y_{t}^{(n)}= y_{t_{k}}^{(n)} + b(y^{(n)}_{t_{k}})(t-t_{k})+\si(t_{k}) (x_{t}-x_{t_{k}}),
\qquad
t\in[t_{k},t_{k+1}]
\label{e.euler1}
\end{eqnarray}
     for $k=0,1,\dots, n-1$.
Denote $\eta(t)=t_{k}$ for $t\in[t_{k},t_{k+1})$. Then \eqref{e.euler1} can also be   written  as: 
\begin{eqnarray}\label{e.euler}
y_{t}^{(n)} = y_{0} +\int_{0}^{t} b(y^{(n)}_{\eta(s)})ds+ \int_{0}^{t}  \si(\eta(s)) dx_{s}\,, \qquad t\in[0,T] .   
\end{eqnarray}
In this  section we establish the exact rate of convergence and asymptotic error distribution for the Euler method 
\eqref{e.euler}.

\subsection{Rate of convergence}
In this subsection we show that under  proper regularity conditions the Euler method converges   with the   rate $ {H+1/2}$. 

\begin{thm}\label{thm.conv}
Let $y$  be the solution of the equation \eqref{e.sde}   and $y^{n}$ be the Euler method defined in \eqref{e.euler1}. Let $\ell\in\NN$ be such that $\ell H>1/2$ and  assume that $b\in C_{b}^{\ell+1}$ and $\si \in C^{2}_{b}$.  
Then 

\noindent(i) There exists a constant $K>0$ such that  the following  estimate holds: 
\begin{eqnarray}
\big|
(y_{t}-y^{(n)}_{t})-(y_{s}-y^{(n)}_{s})
\big|_{L_{2}}\leq K (1/n)^{H+1/2 }  (t-s)^{1/2}
\label{e.rate}
\end{eqnarray}
for all $(s,t)\in\cs_{2}([0,T])$. 

\noindent(ii) The process $ n^{H+1/2}(y_{t}-y^{(n)}_{t})$, $t\in[0,T]$ is tight on $C([0,T])$.  
\end{thm} 

\begin{remark}
A natural question is whether    the same     rate of the (backward) Euler method holds for multidimensional   fBm. 
The proof of Theorem \ref{thm.conv} suggests  that this is indeed the case.  
We leave the   discussion to a future research.  
\end{remark}
\begin{proof}[Proof of Theorem \ref{thm.conv}]The proof is divided into several steps. 

\noindent\emph{Step 1: Linear equation  for   $y  - y^{(n)} $.} In this step we derive  a linear equation for the error process $y  - y^{(n)} $. 
By taking   difference between the   equations   \eqref{e.sde} and \eqref{e.euler} we have:     
 \begin{eqnarray}\label{e.error}
y_{t} - y^{(n)}_{t} &=& \int_{0}^{t} [b(y_{s})- b(y^{(n)}_{\eta(s)}) ]ds
+ \int_{0}^{t}  [\si(s)-\si(\eta(s))] dx_{s}
\notag
\\
 &=& \int_{0}^{t} [b(y_{s})- b(y^{(n)}_{s}) ]ds+ \int_{0}^{t} [b(y^{(n)}_{s})- b(y^{(n)}_{\eta(s)}) ]ds
 + \int_{0}^{t}  [\si(s)-\si(\eta(s))] dx_{s}
\end{eqnarray}
for $  t\in[0,T] $. Denote     
\begin{eqnarray*}
b_{1}(s) = \int_{0}^{1} \partial b( y^{(n)}_{s} +\theta( y_{s} -y^{(n)}_{s}))d\theta, 
\qquad
s\in[0,T], 
\end{eqnarray*}
where recall that $\partial b$ is defined in \eqref{e.pb}. 
Then similar to equation \eqref{e.czn} we have the relation:  
\begin{eqnarray}\label{e.b1}
b(y_{s})- b(y^{(n)}_{s})    
&=&{b_{1}}(s)[ y_{s} -  y^{(n)}_{s} ],  
\end{eqnarray}
where the right-hand side stands for the matrix multiplication:
\begin{eqnarray*}
{b_{1}}(s)[ y_{s} -  y^{(n)}_{s} ] = 
\sum_{j=1}^{m}\int_{0}^{1}  {\partial_{j} b } ( y^{(n)}_{s} +\theta( y_{s} -y^{(n)}_{s}))d\theta \cdot [y^{j}_{s} -  y^{(n),j}_{s}],   
\end{eqnarray*}
  and  $y_{s} = (y_{s}^{1},\dots,y_{s}^{m})$ 
and  $y^{(n)}_{s} = (y_{s}^{(n),1},\dots,y_{s}^{(n),m})$.  
Substituting \eqref{e.b1} into \eqref{e.error} gives: 
\begin{eqnarray}\label{e.error2}
 y_{t} - y^{(n)}_{t}  &=& \int_{0}^{t}  {b_{1}}(s)[ y_{s} -  y^{(n)}_{s} ]ds+ \int_{0}^{t} [b(y^{(n)}_{s})- b(y^{(n)}_{\eta(s)}) ]ds
\notag
 \\
&&\qquad + \int_{0}^{t}  [\si(s)-\si(\eta(s))] dx_{s}, \qquad t\in[0,T] . 
\end{eqnarray}

\noindent\emph{Step 2: Fundamental solution of \eqref{e.error2}.}
Let $\La^{(n)}$ be the 
fundamental solution of the linear equation \eqref{e.error2}, namely,   
\begin{eqnarray}\label{e.lan}
\La_{t}^{(n)} = \id + \int_{0}^{t} b_{1} (s) \La_{s}^{(n)}ds\,, 
\qquad t\in[0,T] . 
\end{eqnarray}
Let $\Ga^{(n)}_{t}=(\La^{(n)}_{t})^{-1}$  be the inverse matrix of $\La^{(n)}_{t}$. It is well-known that $\Ga^{(n)} $ satisfies  the    equation:
\begin{eqnarray}\label{e.ga}
\Ga_{t}^{(n)} = \id - \int_{0}^{t}  \Ga_{s}^{(n)}  b_{1} (s) ds, 
\qquad t\in[0,T] . 
\end{eqnarray}
Applying Gronwall's inequality to \eqref{e.lan}-\eqref{e.ga} we obtain the following   estimates: 
\begin{eqnarray}
 |\La^{(n)}_{t} -\La^{(n)}_{s} | \leq Ke^{KT}(t-s)  
\qquad\text{and}\qquad
 |\Ga^{(n)}_{t} -\Ga^{(n)}_{s} | \leq Ke^{KT}(t-s) 
 \label{e.gabd} 
\end{eqnarray}
for all $(s,t)\in\cs_{2}([0,T])$. 

\noindent\emph{Step 3: Explicit representation of   $y-y^{(n)}$.}
By   definition of $\La^{(n)}$ in \eqref{e.lan} it is easy to verify    that we have 
\begin{eqnarray}
y_{t}-y^{(n)}_{t} &=& \La^{(n)}_{t} \int_{0}^{t} \Ga^{(n)}_{u} [b(y^{(n)}_{u})- b(y^{(n)}_{\eta(u)}) ]du
\notag
\\
&&\qquad+\La^{(n)}_{t}  \int_{0}^{t}  \Ga^{(n)}_{u}[\si(u)-\si(\eta(u))] dx_{u}
 .
 \label{e.error1}
\end{eqnarray}
From \eqref{e.error1}   we can write  
  \begin{eqnarray}
(y_{t}-y^{(n)}_{t})-(y_{s}-y^{(n)}_{s}) &=& (\La^{(n)}_{t}-\La^{(n)}_{s}) E (0,t) 
 +
\La^{(n)}_{s} E (s,t)   
\label{e.error4}
\end{eqnarray}
for $(s,t)\in\cs_{2}([0,T])$, 
where   
\begin{eqnarray}\label{e.ist}
E (s,t) =E_{1} (s,t) +E_{2} (s,t) ,
\end{eqnarray}
and
\begin{eqnarray}
E_{1} (s,t) &=&  \int_{s}^{t} \Ga^{(n)}_{u} [b(y^{(n)}_{u})- b(y^{(n)}_{\eta(u)}) ]du
\\
E_{2} (s,t) &=&\int_{s}^{t}  \Ga^{(n)}_{u}[\si(u)-\si(\eta(u))] dx_{u}\,. 
\notag
\end{eqnarray}

In the following steps 
we show that  there is constant $K$ such that  
\begin{eqnarray}
|E(s,t)|_{L_{p}}&\leq& K(t-s)^{1/2}(1/n)^{H+1/2}  
\label{e.ebd2}
\end{eqnarray}
for all $(s,t)\in\cs_{2}([0,T])$ and $n\in\NN$. 
The estimate \eqref{e.rate} will then follow by applying \eqref{e.ebd2} and  the estimate of $\La^{(n)}_{t}-\La^{(n)}_{s}$ in \eqref{e.gabd} to the right-hand side of \eqref{e.error4}.

\noindent\emph{Step 4: Estimate of    $E(s,t)$ when $|t-s|\leq T/n$.}
By the mean value theorem and then   the definition of $y^{n}$ in \eqref{e.euler1} it is clear that we have 
\begin{eqnarray}
|b(y^{(n)}_{u})- b(y^{(n)}_{\eta(u)}) |_{L_{p}}\lesssim (1/n)^{H}
\label{e.ebdy}
\end{eqnarray}
for $p\geq 1$. 
 Applying this estimate to   \eqref{e.ist} we get:  
\begin{eqnarray}
|E_{1}(s,t)|_{L_{p}}&\lesssim & (t-s)(1/n)^{H}. 
\label{e.ebd1}
\end{eqnarray}
Suppose that   $s$ and $t$ are such that  $|t-s|\leq T/n$, then we can bound the right-hand side of  \eqref{e.ebd1} by $(t-s)^{1/2}(1/n)^{H+1/2}$. By Young's inequality  it is easy to see that the same bound holds for $E_{2}$.  It follows that    the relation \eqref{e.ebd2} holds. 

In the following   we focus on the case when  $|t-s|>T/n$. 

\noindent\emph{Step 4: A decomposition of    $E_{1}(s,t)$.}
We first   consider the following decomposition of   \eqref{e.ist}: 
 \begin{eqnarray}\label{e.error3}
 E_{1}(s,t) &=&    J_{0}+\int_{s}^{t} \Ga^{(n)}_{\eta(u)} [b(y^{(n)}_{u})- b(y^{(n)}_{\eta(u)}) ]du, 
\end{eqnarray}
where
 \begin{eqnarray}\label{e.j0}
J_{0} =  \int_{s}^{t} (\Ga^{(n)}_{u}-\Ga^{(n)}_{\eta(u)}) [b(y^{(n)}_{u})- b(y^{(n)}_{\eta(u)}) ]du. 
\end{eqnarray}

Recall that we have $b\in C_{b}^{\ell+1}$. 
Applying the classical  Taylor's expansion for multivariate functions we obtain:   
\begin{eqnarray}\label{e.bexp}
b(y^{(n)}_{u})- b(y^{(n)}_{\eta(u)}) 
&=&\sum_{i=1}^{\ell}   \partial^{i}b(y^{(n)}_{\eta(u)})  (\delta y^{(n)}_{\eta(u) u})^{\otimes i} /i! 
+ r^{b} (  y^{(n)}_{u}) (\delta y^{(n)}_{\eta(u) u})^{\otimes (\ell+1)} 
\,,
\end{eqnarray}
where    
\begin{eqnarray}
\partial^{i}b(y^{(n)}_{\eta(u)})  (\delta y^{(n)}_{\eta(u) u})^{\otimes i} := \sum_{j_{1},\dots, j_{i}=1}^{m}\partial^{i}_{j_{1}\dots j_{m}}b(y^{(n)}_{\eta(u)})  \delta y^{(n),j_{1}}_{\eta(u) u}\cdots \delta y^{(n),j_{i}}_{\eta(u) u}\,, 
\notag
\end{eqnarray}
$\partial^{i}b$ is defined in \eqref{e.pbl}, 
and  $r^{b}(x)$, $x\in\RR^{m}$ is a function  such that 
\begin{eqnarray}
\sup_{x\in\RR^{m}}|r^{b}(x)| &\leq&   \max_{y\in\RR^{m}}|\partial^{\ell+1}b(y)|  /(\ell+1)!  \,. 
\label{e.rb}
\end{eqnarray} 
Consider the decomposition:    \begin{eqnarray*}
 (\delta y^{(n)}_{\eta(u) u})^{\otimes i}  
=\big( (\delta y^{(n)}_{\eta(u) u})^{\otimes i}   
-  (\si (\eta(u))  x_{\eta(u) u}^{1})^{\otimes i} \big)+ (\si (\eta(u))  x_{\eta(u) u}^{1})^{\otimes i} . 
\end{eqnarray*}
 Substituting this   into   \eqref{e.bexp} we get   
\begin{eqnarray}\label{e.bexp2}
b(y^{(n)}_{u})- b(y^{(n)}_{\eta(u)}) 
   =
    \sum_{i=1}^{\ell}   \partial^{i}b(y^{(n)}_{\eta(u)})  (\si (\eta(u))x^{1}_{\eta(u) u})^{\otimes i} /i!   +
 \tilde{r}^{b}_{\eta(u)u}
 \notag\\
 +r^{b} (  y^{(n)}_{u}) (\delta y^{(n)}_{\eta(u) u})^{\otimes (\ell+1)} \,, 
\end{eqnarray}
where
\begin{eqnarray}
 \tilde{r}^{b}_{\eta(u)u}= \sum_{i=1}^{\ell}   \partial^{i}b(y^{(n)}_{\eta(u)}) \big( (\delta y^{(n)}_{\eta(u) u})^{\otimes i}   
-  (\si (\eta(u))  x_{\eta(u) u}^{1})^{\otimes i} \big)/i! . 
\label{e.rbt}
\end{eqnarray}

Substituting \eqref{e.bexp2} into \eqref{e.error3} we obtain:  
\begin{eqnarray}
 E_{1}(s,t)  = J_{0}+ J_{1}+ R^{b}+\wt{R}^{b},
 \label{e.est}
\end{eqnarray}
where   
\begin{eqnarray}
&&\qquad J_{1}=    \sum_{i=1}^{\ell}  \int_{s}^{t} \Ga^{(n)}_{\eta(u)}  
 \partial^{i}b(y^{(n)}_{\eta(u)})   {(\si (\eta(u))x^{1}_{\eta(u) u})^{\otimes i}}/{i!} \, \,
 du  ,
\quad
\wt{R}^{b}=\int_{s}^{t} \Ga^{(n)}_{\eta(u)} \tilde{r}^{b}_{\eta(u)u}
du, 
\label{e.ji}
\\
&&\qquad R^{b}=\int_{s}^{t} \Ga^{(n)}_{\eta(u)} r^{b} (  y^{(n)}_{u}) (\delta y^{(n)}_{\eta(u) u})^{\otimes (L+1)}
du   .  
\label{e.j6} 
\end{eqnarray}

 \noindent\emph{Step 7: Estimate of $J_{0} $.}
Recall that $J_{0}$ is defined in \eqref{e.j0}. 
Applying   the estimates \eqref{e.gabd}   and \eqref{e.ebdy} to \eqref{e.j0}  we obtain that there is a constant $K$ such that  
\begin{eqnarray}
|J_{0}|_{L_{p}}&\leq& Kn^{-1-H}(t-s)  
\label{e.j0c}
\end{eqnarray}
for all $n\in\NN$,  $
(s,t)\in\cs_{2}([0,T]).$

 \noindent\emph{Step 8: Estimate of ${R}^{b} $.}  By the definition of $y^{(n)}$ it is clear that  $| \delta y^{(n)}_{\eta(u), u} |_{L_{p}}\lesssim (1/n)^{H}$. This   together with the estimates \eqref{e.gabd} and  \eqref{e.rb}  gives: 
\begin{eqnarray}
\big| \Ga^{(n)}_{\eta(u)} r^{b} (  y^{(n)}_{u}) (\delta y^{(n)}_{\eta(u) u})^{\otimes (\ell+1)} \big|_{L_{p}}&\lesssim&  (1/n)^{(\ell+1)H}. 
\label{e.rbdi1}
\end{eqnarray}
Substituting \eqref{e.rbdi1} into \eqref{e.j6} we obtain: 
\begin{eqnarray}
|R^{b}|_{L_{p}}&\leq&K (1/n)^{(\ell+1)H}(t-s) 
\label{e.rbdi}
\end{eqnarray}
for all $n\in\NN$,  $
(s,t)\in\cs_{2}([0,T]).$

 \noindent\emph{Step 8: Estimate of $\wt{R}^{b}$.}
 By the definition of $y^{(n)}$ it is clear that $|\delta y^{(n)}_{\eta(u) u}- \si(\eta(u))\delta x_{\eta(u)u} |_{L_{p}}\lesssim (1/n)$. It follows that for $i\geq1$ we have
 \begin{eqnarray}
|(\delta y^{(n)}_{\eta(u) u})^{\otimes i}   
-  (\si  (\eta(u)) x_{\eta(u) u}^{1})^{\otimes i} |_{L_{p}}&\lesssim&  1/n . 
\notag
\end{eqnarray}
Appying this relation  to 
  \eqref{e.rbt} we obtain the estimate: 
 \begin{eqnarray*}
|\tilde{r}^{b}_{\eta(u), u}|_{L_{p}}&\lesssim&  1/n. 
\end{eqnarray*}
This implies that $ \wt{R}^{b}$ is bounded by: 
\begin{eqnarray}
|\wt{R}^{b}|_{L_{p}}&\leq&K (1/n) (t-s) 
\qquad
\text{for all $n\in\NN$,  $
(s,t)\in\cs_{2}([0,T]).$}
\label{e.rtbd}
\end{eqnarray}


  \noindent\emph{Step 9: Estimate of $ J_{1}$.}
We define the processes:
\begin{eqnarray}
z_{t}^{(i-1)} &:=&   \Ga^{(n)}_{t}\partial^{i} b(y^{(n)}_{t})\si(t)^{\otimes i}
\notag
\\
  &:=&  
\sum_{j_{1},\cdots, j_{i}=1}^{m}\Ga^{(n)}_{t}\partial^{i}_{j_{1}\cdots j_{i}} b(y^{(n)}_{t})\si_{j_{1}}(t)\cdots \si_{j_{i}}(t) 
\label{e.zi}
\end{eqnarray}
for $i=1, \dots, \ell$.
It is easy to verify that    $(z,z',\dots, z^{(\ell-1)})$ satisfies conditions in Definition \ref{def.control} and      therefore is a process controlled by $(x,\ell,H)$ in $L_{p}$ for $p\geq 1$. 
 Note that 
\begin{eqnarray}
  {(\si(\eta(u)) x^{1}_{\eta(u) u})^{\otimes i}}/{i!}  =   {\si(\eta(u))^{\otimes i} (x^{1}_{\eta(u) u})^{i}}/{i!}  =   {\si(\eta(u))^{\otimes i}}x^{i}_{\eta(u) u} . 
\label{e.zi3}
\end{eqnarray}
Substituting \eqref{e.zi3} into \eqref{e.ji} and taking into account relation \eqref{e.zi} we have:  
\begin{eqnarray}
J_{1} &=& \sum_{i=1}^{\ell} \int_{s}^{t}  z_{\eta(u)}^{(i-1)}   x_{\eta(u)u}^{i}du\,.
\notag
\end{eqnarray}


Let $s'$, $t'\in[0,T]$ be such that   
\begin{eqnarray}
\{t_{k}: s'\leq t_{k}<t'\} = \{t_{k}: s\leq t_{k}<t\} . 
\label{e.stp}
\end{eqnarray}
 Then it is easy to see that we have:   
\begin{eqnarray}
J_{1} = 
\ce^{z, n}_{\ell}(s,t) + J_{11}\, , 
\label{e.j1j2i} 
\end{eqnarray}
where   $\ce^{z,n}_{\ell}(s,t)$ is defined in \eqref{e.ce} and 
\begin{eqnarray}
J_{11} &=& \sum_{i=1}^{\ell}
\Big(\int_{s}^{s'}   -
\int_{t}^{t'}\Big)  z_{\eta(u)}^{(i-1)}   x_{\eta(u)u}^{i} du . 
\notag
\end{eqnarray}

Since $ |z_{\eta(u)}^{(i-1)}   x_{\eta(u)u}^{i} |_{L_{p}}\lesssim (1/n)^{H}$ for $i\geq1$ we have the estimate:
\begin{eqnarray}
|J_{11}|_{L_{p}}\leq K n^{-H}(|s'-s|+|t-t'|) \leq  K(1/n)^{H}(t-s) 
 \label{e.j12bd}
\end{eqnarray}
for all $n\in\NN$,  $
(s,t)\in\cs_{2}([0,T]).$ 
On the other hand, according to  Theorem \ref{thm.taylor} we   have: 
 \begin{eqnarray}
|\ce^{z, n}_{\ell}(s,t) |_{L_{p}} \leq  
 K(1/n)^{H+1/2}(t-s)^{1/2} . 
\label{e.j11bd}
\end{eqnarray}

\noindent\emph{Step 10: Estimate of $E_{2}(s,t)$.}
Consider the following decomposition of $E_{2}(s,t)$: 
\begin{eqnarray}
E_{2} (s,t) &=& E_{21} (s,t) 
+ R^{\si}
\notag
\end{eqnarray}
where
\begin{eqnarray}
E_{21} (s,t) &=& \int_{s'}^{t'} \Ga^{(n)}_{\eta(u)}\si'(\eta(u)) (u-\eta(u))dx_{u} 
\label{e.e21}
\\
R^{\si} &=&\int_{s}^{t}  \Ga^{(n)}_{u}[\si(u)-\si(\eta(u))] dx_{u} - E_{21} (s,t) , 
\notag
\end{eqnarray}
and $s' $ and $t'$ are   as in \eqref{e.stp}.    Let us rewrite \eqref{e.e21} as: 
\begin{eqnarray}
E_{21}(s,t) &=& \sum_{s \leq t_{k}<t } \Ga^{(n)}_{t_{k}} \si'(t_{k}) \int_{t_{k}}^{t_{k+1}} (u-t_{k})dx_{u}. 
\notag
\end{eqnarray}
 Applying integration by parts for Young integrals we have: 
\begin{eqnarray}
E_{21}(s,t) &=& E_{211}(s,t) + E_{212}(s,t),
\notag
\end{eqnarray}
where
\begin{eqnarray}
E_{211}(s,t)  &=& -\sum_{s \leq t_{k}<t } \Ga^{(n)}_{t_{k}} \si'(t_{k}) \int_{t_{k}}^{t_{k+1}} x^{1}_{t_{k}u}du 
=- \sum_{s \leq t_{k}<t} \Ga^{(n)}_{t_{k}} \si'(t_{k})  h^{1}_{t_{k}t_{k+1}}
\label{e.e211}
\\
E_{212}(s,t) &=& \sum_{s \leq t_{k}<t } \Ga^{(n)}_{t_{k}} \si'(t_{k})   x^{1}_{t_{k}t_{k+1}} (t_{k+1}-t_{k}).   
\notag
\end{eqnarray}
   Applying Young's inequality  to $E_{212}$ we have the estimate  
   \begin{eqnarray}
|E_{212}(s,t)|_{L_{p}} &\leq& K (1/n)(t-s)^{H}. 
\label{e.e212bd}
\end{eqnarray}
Similarly, applying Young's inequality again to $E_{211}$ and invoking the estimate \eqref{e.cexbdl2} with $L=1$ we obtain:
\begin{eqnarray}
|E_{211}(s,t)|_{L_{p}} &\leq& K (1/n)^{H+1/2}(t-s)^{1/2}. 
\label{e.e211bd}
\end{eqnarray}

It can be shown that $R^{\si}$ is equal to the summation of several triple Young integrals and are bounded by 
\begin{eqnarray}
|R^{\si}|_{L_{p}}&\leq & K (1/n)(t-s)^{H}. 
\label{e.rsbd}
\end{eqnarray}
  For conciseness we leave the details of   proof for \eqref{e.rsbd} to the patient reader.



 \noindent\emph{Step 11: Conclusion.}
 According to Lemma \ref{lem.stnbd}  we can replace  the right-hand sides of the   estimates \eqref{e.j0c},  \eqref{e.rbdi}, \eqref{e.rtbd}, \eqref{e.j12bd} by $K(1/n)^{H+1/2}(t-s)^{1/2}$.  Applying this estimate of $J_{0}$, $ J_{11}$,  $R^{b}$, $\wt{R}^{b}$ together with the estimate \eqref{e.j11bd}  of $\ce^{z, n}_{\ell}(s,t) $   to \eqref{e.j1j2i} and   \eqref{e.est} we obtain that the upper-bound in  
\eqref{e.ebd2} holds for $E_{1}(s,t)$. 
Similarly, in summary of the estimates  \eqref{e.e212bd}-\eqref{e.rsbd} for   $E_{211}$, $E_{212}$, $R^{\si}$  we obtain that \eqref{e.ebd2} holds for $E_{2}(s,t)$. Applying these   estimates of $E_{1}$ and $E_{2}$  to \eqref{e.ist} shows that \eqref{e.ebd2} holds.  
  Recalling the argument in the end of Step~3, we thus conclude the relation \eqref{e.rate}. 
According to  \cite[Theorem 12.3 and  (12.51)]{billingsley},   relation \eqref{e.rate} implies that  the normalized error process  $ n^{H+1/2}(y -y^{(n)} )$   is tight $C([0,T])$. 
\end{proof}

In the following we state a corollary of Theorem \ref{thm.conv}. 
\begin{cor}\label{cor.error}
Let the assumptions be as in Theorem \ref{thm.conv}. Let $\La^{(n)}$, $\bfz$   and $E_{211}$ be defined in \eqref{e.lan}, \eqref{e.zi} and \eqref{e.e211}, respectively. Then we have the convergence:
\begin{eqnarray}
\qquad\lim_{n\to\infty}\sup_{t\in[0,T]}n^{H+1/2} \big |(y_{t}-y^{(n)}_{t} ) -   \La^{(n)}_{t} \ce^{z, n}_{\ell}(0,t) - 
\La^{(n)}_{t}  E_{211}(0,t)  \big|_{L_{p}} =0. 
\label{e.error6}
\end{eqnarray}

\end{cor}
\begin{proof}
According to     \eqref{e.error4}, \eqref{e.est} and \eqref{e.j1j2i} we have the relation: 
\begin{eqnarray}
y_{t}-y^{(n)}_{t}  = \La^{(n)}_{t} E(0,t)
&=&
 \La^{(n)}_{t}( J_{0}+ \ce^{z, n}_{\ell}(0,t)+J_{11}+ R^{b}+\wt{R}^{b} 
 \notag
\\
&&
\qquad\qquad+ E_{212}(0,t)+ E_{211}(0,t)+R^{\si}) 
\label{e.error7}
\end{eqnarray}
for $
t\in[0,T].$
Applying the estimates of $J_{0}$, $R^{b}$, $\wt{R}^{b}$, $J_{11}$, $E_{212}$, $R^{\si}$ in \eqref{e.j0c}, \eqref{e.rbdi}, \eqref{e.rtbd}, \eqref{e.j12bd}, \eqref{e.e212bd}, \eqref{e.rsbd} to \eqref{e.error7} we obtain     the convergence \eqref{e.error6}. 
\end{proof}


\subsection{Asymptotic error distribution}

With the previous preparations  let us now     consider  asymptotic error distribution of the Euler method. 
Let $\La $ be solution  of the   equation:  
\begin{eqnarray}\label{e.la}
\La_{t} = \id+ \int_{0}^{t} \partial b(y_{s})\La_{s}ds
\,, 
\qquad t\in[0,T].
\end{eqnarray}
Let  $\Ga_{t} =\La_{t}^{-1}$ be the inverse matrix of   $\La_{t}$. As  in \eqref{e.ga}-\eqref{e.gabd} we have the  equation:  
\begin{eqnarray}
\Ga_{t}  = \id - \int_{0}^{t}  \Ga_{s} \partial b (y_{s}) ds . 
\label{e.ga1}
\end{eqnarray}
The Gronwall's ienquality implies that we have the estimate: 
\begin{eqnarray}
|\La_{t} -\La_{s} |   \leq Ke^{KT}(t-s), 
\qquad
 |\Ga_{t} -\Ga_{s} |   \leq Ke^{KT}(t-s) , 
 \qquad
 (s,t)\in\cs_{2}([0,T]). 
\notag
\end{eqnarray}

The following result shows that $(\La,\Ga)$ is the limiting process of $(\La^{n},\Ga^{n})$. 

\begin{lemma}
Let $\La$,  $\La^{(n)}$, $\Ga$ and $\Ga^{(n)}$ be defined in \eqref{e.lan}, \eqref{e.ga}, \eqref{e.la}, and \eqref{e.ga1}, respectively.       Then for $p\geq 1$   there is a constant $K>0$ such that   the following estimate holds: 
\begin{eqnarray}
\sup_{t\in[0,T]}\big(|\La_{t}-\La^{(n)}_{t} |_{L_{p}}+|\Ga_{t}-\Ga^{(n)}_{t} |_{L_{p}}\big) \leq K  (1/n)^{H+1/2} . 
\label{e.laconv} 
\end{eqnarray}
\end{lemma}
\begin{proof}
Take the difference between \eqref{e.la} and \eqref{e.lan}.  We obtain: 
\begin{eqnarray}
\La_{t}-\La^{(n)}_{t}= \int_{0}^{t} \partial b(y_{s})(\La_{s}-\La_{s}^{(n)})ds+\int_{0}^{t} (\partial b(y_{s}) -b_{1}(s))\La_{s}^{(n)} ds.
\label{e.ladiff}
\end{eqnarray}
Applying   $L_{p}$-norm in both sides of \eqref{e.ladiff} gives: 
\begin{eqnarray}
|\La_{t}-\La^{(n)}_{t}|_{L_{p}}
&\lesssim& \int_{0}^{t}  |\La_{s}-\La_{s}^{(n)}|_{L_{p}}ds+\int_{0}^{t} |\partial b(y_{s}) -b_{1}(s)|_{L_{2p}}  ds. 
\label{e.lad}
\end{eqnarray}
Note that by applying the mean value theorem and then the estimate \eqref{e.rate} we have:
\begin{eqnarray}
\int_{0}^{t} |\partial b(y_{s}) -b_{1}(s)|_{L_{2p}}  ds \lesssim
\int_{0}^{t} |   y_{s}  -y_{s}^{(n)}|_{L_{2p}}  ds 
 \lesssim  (1/n)^{H+1/2}. 
\label{e.lac1}
\end{eqnarray}
Substituting \eqref{e.lac1} into \eqref{e.lad} we   obtain:
\begin{eqnarray}
|\La_{t}-\La^{(n)}_{t}|_{L_{p}}
&\lesssim& \int_{0}^{t}  |\La_{s}-\La_{s}^{(n)}|_{L_{p}}ds+ (1/n)^{H+1/2}. 
\label{e.lad2}
\end{eqnarray}
It then follows from   Gronwall's inequality   that relation \eqref{e.laconv} holds for $|\La_{t}-\La^{(n)}_{t} |_{L_{p}}$. 
The estimate for $|\Ga_{t}-\Ga^{(n)}_{t} |$ can be shown in the similar way  and is   omitted  for   conciseness.    
\end{proof}

In the following we establish the asymptotic error distribution of the Euler method.  
  \begin{theorem}\label{thm.euler}
Let the assumptions be as in Theorem \ref{thm.conv}. Denote $U^{(n)}_{t}= n^{H+1/2}(y_{t}^{(n)}-y_{t}) $, $t\in[0,T]$.  Then we have the  following  convergence in distribution on $C([0,T])$: 
\begin{eqnarray}
(U^{(n)}_{t},x_{t},  t\in[0,T])  \to (U_{t},x_{t},  t\in [0,T]), 
\label{e.errorc}
\end{eqnarray}
where $U$ is the solution of the linear equation: 
\begin{eqnarray}
U_{t} = \int_{0}^{t}\partial b(y_{u})  U_{u}  du +c_{H}^{1/2}T^{H+1/2}\int_{0}^{t} \big[\partial b (y_{u}) \si (u)-\si'(u)\big]dW_{u}
\label{e.udef1}
\end{eqnarray}
and $W$ is a Brownian motion independent of $x$ and $c_{H}$ is the constant defined in \eqref{e.rho}. 
\end{theorem}
\begin{proof}
Recall that we have  shown the tightness of    $U^{(n)}$ in Theorem \ref{thm.conv} (ii). In order to prove   \eqref{e.errorc} it remains to show the convergence of the f.d.d.  of $(U^{(n)}, x)$. According to Corollary~\ref{cor.error} this can be   reduced to proving the convergence: 
\begin{eqnarray}
\qquad
\big( n^{H+1/2}\La^{(n)}_{t}\big( \ce^{z, n}_{\ell}(0,t)  +E_{211}(0,t)\big) ,\,\,   t\in[0,T]\big)  \xrightarrow{stable\,\,f.d.d.} U
\qquad
\text{as $n\to\infty$}. 
\label{e.uconv2}
\end{eqnarray}


Define the processes:   $
\bar{z}_{t}^{(i-1)} :=   \Ga_{t}\partial^{i} b(y_{t})\si(t)^{\otimes i}$,  
$
i=1, \dots, \ell$. 
As in \eqref{e.zi} it is easy to see that $(\bar{z},\bar{z}',\dots, \bar{z}^{(\ell-1)})$ is controlled by $(x,\ell,H)$ in $L_{p}$ for $p\geq 1$. 
As a first step for the proof of \eqref{e.uconv2}, let us      show that
\begin{eqnarray}
n^{H+1/2}\big(
\La_{t}\ce^{\bar{z}, n}_{\ell}(0,t)-\La^{(n)}_{t} \ce^{ {z}, n}_{\ell}(0,t) \big)
\to 0
\quad \text{in probability as $n\to\infty$.}
\label{e.lajc}
\end{eqnarray}
Note that by the definition  of $\ce^{\bar{z}, n}_{\ell}(0,t)$ and $ \ce^{ {z}, n}_{\ell}(0,t)$ we have
   \begin{eqnarray}
&&
\La_{t}\ce^{\bar{z}, n}_{\ell}(0,t)-\La^{(n)}_{t} \ce^{ {z}, n}_{\ell}(0,t)  
\notag\\
&&  \qquad=  \sum_{i=1}^{\ell}
\int_{0}^{t} \big( \La_{t}\Ga_{\eta(u)} \partial^{i}b(y_{\eta(u)}) -\La^{(n)}_{t}\Ga^{(n)}_{\eta(u)} \partial^{i}b(y^{(n)}_{\eta(u)})  \big)\si^{\otimes i}   x^{i}_{\eta(u)u} du.
\label{e.te2} 
\end{eqnarray}
Applying the mean value theorem and   the estimates \eqref{e.rate} and \eqref{e.laconv} we have: 
\begin{eqnarray*}
\big|\La^{(n)}_{t}\Ga^{(n)}_{\eta(u)} \partial^{i}b(y^{(n)}_{\eta(u)}) - \La_{t}\Ga_{\eta(u)} \partial^{i}b(y_{\eta(u)}) \big|_{L_{p}}&\lesssim& (1/n)^{H+1/2}. 
\end{eqnarray*}
Applying this to \eqref{e.te2} gives: 
\begin{eqnarray*}
| \La_{t}\ce^{\bar{z}, n}_{\ell}(0,t)-\La^{(n)}_{t} \ce^{ {z}, n}_{\ell}(0,t)     |_{L_{p}} &\lesssim& \sum_{i=1}^{\ell}(1/n)^{H+1/2} (1/n)^{iH}.
\end{eqnarray*}
It follows that relation \eqref{e.lajc} holds. 

 Similar to \eqref{e.lajc}, define  
\begin{eqnarray}
\bar{E}_{211}(s,t) &=&-\sum_{s \leq t_{k}<t } \Ga_{t_{k}} \si'(t_{k})  h^{1}_{t_{k}t_{k+1}}. 
\notag
\end{eqnarray}
Then we can show that 
\begin{eqnarray}
n^{H+1/2}\big| 
\La^{(n)}_{t} E_{211}(0,t) -\La_{t}  \bar{E}_{211}(0,t)  
\big| \to 0
\qquad
\text{ in probability as $n\to\infty$.}
\label{e.e211b}
\end{eqnarray}


Now, according to  Theorem \ref{thm.taylor} we have the f.d.d. stable convergence: 
\begin{eqnarray}
\big( n^{H+1/2}\La_{t}( \ce^{\bar{z}, n}_{\ell}(0,t) + \bar{E}_{211}(0,t))  ,\,\, x_{t}  ,\,\, t\in[0,T]\big)  \to ( U_{t}, \,\,x_{t} , \,\,t\in[0,T])
\label{e.uconv}
\end{eqnarray}
as $n\to\infty$, where
\begin{eqnarray}
U_{t}
=
c_{H}^{1/2}T^{H+1/2}\La_{t} 
\int_{0}^{t}\Ga_{u} \big[\partial b(y_{u}) \si (u) -\si' (u)\big]dW_{u}  
 \, .
\label{e.udef}
\end{eqnarray}
It is easy to verify   that the process $U$ defined in \eqref{e.udef} satisfies   equation \eqref{e.udef1}. 
Putting together relations   \eqref{e.lajc}, \eqref{e.e211b} and   \eqref{e.uconv}  we obtain      the convergence in  \eqref{e.uconv2}. 
 This completes the proof. 
\end{proof}

\bibliographystyle{abbrv}
\bibliography{additiveSDE.bib} 


\end{document}